\documentclass[11pt]{article}
\usepackage{authblk}
\usepackage[toc,page]{appendix}
\usepackage[top=2cm, bottom=2cm, left=2cm, right=2cm]{geometry}

\usepackage{color}
\usepackage{helvet}         
\usepackage{courier}        
\usepackage{type1cm}        

\usepackage{framed} 
\usepackage{tikz}
\usepackage{makeidx}         
\usepackage{graphicx}        
\usepackage{multicol}        
\usepackage[bottom]{footmisc}

\usepackage{amsmath}
\usepackage{amssymb}
\usepackage{bbold}
\usepackage{amsthm}
\usepackage{subcaption}
\usepackage{sidecap}
\usepackage{floatrow}
\usepackage{pdflscape}
\usepackage{comment}
\usepackage[font=small]{caption}
\usepackage{enumitem}
\usepackage{esint}

\allowdisplaybreaks

\usepackage{scalerel}
\usepackage{shuffle}
\usepackage{mathrsfs}

\newtheorem{theorem}{Theorem}
\newtheorem{corollary}[theorem]{Corollary}
\newtheorem{lemma}[theorem]{Lemma}

\newtheorem{proposition}[theorem]{Proposition}

\theoremstyle{definition}
\newtheorem{definition}[theorem]{Definition}

\theoremstyle{remark}

\newtheorem{remark}[theorem]{\bf Remark}
\newtheorem{example}[theorem]{\bf Example}

\numberwithin{theorem}{section}
\numberwithin{figure}{section}
\numberwithin{equation}{section}

\begin{document}

\title{Multiradial SLE with spiral: \\ resampling property and boundary perturbation}
\bigskip{}
\author[1]{Chongzhi Huang\thanks{huangchzh2001prob@gmail.com}}
\author[2]{Eveliina Peltola\thanks{eveliina.peltola@hcm.uni-bonn.de}}
\author[1]{Hao Wu\thanks{hao.wu.proba@gmail.com.}}
\affil[1]{Tsinghua University, China}
\affil[2]{Aalto University, Finland, and University of Bonn, Germany}

\date{}

%
%


\global\long\def\qnum#1{\left[#1\right]_q }
\global\long\def\qfact#1{\left[#1\right]_q! }
\global\long\def\qbin#1#2{\left[\begin{array}{c}
	#1\\
	#2 
	\end{array}\right]_q}

\global\long\def\defpatt{\shuffle}
\global\long\def\rainbow{\includegraphics[scale=0.15]{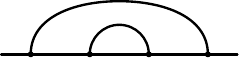}}
\global\long\def\rainbowBig{\includegraphics[scale=0.3]{figures/link-2}}

\newcommand{\nradpartfn}[2]{\mathcal{Z}^{#2}_{#1\mathrm{\textnormal{-}rad}}}
\global\long\def\covmap{h}

\global\long\def\mslitdriv{\omega}
\global\long\def\nnofloops{\mathscr{L}}

\global\long\def\Selberg{S}
\global\long\def\Diff{\Theta}
\global\long\def\SinDiff{\Xi}

\global\long\def\contour{\mathscr{C}}


\global\long\def\U{\mathbb{U}}
\global\long\def\T{\mathbb{T}}
\global\long\def\HH{\mathbb{H}}
\global\long\def\R{\mathbb{R}}
\global\long\def\C{\mathbb{C}}
\global\long\def\N{\mathbb{N}}
\global\long\def\Z{\mathbb{Z}}
\global\long\def\E{\mathbb{E}}
\global\long\def\PP{\mathbb{P}}
\global\long\def\QQ{\mathbb{Q}}
\global\long\def\A{\mathbb{A}}
\global\long\def\one{\mathbb{1}}

\newcommand{\PPspiral}[1]{\mathbb{P}^{\mu}_{#1\mathrm{\textnormal{-}rad}}}
\newcommand{\PPnospiral}[1]{\mathbb{P}^{0}_{#1\mathrm{\textnormal{-}rad}}}
\newcommand{\PPnospiralrho}{\mathbb{P}^{0; \bs{\rho}}}
\newcommand{\PPspiralrho}{\mathbb{P}^{\mu; \bs{\rho}}}

\global\long\def\CR{\mathrm{CR}}
\global\long\def\ST{\mathrm{ST}}
\global\long\def\SF{\mathrm{SF}}
\global\long\def\cov{\mathrm{cov}}
\global\long\def\dist{\mathrm{dist}}
\global\long\def\SLE{\mathrm{SLE}}
\global\long\def\hSLE{\mathrm{hSLE}}
\global\long\def\CLE{\mathrm{CLE}}
\global\long\def\GFF{\mathrm{GFF}}
\global\long\def\inte{\mathrm{int}}
\global\long\def\ext{\mathrm{ext}}
\global\long\def\inrad{\mathrm{inrad}}
\global\long\def\outrad{\mathrm{outrad}}
\global\long\def\dimH{\mathrm{dim}}
\global\long\def\capa{\mathrm{cap}}
\global\long\def\diam{\mathrm{diam}}
\global\long\def\sign{\mathrm{sgn}}
\global\long\def\cat{\mathrm{Cat}}
\global\long\def\cst{\mathrm{C}}
\global\long\def\ck{\mathrm{C}_{\kappa}}
\global\long\def\free{\mathrm{free}}
\global\long\def\hF{{}_2\mathrm{F}_1}
\global\long\def\simple{\mathrm{simple}}
\global\long\def\even{\mathrm{even}}
\global\long\def\odd{\mathrm{odd}}
\global\long\def\st{\mathrm{ST}}
\global\long\def\usf{\mathrm{USF}}
\global\long\def\Leb{\mathrm{Leb}}
\global\long\def\LP{\mathrm{LP}}
\global\long\def\I{\mathrm{I}}
\global\long\def\II{\mathrm{II}}
\global\long\def\hcap{\mathrm{hcap}}

\global\long\def\LA{\mathcal{A}}
\global\long\def\LB{\mathcal{B}}
\global\long\def\LC{\mathcal{C}}
\global\long\def\LD{\mathcal{D}}
\global\long\def\LF{\mathcal{F}}
\global\long\def\LK{\mathcal{K}}
\global\long\def\LE{\mathcal{E}}
\global\long\def\LG{\mathcal{G}}
\global\long\def\LGmu{\mathcal{G}_{\mu}}
\global\long\def\LI{\mathcal{I}}
\global\long\def\LJ{\mathcal{J}}
\global\long\def\LL{\mathcal{L}}
\global\long\def\LM{\mathcal{M}}
\global\long\def\LN{\mathcal{N}}
\global\long\def\OO{\mathcal{O}}
\global\long\def\LQ{\mathcal{Q}}
\global\long\def\LR{\mathcal{R}}
\global\long\def\LT{\mathcal{T}}
\global\long\def\LS{\mathcal{S}}
\global\long\def\LU{\mathcal{U}}
\global\long\def\LV{\mathcal{V}}
\global\long\def\LW{\mathcal{W}}
\global\long\def\LX{\mathcal{X}}
\global\long\def\LY{\mathcal{Y}}
\global\long\def\PartF{\mathcal{Z}}
\global\long\def\LH{\mathcal{H}}
\global\long\def\LJ{\mathcal{J}}

\global\long\def\blm{m}

\global\long\def\LZ{\mathcal{Z}}
\global\long\def\LZrp{\mathcal{Z}_{\alpha; \bs{s}}^{(p)}}
\global\long\def\LJrp{\mathcal{J}_{\alpha; \bs{s}}^{(p)}}
\global\long\def\chamberrp{\chamber_{\alpha; \bs{s}}^{(p)}}
\global\long\def\LErp{\mathcal{E}_{\alpha; \bs{s}}^{(p)}}
\global\long\def\Greenrp{G_{\alpha; \bs{s}}^{(p)}}
\global\long\def\Prp{P_{\alpha; \bs{s}}^{(p)}}
\global\long\def\norcst{\mathrm{C}_{\kappa}^{(\mathfrak{r})}}
\global\long\def\LZalphar{\mathcal{Z}_{\alpha}^{(\mathfrak{r})}}

\newcommand{\LZtwo}{\mathcal{Z}_{\includegraphics[scale=0.15]{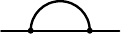}}}
\newcommand{\Etwo}{\mathbb{E}_{\includegraphics[scale=0.15]{figures/link-0}}}
\newcommand{\QQtwo}{\QQ_{\includegraphics[scale=0.15]{figures/link-0}}}
\newcommand{\LEtwo}{\mathcal{E}_{\includegraphics[scale=0.15]{figures/link-0}}}
\newcommand{\LItwo}{\mathcal{I}_{\includegraphics[scale=0.15]{figures/link-0}}}
\newcommand{\LUtwo}{\mathcal{U}_{\includegraphics[scale=0.15]{figures/link-0}}}
\newcommand{\LZtwor}{\LZtwo^{(\mathfrak{r})}}
\newcommand{\LHtwo}{\mathcal{H}_{\includegraphics[scale=0.15]{figures/link-0}}}
\newcommand{\LFtwo}{\mathcal{F}_{\includegraphics[scale=0.15]{figures/link-0}}}
\newcommand{\chambertwo}{\chamber_{\includegraphics[scale=0.15]{figures/link-0}}}

\newcommand{\LZfoura}{\mathcal{Z}_{\includegraphics[scale=0.15]{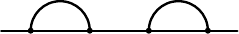}}}
\newcommand{\LZfourb}{\mathcal{Z}_{\includegraphics[scale=0.15]{figures/link-2}}}
\newcommand{\LHfoura}{\mathcal{H}_{\includegraphics[scale=0.15]{figures/link-1}}}
\newcommand{\LHfourb}{\mathcal{H}_{\includegraphics[scale=0.15]{figures/link-2}}}
\newcommand{\LFfoura}{\mathcal{F}_{\includegraphics[scale=0.15]{figures/link-1}}}
\newcommand{\LFfourb}{\mathcal{F}_{\includegraphics[scale=0.15]{figures/link-2}}}
\newcommand{\LFfouraRenorm}{\widehat{\mathcal{F}}_{\includegraphics[scale=0.15]{figures/link-1}}}
\newcommand{\LFfourbRenorm}{\widehat{\mathcal{F}}_{\includegraphics[scale=0.15]{figures/link-2}}}

\newcommand{\QQrainbow}[1]{\mathbb{Q}_{\rainbow_{#1}}}
\newcommand{\LZrainbow}[1]{\mathcal{Z}_{\rainbow_{#1}}}
\newcommand{\Frainbow}[1]{\mathscr{F}_{\rainbow_{#1}}}
\newcommand{\chamberrainbow}[1]{\chamber_{\rainbow_{#1}}}

\newcommand{\QQfusion}[1]{\mathbb{Q}_{\includegraphics[scale=0.5]{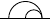}_{#1}}}
\newcommand{\LZfusion}[1]{\mathcal{Z}_{\includegraphics[scale=0.5]{figures/link4fusion}_{#1}}}
\newcommand{\LZfusionInv}[1]{\mathcal{Z}_{\rotatebox{180}{\scalebox{1}[-1]{\includegraphics[scale=0.5]{figures/link4fusion}}}_{#1}}}
\newcommand{\Ffusion}[1]{\mathscr{F}_{\includegraphics[scale=0.5]{figures/link4fusion}_{#1}}}

\newcommand{\coulombGasHRenorm}{\widehat{\coulombGasH}}

\global\long\def\coulomb{\LH}
\global\long\def\auxcoulomb{\hat{\coulomb}}
\global\long\def\coulombGas{\LF}
\global\long\def\coulombnew{\LK}
\global\long\def\coulombLine{\LG}
\global\long\def\kfunc{p}

\global\long\def\eps{\epsilon}
\global\long\def\ov{\overline}
\global\long\def\QQrp{\QQ_{\alpha; \bs{s}}^{(p)}}

\global\long\def\bn{\mathbf{n}}
\global\long\def\MR{MR}
\global\long\def\cond{\,|\,}
\global\long\def\bigcond{\,\big|\,}
\global\long\def\Bigcond{\;\Big|\;}
\global\long\def\la{\langle}
\global\long\def\ra{\rangle}
\global\long\def\tree{\Upsilon}
\global\long\def\prob{\mathbb{P}}
\global\long\def\hm{\mathrm{Hm}}
%

\global\long\def\Im{\operatorname{Im}}
\global\long\def\Re{\operatorname{Re}}

\global\long\def\ud{\mathrm{d}}
\global\long\def\pder#1{\frac{\partial}{\partial#1}}
\global\long\def\pdder#1{\frac{\partial^{2}}{\partial#1^{2}}}
\global\long\def\pddder#1{\frac{\partial^{3}}{\partial#1^{3}}}
\global\long\def\der#1{\frac{\ud}{\ud#1}}

\global\long\def\bZnn{\mathbb{Z}_{\geq 0}}
\global\long\def\bZpos{\mathbb{Z}_{> 0}}
\global\long\def\bZneg{\mathbb{Z}_{< 0}}

\global\long\def\Vfunc{\LG}
\global\long\def\gfunc{g^{(\rr)}}
\global\long\def\hfunc{h^{(\rr)}}

\global\long\def\SimplexInt{\rho}
\global\long\def\CubeInt{\widetilde{\rho}}

\global\long\def\ii{\mathrm{i}}
\global\long\def\ee{\mathrm{e}}
\global\long\def\rr{\mathfrak{r}}
\global\long\def\chamber{\mathfrak{X}}
\global\long\def\Wchamber{\mathfrak{W}}

\global\long\def\SimplexIntKappa8{\SimplexInt}

\global\long\def\nested{\boldsymbol{\underline{\Cap}}}
\global\long\def\unnested{\boldsymbol{\underline{\cap\cap}}}
\global\long\def\unnested{\boldsymbol{\underline{\cap\cap}}}

\global\long\def\acycle{\vartheta}
\global\long\def\bcycle{\tilde{\acycle}}

\global\long\def\metric{\mathrm{dist}}

\global\long\def\adj#1{\mathrm{adj}(#1)}

\global\long\def\bs{\boldsymbol}

\global\long\def\edge#1#2{\langle #1,#2 \rangle}
\global\long\def\graph{G}

\newcommand{\conn}{\varsigma}
\newcommand{\realacycle}{\smash{\mathring{\acycle}}}
\newcommand{\realpt}{\smash{\mathring{x}}}
\newcommand{\corrind}{\LC}
\newcommand{\bssymb}{\pi}
\newcommand{\coeff}{p}
\newcommand{\MainConst}{C}

\global\long\def\removeLink{/}

\global\long\def\domainofdef{\mathfrak{U}}
\global\long\def\Test_space{C_c^\infty}
\global\long\def\Distr_space{(\Test_space)^*}

\global\long\def\bs{\boldsymbol}
\global\long\def\cst{\mathrm{C}}

\newcommand{\red}{\textcolor{red}}
\newcommand{\blue}{\textcolor{blue}}
\newcommand{\green}{\textcolor{green}}
\newcommand{\magenta}{\textcolor{magenta}}
\newcommand{\cyan}{\textcolor{cyan}}

\newcommand{\coulombGasH}{\mathcal{H}}
\newcommand{\secondbeta}{\intloop}

\newcommand{\cev}[1]{\reflectbox{\ensuremath{\vec{\reflectbox{\ensuremath{#1}}}}}}

\global\long\def\anticonf{\zeta}
\global\long\def\intloop{\varrho}
\global\long\def\Gloop{\smash{\mathring{\intloop}}}

\global\long\def\SLEmeasure{\mathrm{P}}
\global\long\def\SLEmeasureEx{\mathrm{E}}

\global\long\def\fugacity{\nu}
\global\long\def\meanderMat{\mathcal{M}}
\global\long\def\LM{\mathcal{M}}
\global\long\def\meanderMatrix{\meanderMat_{\fugacity}}
\global\long\def\meanderMatrixPrime{\meanderMat_{\fugacity(\kappa')}}
\global\long\def\meanderRenorm{\widehat{\mathcal{M}}}

\global\long\def\PartFRenorm{\widehat{\PartF}}
\global\long\def\coulombGasRenorm{\widehat{\coulombGas}}

\global\long\def\hexa{\scalebox{1.3}{\hexagon}}

\global\long\def\np{p}

\global\long\def\FKdual{\mathcal{L}}

\global\long\def\fixedindex{\flat}

\maketitle
\vspace{-1cm}
\begin{center}
\begin{minipage}{0.95\textwidth}
\abstract{
We consider multiple radial Schramm–Loewner evolution (SLE) curves with various time parameterizations and possible spiraling behavior. We construct them by tilting independent radial SLEs with a suitable local martingale, generalizing the earlier construction by Healey and Lawler. We prove that the curves are almost surely transient (i.e., they emanate from boundary points and terminate at a common interior target point). We show that they enjoy the resampling property: conditional on all of the curves but one, the remaining curve is distributed as chordal SLE in the remaining domain. We also verify that the multiradial SLE measure satisfies a natural boundary perturbation property analogous to that of the known SLE variants, involving its partition function (which is finite).

Interestingly, in the parlance of Coulomb gas formalism in conformal field theory, partition functions of multiradial SLE processes with spiral involve both electric and magnetic charges. 
}

%
\end{minipage}
\end{center}
\tableofcontents


\section{Introduction}

Motivated by providing a description for scaling limits of interfaces in critical two-dimensional lattice models, 
in his pioneering work~\cite{Schramm:Scaling_limits_of_LERW_and_UST}, 
Oded Schramm introduced SLE curves and (among other things) showed that in the chordal case, 
they can be axiomatically understood as probability measures on curves connecting a pair of distinct boundary points of a simply connected domain, 
uniquely characterized by two natural properties: conformal invariance (CI) and the domain Markov property (DMP). 
Indeed, these two properties single out a one-parameter family $(\SLE_\kappa)_{\kappa > 0}$ of such random chords. 
Notably, since the geometry of a chord connecting two boundary points is preserved under scaling, chordal $\SLE_\kappa$ is naturally a globally scale-invariant model. 
Early on, these curves turned out to be not only crucial in the rigorous description of scaling limits 
(e.g., in~\cite{Smirnov:Critical_percolation_in_the_plane, 
LSW:Conformal_invariance_of_planar_LERW_and_UST,
Camia-Newman:Critical_percolation_exploration_path_and_SLE_proof_of_convergence, 
Schramm-Sheffield:Contour_lines_of_2D_discrete_GFF,
CDHKS:Convergence_of_Ising_interfaces_to_SLE}), 
but were also used to solve important problems in probability theory 
(e.g.~\cite{LSW:The_dimension_of_the_planar_Brownian_frontier_is_four_thirds, LSW:Brownian_intersection_exponents1, LSW:Brownian_intersection_exponents2}), 
and later in many contexts in random conformal geometry and mathematical physics. 
Moreover, SLEs immediately gained the attention of physicists, 
and were observed to share very deep connections to conformal field theory 
(CFT)~\cite{Bauer-Bernard:SLE_growth_processes_and_CFTs, 
Friedrich-Werner:Conformal_restriction_highest_weight_representations_and_SLE, 
Bauer-Bernard:CFTs_of_SLE_radial_case,
Kontsevich:CFT_SLE_and_phase_boundaries}.
They thus became a key object in many areas, and still remain a very active topic of research.

\smallskip

In the radial case, where a single curve connects a boundary point to an interior target point, Schramm's curves lose global scale-invariance due to the presence of a fixed interior point.
These curves are typically described via the Loewner evolution, a dynamical process that grows the curve from the boundary into the domain.
This growth is governed by a one-dimensional stochastic process $(\xi_t)_{t \geq 0}$ on the boundary, referred to as the \emph{driving function}.
In fact, a single radial SLE curve, which is a random curve connecting a boundary point to an interior point of a simply connected domain, 
is only characterized by (CI, DMP) up to a two-parameter family of probability measures of such curves, $(\SLE_\kappa^\mu)_{\kappa>0, \, \mu \in \R}$~\cite{Zhan:Reversibility_of_whole-plane_SLE, Miller-Sheffield:Imaginary_geometry4}. 
Namely, (CI, DMP) imply that the Loewner driving function of the curve (which describes its time evolution for $t \geq 0$) 
has to contain a multiple of (one-dimensional) Brownian motion, 
which in the radial case can also possess a linear drift:~$\sqrt{\kappa} B_t + \mu t$ (while in the chordal case the driving function is simply $\sqrt{\kappa} B_t$).
The parameter $\kappa>0$ is, similarly to the chordal case, the variance (speed) of the Brownian motion, which describes, e.g., the Hausdorff dimension of the 
curve~\cite{Rohde-Schramm:Basic_properties_of_SLE, Beffara:Dimension_of_SLE}.  
The parameter $\mu \in \R$ describes the spiraling rate of the curve around its interior endpoint, 
and is closely related to multifractal properties of the curve~\cite{Binder:Harmonic_measure_and_rotation_of_simply_connected_planar_domains, Duplantier-Binder:Harmonic_measure_and_winding_of_conformally_invariant_curves}.
In the parlance of Coulomb gas formalism in CFT, the spiral gives rise to a magnetic charge~\cite{Gruzberg:Stochastic_geometry_of_critical_curves_Schramm-Loewner_evolutions_and_conformal_field_theory, RBGW:Critical_curves_in_conformally_invariant_statistical_systems, BGR:Statistics_of_harmonic_measure_and_winding_of_critical_curves_from_conformal_field_theory}. 

\smallskip

When one considers models involving several interacting curves, the classification problem involves even more conformal moduli.
In the case of multiple chordal curves, the moduli are encoded by the positions of the points and the topological connectivity pattern of the random curves.
Namely, one can prove that given this data, the law of the curves is uniquely characterized by a resampling property~\cite{Miller-Sheffield:Imaginary_geometry2, BPW:On_the_uniqueness_of_global_multiple_SLEs, Zhan:Existence_and_uniqueness_of_nonsimple_multiple_SLE}: 
conditional on all of the curves but one, the remaining curve is distributed as chordal SLE in the remaining domain.
Such a property is naturally inherited from multiple interfaces in discrete lattice models.
The present work is concerned with the case of multiple radial curves, a generalization of the model constructed by Healey~\&~Lawler~\cite{Healey-Lawler:N_sided_radial_SLE},
where the curves are targeted to a common interior endpoint. 
In this case, the resampling property does not characterize the curve model uniquely.
Indeed, because the curves can spiral around their common endpoint, any linear combination of multiradial SLE measures with different spiraling rates satisfies the resampling property.
See~\cite{KWW:Commutation_relations_for_two-sided_radial_SLE} for two curves.

\smallskip

The construction of multiple SLEs and their inherent resampling property relies on their conformal restriction property~\cite{LSW:Conformal_restriction_the_chordal_case, 
Kozdron-Lawler:Configurational_measure_on_mutually_avoiding_SLEs, 
Lawler:Partition_functions_loop_measure_and_versions_of_SLE, 
Lawler:Defining_SLE_multiply_connected, 
Jahangoshahi-Lawler:Multiple-paths_SLE_in_multiply_connected_domains}. 
We use it in the form of a  ``boundary perturbation property'' formulated in terms of a modification of the underlying domain by chopping off part of it 
(for example, by discovering a curve in a multiple SLE). 
The restriction property played an instrumental role in the early breakthrough works of Lawler, Schramm, and Werner  
relating $\SLE_6$ and $\SLE_{8/3}$ to Brownian motion and self-avoiding walk~\cite{LSW:The_dimension_of_the_planar_Brownian_frontier_is_four_thirds, LSW:On_the_scaling_limit_of_planar_SAW}. 
It is also a manifestation of the deep relationship of SLE with CFT --- 
the ``conformal anomaly'' appearing upon deforming the domain is also 
the one that gives rise to the action of the central element of the Virasoro algebra with the appropriate central charge~\cite{Bauer-Bernard:SLE_growth_processes_and_CFTs, Kontsevich:CFT_SLE_and_phase_boundaries} 
(see also~\cite{Maibach-Peltola:From_the_conformal_anomaly_to_the_Virasoro_algebra} for a different perspective on this).
This property lies at the heart of canonical SLE measures introduced by 
Lawler, Schramm~\&~Werner~\cite{LSW:Conformal_restriction_the_chordal_case, Werner:The_conformally_invariant_measure_on_self-avoiding_loops} and 
Kontsevich~\&~Suhov~\cite{Kontsevich-Suhov:On_Malliavin_measures_SLE_and_CFT}. 
See also~\cite{Chavez-Pickrell:Werners_measure_on_self-avoiding_loops_and_welding, Benoist-Dubedat:SLE2_loop_measure, Zhan:SLE_loop_measures, Baverez-Jego:The_CFT_of_SLE_loop_measures_and_the_Kontsevich-Suhov_conjecture} and references therein. 

\medskip

This article is organized as follows. Section~\ref{subsec::intro_preli} describes the setup and gathers notation. 
In Section~\ref{subsec::intro_multiradialSLE}, we define the multiradial SLE with spiral and briefly discuss its CFT/Coulomb gas interpretation. 
We state the radial resampling property in Theorem~\ref{thm::multiradial_resampling} in Section~\ref{subsec:resampling_property}, and describe the marginal law of one curve in a multiradial SLE sample. 
In Section~\ref{subsec:bdry_perturbation}, we state the boundary perturbation property for the multiradial SLE curves in Proposition~\ref{prop::multiradial_bp}.

Section~\ref{sec::multitime_mart} concerns the definition of multiradial SLE with spiral, the associated martingale, and basic properties.
In Section~\ref{sec::halfwatermelonSLE}, we consider half-watermelon SLEs, which are multichordal SLEs whose endpoints have been fused together. 
We show that when the endpoint of multiradial SLE tends to the boundary, the multiradial curves become half-watermelon curves.
Section~\ref{sec::halfwatermelonSLE} also involves delicate, important estimates for SLE partition functions. 
The last Section~\ref{sec::Multiradial_final} concludes with the proofs of the main results. 
In Appendix~\ref{app:transience}, we prove transience of radial SLE with multiple force points and spiral, generalizing Lawler's earlier work.

\subsection{Multiradial Loewner evolutions}
\label{subsec::intro_preli}

\paragraph*{Polygons.}
We say that $(\Omega; x^1, x^2, \ldots, x^p) =: (\Omega; \bs x)$ 
is a (topological) $p$-\emph{polygon} if $\Omega\subsetneq\C$ is simply connected and $x^1, x^2, \ldots, x^p \in \partial\Omega$ are distinct points lying counterclockwise along the boundary. 
We will also assume throughout that $\partial\Omega$ is locally connected and the marked boundary points $x^1, x^2, \ldots, x^p$ lie on $C^{1+\eps}$-boundary segments, for some $\eps>0$, so that derivatives of conformal maps on $\Omega$ are defined there.

We also consider $p$-polygons $(\Omega; \bs x; z)$ with a marked interior point $z\in\Omega$.
We frequently use the reference polygon $(\U; \ee^{\ii\bs{\theta}};0)$ which is the unit disc $\U := \{ z \in \C \colon |z| < 1 \}$ with marked boundary points $\ee^{\ii\bs{\theta}} :=(\ee^{\ii\theta^1},\ee^{\ii\theta^2}, \ldots, \ee^{\ii\theta^p})$
parameterized by angle coordinates $\bs{\theta} := (\theta^1, \theta^2, \ldots, \theta^p)\in \LX_p$ in the torus
\begin{align}\label{eqn::torus}
\LX_p := \big\{\bs{\theta} = (\theta^1, \theta^2, \ldots, \theta^p)\in\R^p \;|\;  \theta^1<\theta^2<\cdots<\theta^p<\theta^1+2\pi \big\} .
\end{align}

We will investigate probability measures on curves in polygons connecting some of the marked points together. 
In the Loewner description of such curves, a nontrivial detail that needs to be addressed is the parameterization of these curves when they are grown simultaneously. 
We focus on the radial case. 

\paragraph*{Radial Loewner equation.}
Fix $\theta\in \LX_1 = \R$ and $T \in (0,\infty)$.
Let $\gamma \colon [0,T]\to \overline{\U}$ be a continuous simple\footnote{We will be mainly focusing on $\SLE_\kappa$ curves with $\kappa \in (0,4]$, which are almost surely simple. In general, the mapping-out function should be defined as a map $g_t \colon U_t \to \U$, where $U_t$ is the connected component of $\U\setminus \gamma_{[0,t]}$ containing the origin (target point of the curve). Since we only need the cases where $\U\setminus \gamma_{[0,t]}$ is simply connected, for notational ease, we omit this detail throughout.} curve such that $\gamma_0=\ee^{\ii \theta}$ and $\gamma_{(0,T)} \subset \U \setminus \{0\}$. 
For each $t \in [0,T]$, let $g_t \colon \U\setminus \gamma_{[0,t]} \to \U$ be the unique conformal bijection 
normalized at the origin, i.e., 
$g_t(0)=0$ and $g'_t(0)>0$, which we call the \emph{mapping-out function of} $\gamma$.
We say that the curve $\gamma$ is \emph{parameterized by capacity} if $g'_t(0) = \ee^t$ for all $t \in [0,T]$. 
These maps $\{g_t \colon t \in [0,T] \}$ solve the \emph{radial Loewner equation}
\begin{align}\label{eq:single_radial_Loewner_equation}\tag{LE}
\partial_t g_t(z)=g_t(z) \, \frac{\ee^{\ii \xi_t}+g_t(z)}{\ee^{\ii \xi_t}-g_t(z)}, \qquad g_0(z)=z,
\end{align}
where $\xi \colon [0,T]\to \R$ is a continuous function called the Loewner \emph{driving function} of $\gamma$. 
It is convenient to use the \emph{covering map} $\covmap_t \colon \R \to \R$ 
of $g_t$ defined via $g_t(\ee^{\ii \theta}) = \exp( \ii \covmap_t(\theta) )$ satisfying the Loewner equation
\begin{align}\label{eq:single_radial_Loewner_equation_cov}
\partial_t \covmap_t(\theta) = \cot \bigg(\frac{\covmap_t(\theta) - \xi_t}{2}\bigg) , \qquad \covmap_0(\theta)=\theta .
\end{align}

\begin{figure}[ht!]
\begin{center}
\includegraphics[width=0.6\textwidth]{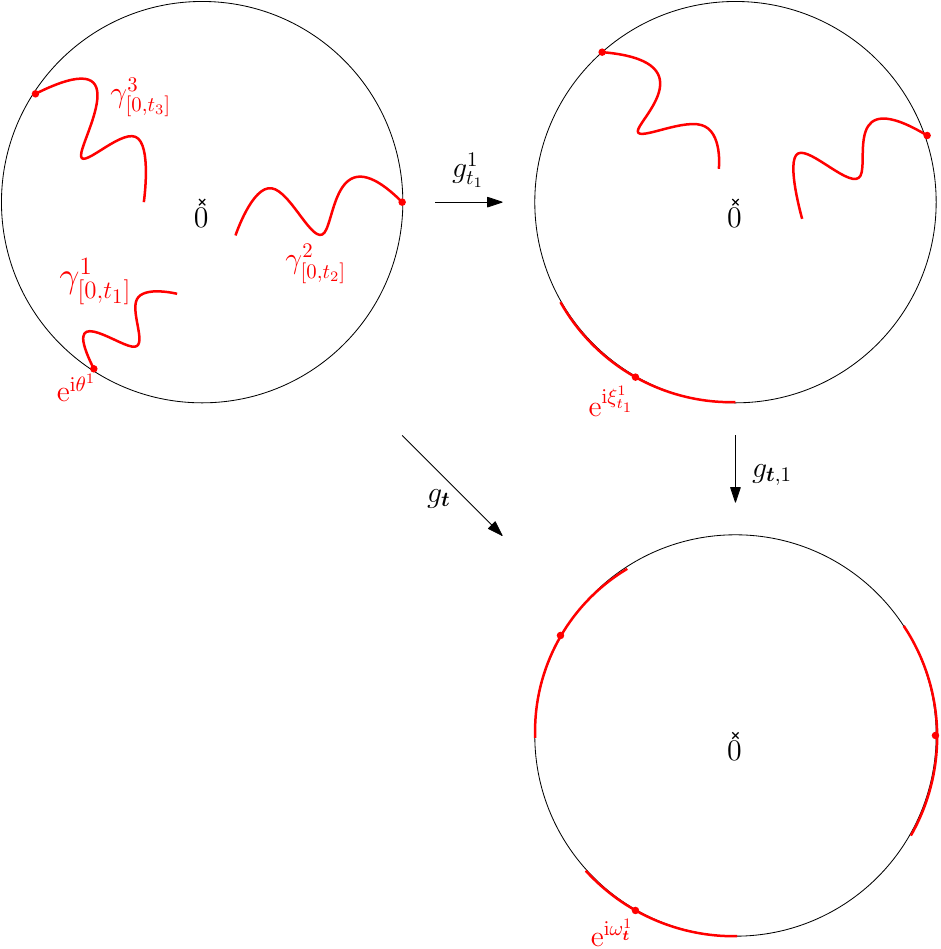}
\end{center}
\caption{\label{fig::multitime}Illustration for the notation related to the multi-time parameter.}
\end{figure}

\paragraph*{Multi-time parameter.}
(See Figure~\ref{fig::multitime}.)
Fix $p\ge 2$ and $\bs{\theta} = (\theta^1, \ldots, \theta^p)\in\LX_p$. Consider a $p$-tuple 
$\bs{\gamma}_{\bs{t}} := ( \gamma^{1}_{t_1}, \ldots, \gamma^{p}_{t_p} )$ 
of simple curves in the polygon $(\U; \ee^{\ii\bs{\theta}};0)$,
parameterized by $\bs{t} := (t_1,\ldots,t_p) \in [0,\infty)^p$,  
such that $\gamma^{j}_{0} = \ee^{\ii\theta^j}$ for each $j$
and the segments $\{ \gamma^{j}_{[0,t_j]}\}_{1\le j\le p}$ are disjoint. 
Consider the conformal bijection 
\begin{align*}
g_{\bs{t}} \colon \U \setminus \bigcup_{j=1}^p \gamma^{j}_{[0,t_j]} \longrightarrow \U ,
\qquad g_{\bs{t}}(0)=0, \; g'_{\bs{t}}(0)>0 ,
\end{align*}
normalized at the origin, 
and the mapping-out functions $g_{t_j}^{j}$ of $\gamma^{j}$, for $1\le j\le p$. 
Then, we have 
\begin{align*}
g_{\bs{t}}=g_{\bs{t},j} \circ g_{t_j}^{j} 
\qquad \textnormal{for all } 1\le j\le p, \qquad \textnormal{where } \;
g_{\bs{t},j} \colon \U \setminus g_{t_j}^{j} \Big( \bigcup_{i\neq j} \gamma^{i}_{[0,t_i]} \Big) \longrightarrow \U 
\end{align*}
is also normalized at the origin.
We say that the $p$-tuple $\bs{\gamma}_{\bs{t}}$ of curves has \emph{$p$-time parameter} 
if the mapping-out functions of each curve are separately parameterized by capacity, i.e., 
$( g_{t_j}^{j} )'(0) = \ee^{t_j}$ for all $1\le j\le p$.
The construction of multi-time parameter is well-defined up to the collision time which is the first time when any two curves in $\bs{\gamma}$ meet. 
Let $\covmap_{t_j}^{j},\covmap_{\bs{t}},\covmap_{\bs{t},j}$ be the covering maps of $g_{t_j}^{j},g_{\bs{t}},g_{\bs{t},j}$, respectively. 
Note that $\covmap_{\bs{t}} = \covmap_{\bs{t},j} \circ \covmap_{t_j}^{j}$. 
Denoting by $\xi^{j}$ the driving function of each $\gamma^{j}$ (as in~\eqref{eq:single_radial_Loewner_equation}), 
we can define the Loewner driving function of the multi-slit $\bs{\gamma}_{\bs{t}}$, started at $\bs{\mslitdriv}_{\bs{0}} = \bs{\theta} = (\theta^1, \ldots, \theta^p)$, 
by 
\begin{align}\label{eqn::Ntuple_driving}
\bs{\mslitdriv}_{\bs{t}} := \big(\mslitdriv^{1}_{\bs{t}},\ldots,\mslitdriv^{p}_{\bs{t}}\big),
\qquad \textnormal{with} \quad \mslitdriv_{\bs{t}}^{j}:=\covmap_{\bs{t},j} ( \xi_{t_j}^{j} ), \quad \textnormal{for }1\le j\le p.
\end{align}
Writing $a_{\bs{t}}^{j} := \partial_{t_j} \log g_{\bs{t}}'(0) = (\covmap_{\bs{t},j}' (\xi_{t_j}^{j}))^2$, 
one can check that $g_{\bs{t}}$ solves the Loewner equation
\begin{align}\label{eq:LE_multislitg}\tag{$p$-LE}
\ud g_{\bs{t}}(z) 
= \sum_{j=1}^{p} \partial_{t_j} g_{(t_1,\ldots,t_p)}(z) \, \ud t_j
= g_{\bs{t}}(z) \sum_{j=1}^{p} a_{\bs{t}}^{j}  \, \frac{\ee^{\ii \mslitdriv_{\bs{t}}^{j}}+g_{\bs{t}}(z)}{\ee^{\ii \mslitdriv_{\bs{t}}^{j}}-g_{\bs{t}}(z)} \, \ud t_j , \qquad g_{\bs{0}}(z) =z .
\end{align}
Moreover, $\covmap_{\bs{t}} = \covmap_{\bs{t},j} \circ \covmap_{t_j}^{j}$ satisfies 
\begin{align}\label{eq:LE_multislith}
\ud \covmap_{\bs{t}} (\theta)
= \sum_{j=1}^{p} \partial_{t_j} \covmap_{(t_1,\ldots,t_p)}(\theta) \, \ud t_j
= \sum_{j=1}^{p} a_{\bs{t}}^{j} \,  \cot \bigg(\frac{\covmap_{\bs{t}} (\theta) - \mslitdriv_{\bs{t}}^{j}}{2}\bigg) \, \ud t_j , \qquad \covmap_{\bs{0}}(\theta) = \theta.
\end{align}

\subsection{Multiradial SLE with spiral}
\label{subsec::intro_multiradialSLE}

Throughout, we will use the following CFT/SLE parameters indexed by $\kappa > 0$:
\begin{align}\label{eqn::universal_parameters}
\mathfrak{b} := \frac{6-\kappa}{2\kappa} = \Delta(\mathfrak{e}_{1,2}) ,\qquad 
\tilde{\mathfrak{b}} := \frac{(6-\kappa)(\kappa-2)}{8\kappa} = 2 \Delta(\mathfrak{e}_{0,1/2}) , 
\qquad \textnormal{and} \qquad 
\mathfrak{c} := \frac{(6-\kappa)(3\kappa-8)}{2\kappa} ,
\end{align}
where $\Delta(\mathfrak{e}) \in \R$ are Coulomb gas conformal weights associated to electric charges %
$\mathfrak{e} \in \R$ defined as
\begin{align}\label{eqn::general_charge_weight}
\Delta(\mathfrak{e}) := \mathfrak{e}^2 - 2 \mathfrak{e} \mathfrak{e}_0 , \qquad 
\mathfrak{e}_0 : = \frac{\kappa - 4}{4 \sqrt{\kappa}} ,
\qquad \kappa > 0 ,
\end{align}
$\mathfrak{c} = 1 - 24 \mathfrak{e}_0^2$ in~\eqref{eqn::universal_parameters} is the central charge, 
and $\mathfrak{b}$ and $\tilde{\mathfrak{b}}$ are $\SLE_\kappa$ \emph{boundary and interior scaling exponents} 
given in terms of the special class of \emph{Kac type conformal 
weights}\footnote{The (extended) Kac table comprises those conformal weights $\{ \Delta(\mathfrak{e}_{r,s}) \colon r,s \in \N_{>0} \}$ where the indices are positive integers, which correspond to degenerate modules for the Virasoro algebra. However, many important observables in lattice models and for $\SLE_\kappa$ do not belong to the Kac table as such, but rather to the (more extended) family $\{ \Delta(\mathfrak{e}_{r,s}) \colon r,s \in \R_{\geq 0} \}$. This is the case, in particular, for many (but not all) bulk observables related to paths and loops in the interior of the domain.} 
$\Delta(\mathfrak{e}_{r,s}) = \tfrac{1}{4}(r^2-1) \tfrac{\kappa}{4} + \tfrac{1}{4}(s^2-1) \tfrac{4}{\kappa} + \tfrac{1}{2}(1- rs)$ 
labeled by 
\begin{align} \label{eqn::Kac_type_charge}
\mathfrak{e}_{r,s} := \frac{1}{2}\Big((1-r)\frac{\sqrt{\kappa}}{2} - (1-s)\frac{2}{\sqrt{\kappa}}\Big) ,
\qquad 
\qquad r,s \in \R_{\geq 0} .
\end{align}

\begin{definition}[Multiradial SLE with spiral]\label{def::multiradialSLEspiral}
For each $j\in \{1,\ldots, p\}$, let $\gamma^{j}$ be radial $\SLE_{\kappa}$ in $(\U; \ee^{\ii\theta^j}; 0)$ 
and let $\mathsf{P}_p$ 
be the probability measure on
$\bs{\gamma}=(\gamma^{1}, \ldots, \gamma^{p})$ 
under which the curves  
are independent. 
We parameterize $\bs{\gamma}$ by $p$-time parameter $\bs{t}$. 
We define \emph{$p$-radial $\SLE_{\kappa}$ with spiraling rate $\mu \in \R$ in} $(\U; \ee^{\ii\bs{\theta}};0)$, 
dubbed \emph{$p$-radial $\SLE_{\kappa}^\mu$} from now on,
as the probability measure $\PPspiral{p} = \PPspiral{p}(\U; \ee^{\ii\bs{\theta}};0)$ obtained by tilting $\mathsf{P}_p$ by the local martingale 
\begin{align}\label{eqn::multiradialSLE_mart}
M_{\bs{t}}(\nradpartfn{p}{\mu})
:= \one_{\LE_{\emptyset}(\bs{\gamma}_{\bs{t}})} \, \exp\bigg(\frac{\mathfrak{c}}{2}\blm_{\bs{t}} - \tilde{\mathfrak{b}}\sum_{j=1}^{p}t_j\bigg) 
\; (g_{\bs{t}}'(0))^{\tilde{\mathfrak{b}}+\frac{p^2-1}{2\kappa}-\frac{\mu^2}{2\kappa}} 
\Big(\prod_{j=1}^{p} \covmap_{\bs{t},j}'(\xi_{t_j}^{j}) \Big)^{\mathfrak{b}}
\nradpartfn{p}{\mu}(\bs{\mslitdriv}_{\bs{t}}),
\end{align}
where $\LE_{\emptyset}(\bs{\gamma}_{\bs{t}}) = \{ \gamma_{[0,t_j]}^{j} \cap \gamma_{[0,t_i]}^{i}=\emptyset, \, \forall i\neq j\}$ is the event that different curves are disjoint, 
$\blm_{\bs{t}} = \blm(\bs\xi_{\bs{t}})$ is the unique potential solving the exact differential equation (see Lemma~\ref{lem::blm_exact})
\begin{align}\label{eqn::mt_def}
\ud  \blm_{\bs{t}} 
= \sum_{j=1}^{p} \partial_{t_j} \blm_{(t_1,\ldots,t_p)} \ud t_j
= \sum_{j=1}^{p} \Big( \! -\frac{1}{3} (\LS \covmap_{\bs{t},j})(\xi_{t_j}^{j})+\frac{1}{6} \big(1- (\covmap'_{\bs{t},j}(\xi_{t_j}^{j}))^2 \big) \Big) \ud t_j, 
\qquad \blm_{\bs{0}} = 0 ,
\end{align}
with $\LS\covmap=\frac{\covmap'''}{\covmap'}-\frac{3}{2}\big(\frac{\covmap''}{\covmap'}\big)^2$ denoting the Schwarzian derivative of a function $\covmap$, 
and where 
$\nradpartfn{p}{\mu}$ is the multiradial $\SLE_{\kappa}^\mu$ partition function defined in Equations~(\ref{eq:multiradial_partition_function},~\ref{eqn:spiralpartitionfunction}) below,
and $\bs\xi_{\bs{t}} = (\xi_{t_1}^1, \ldots, \xi_{t_p}^p)$ are the driving functions the individual curves, 
$\bs{\mslitdriv}_{\bs{t}} = (\mslitdriv_{\bs{t}}^1, \ldots, \mslitdriv_{\bs{t}}^p)$ is the driving function of the multi-slit $\bs{\gamma}$.
\end{definition}

\noindent 
Before discussing partition functions, we make some comments and extensions to the above definition.
\begin{itemize}
\item 
We define $p$-radial $\SLE_{\kappa}^\mu$ measure $\PPspiral{p}(\Omega; \bs x; z)$ 
in a general $p$-polygon as the pushforward of $\PPspiral{p}$ by the map $\varphi^{-1}$, where 
$\varphi \colon \Omega \to \U$ is any conformal bijection such that $\varphi(z)=0$, 
and $\varphi(x^j) =: \ee^{\ii\theta^j}$, for each $1\le j\le p$.
Note that any two such mappings differ by a rotation, which does not affect the law.
\item 
When $\kappa \in (0,4]$, the local martingale~\eqref{eqn::multiradialSLE_mart} is a true martingale. 
In particular, $p$-radial $\SLE_{\kappa}^{\mu}$ curves almost surely never reach the collision time, so the process can be defined for all time (Lemma~\ref{lem::multiradial_collision}).
\item 
When $\mu=0$, the measure $\PPnospiral{p}(\U; \ee^{\ii\bs{\theta}};0)$ is the same as the $p$-sided radial $\SLE_{\kappa}$ (without spiral) introduced by Healey and Lawler~\cite{Healey-Lawler:N_sided_radial_SLE}. 
In accordance with the well-known conformal restriction properties of SLE, the quantity $\blm_{\bs{t}}$ also has 
an interesting Brownian loop measure interpretation --- see Lemma~\ref{lem::blm},
which reflects the fact that the central charge $\mathfrak{c}$ encodes the ``conformal anomaly'' for the SLE curves,
i.e., the response of the curves to changes of the underlying metric. 
\end{itemize}

\paragraph*{Partition functions.}
The $p$-radial $\SLE_\kappa$ partition function is\footnote{We use a convenient normalization for the angle coordinates, and the formula $\sin(\frac{\theta^j - \theta^i}{2}) = \frac{1}{2}|\ee^{\ii\theta^j} - \ee^{\ii\theta^i}|$.} 
\begin{align}\label{eq:multiradial_partition_function}
\begin{split}
\nradpartfn{p}{0}(\bs{\theta})  =  \nradpartfn{p}{0}(\U; \ee^{\ii\bs{\theta}};0) 
:= \; & 
\bigg( \underset{1\leq i < j \leq p}{\prod} \, |\ee^{\ii\theta^j} - \ee^{\ii\theta^i}| \bigg)^{2/\kappa} 
\\
= \; & 2^{\frac{p(p-1)}{\kappa}} \bigg( \underset{1\leq i < j \leq p}{\prod} \, \sin \Big(\frac{\theta^j - \theta^i}{2}\Big) \bigg)^{2/\kappa} , \qquad \bs{\theta} \in \LX_p .
\end{split}
\end{align} 
Note that~\eqref{eq:multiradial_partition_function} is rotation-invariant: we have $\nradpartfn{p}{0}(\theta^1+\alpha,\ldots,\theta^p+\alpha) = \nradpartfn{p}{0}(\theta^1,\ldots,\theta^p)$ for any $\alpha \in \R$.
Therefore, we can extend its definition to $p$-polygons $(\Omega; \bs x; z)$ via conformal covariance as
\begin{align}\label{eqn::multiradial_partition_function_cov}
\nradpartfn{p}{0}(\Omega; \bs x; z)
:= |\varphi'(z)|^{\tilde{\mathfrak{b}}+\frac{p^2-1}{2\kappa}} \Big( 
\prod_{j=1}^p |\varphi'(x^j)| \Big)^{\mathfrak{b}} \nradpartfn{p}{0}(\bs \theta) , 
\end{align}
where $\varphi \colon \Omega \to \U$ is  
any\footnote{Any two such conformal bijections differ by a rotation, which does not affect the formula~\eqref{eqn::multiradial_partition_function_cov}.} 
conformal bijection such that $\varphi(z)=0$, and $\varphi(x^j) =: \ee^{\ii\theta^j}$, for each $1\le j\le p$.
In terms of the Coulomb gas formalism of CFT (see~\cite[Lemma~11]{Dubedat:Commutation_relations_for_SLE}), we have
\begin{align*}
\tilde{\mathfrak{b}}+\frac{p^2-1}{2\kappa} 
\; = \; 2 \Delta(\mathfrak{e}_{0,p/2}) 
\; = \; \frac{4 p^2-(\kappa-4)^2}{8\kappa } .
\end{align*}
Thus, $\nradpartfn{p}{0}$ defined in~(\ref{eq:multiradial_partition_function},~\ref{eqn::multiradial_partition_function_cov}) 
can be thought of as a correlation function of conformal fields of weights $\Delta(\mathfrak{e}_{1,2})$ at the boundary points $x^1, \ldots, x^p$
and a spinless field of weights $(\Delta(\mathfrak{e}_{0,p/2}),\Delta(\mathfrak{e}_{0,p/2}))$ at the bulk point $z$, 
in accordance with the physics literature~\cite{Duplantier-Saleur:Exact_surface_and_wedge_exponents_for_polymers_in_two_dimensions, Duplantier:Conformal_fractal_geometry_and_boundary_quantum_gravity, Gruzberg:Stochastic_geometry_of_critical_curves_Schramm-Loewner_evolutions_and_conformal_field_theory, RBGW:Critical_curves_in_conformally_invariant_statistical_systems}. 
The partition function $\nradpartfn{p}{0}$ is also naturally related to discrete statistical physics models: 
in the follow-up paper~\cite{DingLiuWuTripodUST}, 
it is proven that the scaling limit of a certain trifurcation pattern in a uniform spanning tree (UST) is encoded by $\nradpartfn{3}{0}$ with $\kappa=2$. 
(A related asymptotic result is derived~\cite[Theorem~1.3]{BLLP:Uniform_spanning_trees_and_random_matrix_statistics}.)
The associated tripod in the UST converges to three-sided radial $\SLE_2$ (\cite[Theorem~1.3]{DingLiuWuTripodUST} and~\cite[Theorem~1.3]{BLLP:Uniform_spanning_trees_and_random_matrix_statistics}). 
Notably, the tools developed in the present article play a crucial role in~\cite{DingLiuWuTripodUST}.

\smallskip

Next, we include a \emph{spiraling rate} $\mu\in\R$, and define the $p$-radial $\SLE_\kappa^\mu$ partition function as
\begin{align}\label{eqn:spiralpartitionfunction}
\nradpartfn{p}{\mu}(\bs{\theta}) = \nradpartfn{p}{\mu}(\U; \ee^{\ii\bs{\theta}};0) 
:= \nradpartfn{p}{0}(\bs{\theta}) \, \exp \Big( \frac{\mu}{\kappa} \sum_{j=1}^p \theta^j \Big) ,
\qquad \bs{\theta} \in \LX_p .
\end{align}
Observe that~\eqref{eqn:spiralpartitionfunction} is not rotation-invariant: we have 
$\nradpartfn{p}{\mu}(\theta^1+\alpha,\ldots,\theta^p+\alpha) = \ee^{\frac{\mu}{\kappa} p \alpha } \, \nradpartfn{p}{\mu}(\theta^1,\ldots,\theta^p)$, so we cannot extend its definition similarly as for~\eqref{eqn::multiradial_partition_function_cov}. 
Indeed, the appropriate CFT field at the bulk point $z$ should now involve a spin, that is, its holomorphic and antiholomorphic weights should be different. 
In accordance with the physics literature~\cite{Gruzberg:Stochastic_geometry_of_critical_curves_Schramm-Loewner_evolutions_and_conformal_field_theory, RBGW:Critical_curves_in_conformally_invariant_statistical_systems, BGR:Statistics_of_harmonic_measure_and_winding_of_critical_curves_from_conformal_field_theory}, 
we take the weights to be of type $(\Delta, \Delta^*)$, 
namely $\Delta = \Delta(\mathfrak{e},\mathfrak{m}) := (\mathfrak{e} + \mathfrak{m}) (\mathfrak{e} + \mathfrak{m} - 2\mathfrak{e}_0)$ 
and $\Delta^* = \Delta^*(\mathfrak{e},\mathfrak{m}) := (\mathfrak{e} - \mathfrak{m}) (\mathfrak{e} - \mathfrak{m} - 2\mathfrak{e}_0)$ 
involving both an electric charge $\mathfrak{e} \in \R$ and a magnetic charge $\mathfrak{m} \in \ii \R$,  
which in the case of present interest are 
\begin{align*}
\mathfrak{e} = \mathfrak{e}_{0,p/2} 
\qquad \textnormal{and} \qquad 
\mathfrak{m} = \mathfrak{m}_\mu := \ii \, \frac{\mu }{2 \sqrt{\kappa}}  . 
\end{align*}
Hence the partition function $\nradpartfn{p}{\mu}$ defined in~\eqref{eqn:spiralpartitionfunction} 
can be thought of as a CFT correlation function of weights $\Delta(\mathfrak{e}_{1,2})$ at the boundary points $x^1, \ldots, x^p$
and at the bulk point $z$, a field of complex weights 
\begin{align}\label{eqn::spiral_conformal_weight}
\Delta(\mathfrak{e}_{0,p/2},\mathfrak{m}_\mu) \; = \; \Delta(\mathfrak{e}_{0,p/2}) - \frac{\mu^2}{4\kappa} + \ii \, \frac{\mu}{2\kappa} p 
\qquad \textnormal{and} \qquad 
\Delta^*(\mathfrak{e}_{0,p/2},\mathfrak{m}_\mu) \; = \; \Delta(\mathfrak{e}_{0,p/2}) - \frac{\mu^2}{4\kappa} - \ii \, \frac{\mu}{2\kappa} p .
\end{align}
Motivated by this, we extend the definition of $\nradpartfn{p}{\mu}$ to more general polygons
$(\Omega; \bs x; z)$ with $z\in\Omega$ as\footnote{Here and throughout, we use the principal branch of $\arg(\cdot)$, and denote by $w^*$ the complex conjugate of $w \in \C$.}
\begin{align}\label{eqn::Gmu_cov}
\begin{split}
\nradpartfn{p}{\mu}(\Omega; \bs x; z)
:= \; & (\varphi'(z))^{\Delta(\mathfrak{e}_{0,p/2},\mathfrak{m}_\mu)} \, ({\varphi'(z)}^*)^{\Delta^*(\mathfrak{e}_{0,p/2},\mathfrak{m}_\mu)}
\Big( \prod_{j=1}^p |\varphi'(x^j)| \Big)^{\mathfrak{b}} \nradpartfn{p}{\mu}(\bs \theta) \\
:= \; & \exp \Big(\! -\frac{\mu}{\kappa} p \arg \varphi'(z)\Big) \, 
|\varphi'(z)|^{\tilde{\mathfrak{b}}+\frac{p^2-1}{2\kappa}-\frac{\mu^2}{2\kappa}} 
\Big( \prod_{j=1}^p |\varphi'(x^j)| \Big)^{\mathfrak{b}} \nradpartfn{p}{\mu}(\bs \theta) ,
\end{split}
\end{align}
where $\varphi \colon \Omega \to \U$ is any conformal bijection such that $\varphi(z)=0$, and $\varphi(x^j) =: \ee^{\ii\theta^j}$, for each $1\le j\le p$. 

\subsection{Radial resampling property}\label{subsec:resampling_property}

The resampling property of multiple SLE curves is a natural feature inherited from multiple interfaces in discrete lattice models, 
see~\cite{Miller-Sheffield:Imaginary_geometry1, Miller-Sheffield:Imaginary_geometry2, Peltola-Wu:Global_and_local_multiple_SLEs_and_connection_probabilities_for_level_lines_of_GFF, Wu:Convergence_of_the_critical_planar_ising_interfaces_to_hypergeometric_SLE, BPW:On_the_uniqueness_of_global_multiple_SLEs, AHSY:Conformal_welding_of_quantum_disks_and_multiple_SLE_the_non-simple_case, Zhan:Existence_and_uniqueness_of_nonsimple_multiple_SLE, FLPW:Multiple_SLEs_Coulomb_gas_integrals_and_pure_partition_functions}.
It essentially says that the multiple SLE curve models have a symmetry with respect to exploring a subset of the curves --- i.e., conditionally on some of them, 
the rest are described by a similar SLE curve model with fewer curves. As argued in~\cite[Section~4]{Miller-Sheffield:Imaginary_geometry2} (see also~\cite{BPW:On_the_uniqueness_of_global_multiple_SLEs}), 
the resampling property in the chordal setting uniquely characterizes the law of the collection of curves (cf. Definition~\ref{def::halfwatermelonSLE}). 
However, this is not the case in the radial setting. 
Namely, Theorem~\ref{thm::multiradial_resampling} tells that any linear combination of multiradial SLE with different spiraling rates satisfies the radial resampling property 
(see also~\cite{KWW:Commutation_relations_for_two-sided_radial_SLE}).

\begin{definition}[Half-watermelon SLE]\label{def::halfwatermelonSLE}
Fix $\kappa\in (0,4]$. Let $(\Omega; \bs{x}, y) = (\Omega; x^1, \ldots, x^n, y)$ be an $(n+1)$-polygon. 
Consider curves $\eta^{j}$ in $\Omega$ from $x^j$ to $y$ for each $j$; and suppose $\eta^{i}\cap\eta^{j}=\{y\}$ for $i\neq j$. 
We call the law $\QQfusion{n}(\Omega; \bs{x}; y)$ of $\bs{\eta} = (\eta^{1}, \ldots, \eta^{n})$ \emph{half-$n$-watermelon} $\SLE_{\kappa}$ in $(\Omega; \bs{x}, y)$ from $\bs{x}$ to $y$ if 
it satisfies the \emph{chordal resampling property}:  
for each $j\in \{1, \ldots, n\}$, the conditional law of $\eta^{j}$ given $\{\eta^{i} \colon i\neq j\}$ is chordal $\SLE_{\kappa}$ from $x^j$ to $y$ 
in the connected component $\Omega_j$ of $\Omega\setminus\cup_{i\neq j}\eta^{i}$ having $x^j$ on its boundary.
\end{definition}

The existence and uniqueness of (global) half-$n$-watermelon $\SLE_{\kappa}$ was argued in~\cite[Section~4]{Miller-Sheffield:Imaginary_geometry2}. 
The marginal law of $\eta^{1}$ under $\QQfusion{n}(\Omega; \bs x; y)$ is chordal $\SLE_{\kappa}(2, \ldots, 2)$ in $\Omega$ from $x^1$ to $y$ with force points $(x^2, \ldots, x^n)$. 
Similarly, the marginal law of one curve in a $p$-radial $\SLE_{\kappa}^\mu$ is radial $\SLE_{\kappa}^{\mu}(2, \ldots, 2)$ with spiral (defined in Section~\ref{subsec::pre_radialkapparhomu}), 
and in particular, each individual curve is almost surely transient (see Proposition~\ref{prop::radialSLE_transience}). 
Moreover, we will show in Theorem~\ref{thm::multiradial_halfwatermelon} that multiradial $\SLE_{\kappa}$ converges to half-watermelon $\SLE_{\kappa}$ when the interior point of the polygon approaches a boundary point. 
Using this convergence, we then derive the following characterization.

\begin{theorem}\label{thm::multiradial_resampling}
Fix $\kappa\in (0,4], \, \mu\in\R, \, p\ge 2$, and a $p$-polygon $(\Omega; \bs{x}; z)$ with $\bs{x}=( x^1, \ldots, x^p)$ and $z\in\Omega$.
The $p$-radial $\SLE_{\kappa}^\mu$ curves $\bs \gamma = (\gamma^{1}, \ldots, \gamma^{p})$ 
in $(\Omega; \bs x; z)$ satisfy the following properties. 
\begin{enumerate}
\item 
The marginal law of $\gamma^{1}$ is radial $\SLE_{\kappa}^{\mu}(2, \ldots, 2)$ in $(\Omega; \bs x; z)$ with force points $\hat{\bs x} := (x^2, \ldots, x^p)$. 
\item 
Given $\gamma^{1}$, the conditional law of $(\gamma^{2}, \ldots, \gamma^{p})$ is half-$(p-1)$-watermelon $\SLE_{\kappa}$ in $(\Omega\setminus\gamma^{1}; \hat{\bs x}; z)$. 
\end{enumerate}
\noindent 
In particular, $\bs \gamma$ is a.s.~transient 
\textnormal{(}$\bs{\gamma}_{\bs{t}} \to \bs{0}$ as $\bs{t}\to\infty$; see Proposition~\ref{prop::multiradial_transience} for details\textnormal{)} 
and satisfies the \emph{radial resampling property}:  
for each $j\in \{1, \ldots, p\}$, the conditional law of $\gamma^{j}$ given $\{\gamma^{i} \colon i\neq j\}$ is chordal $\SLE_{\kappa}$ from $x^j$ to $z$ 
in the connected component $\Omega_j$ of $\Omega\setminus\cup_{i\neq j}\gamma^{i}$ having $x^j$ on its boundary.
\end{theorem}

Theorem~\ref{thm::multiradial_resampling} is proved for $p=2$ in~\cite{KWW:Commutation_relations_for_two-sided_radial_SLE}. 
We prove it for general $p\ge 2$ in Section~\ref{subsec::resampling}. 
The proof relies on the fact that for any spiraling rate $\mu \in \R$, multiradial $\SLE_{\kappa}^\mu$ in $(\Omega; \bs{x}, z)$ converges to half-watermelon $\SLE_{\kappa}$ in $(\Omega; \bs{x}, y)$ as $z\to y$. 
Thus, the resampling property is inherited by the radial case from the chordal case.
In contrast to the case of $p=2$, 
this requires rather delicate analysis on the associated partition functions (see Theorem~\ref{thm::multiradial_halfwatermelon} and Lemma~\ref{lem::multiradial_spiral_halfwatermelon}).

To state the radial resampling property in Theorem~\ref{thm::multiradial_resampling}, it is important to know that the curves are transient.
The transience for radial $\SLE_{\kappa}(2)$ follows from~\cite[Theorem~1.3]{Lawler:Continuity_of_radial_and_two-sided_radial_SLE_at_the_terminal_point}. 
We generalize Lawler's proof to derive the transience of radial SLE with multiple force points and spiral in Proposition~\ref{prop::radialSLE_transience} in Appendix~\ref{app:transience}. 
Our proof involves an estimate on radial Bessel processes (Lemma~\ref{lem::Bessel_transience}) and a monotone coupling argument (Section~\ref{subsec::monotone_coupling}). 
These techniques will also be useful to derive detailed return estimates for radial $\SLE_{\kappa}^{\mu}(2, \ldots, 2)$ processes, 
which is a key ingredient in proving large deviation principles for multiradial SLEs (see~\cite{Abuzaid-Peltola:Large_deviations_of_radial_SLE0, HPW:Multiradial_SLE_with_spiral_LDP, AHP:In_prep}).

\subsection{Boundary perturbation property}\label{subsec:bdry_perturbation}

As a consequence of Theorem~\ref{thm::multiradial_resampling}, we derive a boundary perturbation property (Proposition~\ref{prop::multiradial_bp})
for multiradial SLE with spiral, which describes the response of the measure to deformations of the domain. 

\begin{proposition}[Boundary perturbation]\label{prop::multiradial_bp}
Fix $\kappa\in (0,4], \, \mu\in\R, \, p\ge 1$, and a $p$-polygon $(\Omega; \bs{x}; z)$ with $\bs{x}=( x^1, \ldots, x^p)$ and $z\in\Omega$.
For $p$-radial $\SLE_{\kappa}^\mu$ curves $\bs \gamma = (\gamma^{1}, \ldots, \gamma^{p})$ in $(\Omega; \bs x; z)$, 
denote $\cup_{i=1}^p\gamma^{i} = \bs{\gamma}$.  
Suppose $K$ is a compact subset of $\overline{\Omega}$ such that $\Omega\setminus K$ is simply connected and contains $z$, 
and $K$ has a positive distance from $\{x^1, \ldots, x^p\}$. 
Then, the $p$-radial $\SLE_{\kappa}^\mu$ probability measure 
in the smaller polygon $(\Omega\setminus K; \bs{x}; z)$ is absolutely continuous with respect to that in $(\Omega; \bs{x}; z)$, 
with Radon-Nikodym derivative 
\begin{align}
\label{eqn::multiradialSLEbp}
\frac{\nradpartfn{p}{\mu}(\Omega; \bs{x};z)}{\nradpartfn{p}{\mu}(\Omega\setminus K; \bs{x}; z)}
\, \one\{ \bs{\gamma}\cap K=\emptyset \} \, \exp\Big(\frac{\mathfrak{c}}{2} \blm(\Omega;\bs{\gamma},K)\Big),
\end{align}
where $\blm(\Omega; \bs{\gamma}, K)$ is the Brownian loop measure~\eqref{eqn::blm_def} in $\Omega$ of those loops that intersect both $\bs{\gamma}$ and $K$. 
\end{proposition}

The boundary perturbation property was proved for standard radial SLE in~\cite[Proposition~5]{Jahangoshahi-Lawler:Multiple-paths_SLE_in_multiply_connected_domains}. We prove it for multiradial SLE with spiral in Section~\ref{subsec::bp}. The proof relies on the boundary perturbation property of $\SLE_{\kappa}^{\mu}(2, \ldots, 2)$ processes proved in Lemma~\ref{lem::radialSLEkapparho_bp}, 
the boundary perturbation property of half-watermelon SLEs proved in Proposition~\ref{prop::halfwatermelon_bp}, 
and the resampling property in Theorem~\ref{thm::multiradial_resampling}.

\subsection*{Acknowledgments}

We thank Yilin Wang and Lu Yang for helpful discussions.
Part of this research was performed while the authors participated in a program hosted by the Hausdorff Research Institute for Mathematics (HIM), which is supported by the Deutsche Forschungsgemeinschaft (DFG, German Research Foundation) under Germany's Excellence Strategy EXC-2047/1-390685813.

This material is part of a project that has received funding from the  European Research Council (ERC) under the European Union's Horizon 2020 research and innovation programme (101042460): 
ERC Starting grant ``Interplay of structures in conformal and universal random geometry'' (ISCoURaGe) 
and from the Academy of Finland grant number 340461 ``Conformal invariance in planar random geometry.''
E.P.~is also supported by 
the Academy of Finland Centre of Excellence Programme grant number 346315 ``Finnish centre of excellence in Randomness and STructures (FiRST)'' 
and by the Deutsche Forschungsgemeinschaft (DFG, German Research Foundation) under Germany's Excellence Strategy EXC-2047/1-390685813, 
as well as the DFG collaborative research centre ``The mathematics of emerging effects'' CRC-1060/211504053.

H.W. is supported by New Cornerstone Investigator Program 100001127. H.W. is partly affiliated at Yanqi Lake Beijing Institute of Mathematical Sciences and Applications, Beijing, China. 

\section{Multi-time parameter local martingale}
\label{sec::multitime_mart}
The main goal of this section is the derivation of the multi-time parameter local martingale~\eqref{eqn::multiradialSLE_mart} in Proposition~\ref{prop::multitime_mart_universal}, which validates Definition~\ref{def::multiradialSLEspiral}.
To this end, we first gather preliminaries in Section~\ref{subsec::pre_radialkapparhomu}. 
The multi-time parameter local martingale and its relation to Brownian loop measure is considered in Sections~\ref{subsec::multitime_mart_proof}--\ref{subsec::BLM_interpretation}. 
Sections~\ref{subsec::multiradial_spiral_basic}--\ref{subsec::coordinate_change} concern basic properties of multiradial $\SLE_\kappa^{\mu}$:
we derive the marginal law of a single curve, and show that when $\kappa \leq 4$, the process does not reach collision in finite time.

\smallskip

Throughout, we use the parameters~\eqref{eqn::universal_parameters}: 
\begin{align*}
\mathfrak{b} := \frac{6-\kappa}{2\kappa} = \Delta(\mathfrak{e}_{1,2}) 
, \qquad 
\tilde{\mathfrak{b}} := \frac{(6-\kappa)(\kappa-2)}{8\kappa} = 2 \Delta(\mathfrak{e}_{0,1/2}) 
, \qquad \textnormal{and} \qquad 
\mathfrak{c} := \frac{(6-\kappa)(3\kappa-8)}{2\kappa} .
\end{align*}

\subsection{Radial SLE with spiral and force points}
\label{subsec::pre_radialkapparhomu}

For $\kappa > 0$ and $\mu\in\R$, \emph{radial} $\SLE_{\kappa}^{\mu}$ in $(\U; \ee^{\ii\theta};0)$ is defined as the curve with Loewner driving function $\xi_t=\theta+\sqrt{\kappa}B_t+\mu t$ (as in~\eqref{eq:single_radial_Loewner_equation}), 
where $B$ is standard one-dimensional Brownian motion. 
When $\kappa\in (0,8)$,
it is almost surely generated by a continuous curve $\gamma$, which is transient: $\underset{t\to\infty}{\lim}\gamma_t=0$~\cite{Lawler:Continuity_of_radial_and_two-sided_radial_SLE_at_the_terminal_point, Field-Lawler:Escape_probability_and_transience_for_SLE}.

Radial $\SLE_{\kappa}^{\mu}$ in a polygon $(\Omega; x; z)$ is defined as the image under $\varphi^{-1}$ of radial $\SLE_{\kappa}^{\mu}$ in $(\U; \ee^{\ii\theta};0)$, 
where $\varphi^{-1} \colon \Omega \to \U$ is the conformal bijection such that $\varphi(x)=\ee^{\ii\theta}$ and $\varphi(z)=0$.
It has partition function 
\begin{align}\label{eqn:LGmu_pequalsone}
\begin{split}
\nradpartfn{1}{\mu}(\Omega; x; z)
:= \; & \exp\Big(\! -\frac{\mu}{\kappa} \arg \varphi'(z) \Big) \,  |\varphi'(z)|^{\tilde{\mathfrak{b}}-\frac{\mu^2}{2\kappa}} \;  |\varphi'(x)|^{\mathfrak{b}} \; 
\nradpartfn{1}{\mu}(\U; \ee^{\ii\theta};0) , 
\\
\qquad \textnormal{where} \qquad  
\nradpartfn{1}{\mu}(\U; u ;0) 
:= \; & \exp \big(\tfrac{\mu}{\kappa} \arg u \big) , \qquad u \in \partial \U .
\end{split}
\end{align} 
For convenience, we also write $\nradpartfn{1}{\mu}(\theta) := \nradpartfn{1}{\mu}(\U; \ee^{\ii\theta};0) = \exp (\tfrac{\mu}{\kappa} \theta )$ in the angular coordinate.

\bigskip

Let us now record the boundary perturbation property for radial $\SLE_{\kappa}^{\mu}$ when $\kappa\in (0,4]$.

\begin{lemma}[Boundary perturbation]\label{lem::radialSLE_bp}
Fix $\kappa\in (0,4], \, \mu\in\R$, and a polygon $(\Omega; x; z)$. 
Suppose $K$ is a compact subset of $\overline{\Omega}$ such that $\Omega\setminus K$ is simply connected and contains $z\in\Omega$ and $K$ has a positive distance from $x\in\partial\Omega$. 
Then, the radial $\SLE_{\kappa}^{\mu}$ probability measure 
in the smaller polygon $(\Omega\setminus K; x; z)$ is absolutely continuous with respect to that in $(\Omega; x; z)$, with Radon-Nikodym derivative 
\begin{align}\label{eqn::radialSLE_bp_RN}
\frac{\nradpartfn{1}{\mu}(\Omega;x;z)}{\nradpartfn{1}{\mu}(\Omega\setminus K; x; z)} 
\, \one\{\gamma\cap K=\emptyset \} \, \exp\Big(\frac{\mathfrak{c}}{2} \blm(\Omega;\gamma,K)\Big) ,
\end{align}
where $\blm(\Omega; \gamma, K)$ is the Brownian loop measure in $\Omega$ of those loops that intersect both $\gamma$ and $K$. 
\end{lemma}

\begin{proof}  
This can be proven by standard arguments; the case of $\mu=0$ is proved in~\cite[Proposition~5]{Jahangoshahi-Lawler:Multiple-paths_SLE_in_multiply_connected_domains}, and the case $\mu\in\R$ can be proved similarly.
We leave the details to the reader (see also Lemma~\ref{lem::radialSLEkapparho_bp}). 
\end{proof}

\paragraph*{Radial SLE with force points.}
Fix $p\ge 2$ and $\bs{\theta}=(\theta^1, \ldots, \theta^p)\in\LX_p$.  
For $\kappa > 0, \, \mu\in\R$, and $\bs{\rho} := (\rho_2, \ldots, \rho_p) \in\R^{p-1}$, we define 
\emph{radial} $\SLE_{\kappa}^{\mu}(\bs{\rho})$ in the reference polygon $(\U; \ee^{\ii\bs{\theta}}; 0)$ 
as the curve from $\ee^{\ii\theta^{1}}$ to $0$ with force points $(\ee^{\ii\theta^{2}}, \ldots, \ee^{\ii\theta^{p}})$ 
whose Loewner driving function $\xi_t$ (as in~\eqref{eq:single_radial_Loewner_equation}) solves the SDE 
\begin{align}\label{eqn::SLEkapparhomu_SDE}
\begin{cases}
\ud \xi_t = \sqrt{\kappa} \, \ud B_t 
- \displaystyle \sum_{j=2}^p \frac{\rho_j}{2} \cot \bigg(\frac{V_t^{j} - \xi_t}{2}\bigg) \ud t + \mu \, \ud t , 
\qquad \xi_0 = \theta^1; \\
\ud V_t^{j} =  \displaystyle  \cot \bigg(\frac{V_t^{j} - \xi_t}{2}\bigg) \ud t, \quad V_0^{j}=\theta^{j}, 
\qquad j\in \{2, \ldots, p\} .
\end{cases}
\end{align}
We define radial $\SLE_{\kappa}^{\mu}(\bs{\rho})$ in a more general $p$-polygon $(\Omega; \bs x;z)$
as the pushforward measure of $\SLE_{\kappa}^{\mu}(\bs{\rho})$ from $(\U; \ee^{\ii\bs{\theta}}; 0)$ 
by the map $\varphi^{-1}$, where 
$\varphi \colon \Omega \to \U$ is any conformal bijection such that $\varphi(z)=0$, 
and $\varphi(x^j) =: \ee^{\ii\theta^j}$, for each $j$.
Any two such mappings differ by a rotation, which does not affect the law.

\smallskip

The next Lemma~\ref{lem::sleakapprhomu_mart} compares the laws of radial $\SLE_{\kappa}^{\mu}(\bs{\rho})$ and radial $\SLE_{\kappa}$ (without spiral).
The comparison involves a martingale built from the partition function
of radial $\SLE_{\kappa}^{\mu}(\bs{\rho})$, defined as 
\begin{align}
\nradpartfn{1}{\mu; \bs{\rho}}(\bs\theta)
:= \; & \exp \bigg( \frac{\mu}{\kappa} \theta^1 + \frac{\mu}{\kappa} \sum_{j=2}^p \frac{\rho_j}{2} \theta^j \bigg) \, 
\prod_{\ell=2}^p |\ee^{\ii\theta^\ell} - \ee^{\ii\theta^1}|^{\frac{\rho_\ell}{\kappa}} \,  
\underset{2\leq j < \ell \leq p}{\prod} |\ee^{\ii\theta^\ell} - \ee^{\ii\theta^j}|^{\frac{\rho_\ell \rho_j}{2\kappa}} 
\label{eqn::pf_radialSLEkapparhomu}
\\
\notag
= \; &  2^{\frac{2}{\kappa} ( \bar{\rho} \, + \underset{2\leq i < j \leq p}{\sum}  \, \frac{\rho_i\rho_j}{2})} \, 
\exp \bigg( \frac{\mu}{\kappa}\sum_{j=1}^p \frac{\rho_j}{2} \theta^j \bigg) \, 
\prod_{\ell=2}^p \, \Big( \sin \Big(\frac{\theta^\ell - \theta^1}{2}\Big) \Big)^{\frac{\rho_\ell}{\kappa}}  \, 
\underset{2\leq j < \ell \leq p}{\prod} \, \Big( \sin \Big(\frac{\theta^\ell - \theta^j}{2}\Big) \Big)^{\frac{\rho_\ell \rho_j}{2\kappa}} ,
\\
\nradpartfn{1}{\mu; \bs{\rho}}(\Omega; \bs{x}; z)
:= \; & \exp\Big(\!-\frac{\mu}{2\kappa}(2+\bar{\rho}) \arg \varphi'(z) \Big) \, 
|\varphi'(z)|^{ \tilde{\mathfrak{b}}+\frac{\bar{\rho}(\bar{\rho}+4)}{8\kappa}-\frac{\mu^2}{2\kappa} } \,  
|\varphi'(x^1)|^{\mathfrak{b}} \, \Big( 
\prod_{j=2}^p |\varphi'(x^j)|^{\frac{\rho_j(\rho_j+4-\kappa)}{4 \kappa}} \Big) 
 \, \nradpartfn{1}{\mu; \bs{\rho}}(\bs\theta), \notag
\end{align}
where we write $\bar{\rho} := \sum_{j=2}^{p} \rho_j$, 
and set $\rho_1=2$ for convenience.

\begin{remark}
Recall~(\ref{eqn::universal_parameters},~\ref{eqn::general_charge_weight}). 
In terms of the Coulomb gas formalism of CFT, we have $\mathfrak{b} = \Delta(\mathfrak{e}_{1,2})$ and
\begin{align*}
\frac{\rho(\rho+4-\kappa)}{4 \kappa} = \Delta(\mathfrak{e}(\rho)), \qquad \textnormal{where }\; \mathfrak{e}(\rho) := \frac{\rho}{2 \sqrt{\kappa}} 
\end{align*}
is the Coulomb charge associated a boundary force point with strength $\rho$~\cite{Kytola:On_CFT_of_SLE_kappa_rho}. 
In fact, in accordance with the physics literature~\cite{Gruzberg:Stochastic_geometry_of_critical_curves_Schramm-Loewner_evolutions_and_conformal_field_theory, RBGW:Critical_curves_in_conformally_invariant_statistical_systems, BGR:Statistics_of_harmonic_measure_and_winding_of_critical_curves_from_conformal_field_theory}, 
the partition function $\nradpartfn{1}{\mu; \bs{\rho}}$ can be thought of as a correlation function of conformal fields of weight $\Delta(\mathfrak{e}_{1,2})$ at the boundary point $x^1$,
weights $\Delta(\mathfrak{e}(\rho_j))$ at the boundary points $x^j$ for $2 \leq j \leq p$, and 
at the bulk point $z$, a field of complex weights $(\Delta,\Delta^*)$, 
\begin{align*}
\Delta = \Delta(\mathfrak{e}_{\bar{\rho}},\mathfrak{m}_\mu) \; =  \;\; & \frac{8-\kappa}{16} + \frac{(\bar{\rho} -2) (\bar{\rho} +6)}{16\kappa} - \frac{\mu^2}{4 \kappa} + \ii \, \mu \frac{(2+\bar{\rho})}{4\kappa} ,
\\
\qquad \textnormal{where } \;
\mathfrak{e}_{\bar{\rho}} 
:= \; & \frac{\kappa - 2 + \bar{\rho}}{4 \sqrt{\kappa}}
\qquad \textnormal{and} \qquad 
\mathfrak{m}_\mu := \ii \, \frac{\mu }{2 \sqrt{\kappa}} .
\end{align*}
As an important example, let us consider the case where $\rho_j=2$ for all $j$.
Then, Eq.~\eqref{eqn::pf_radialSLEkapparhomu} reads
\begin{align*}
\nradpartfn{1}{\mu; 2,2,\ldots,2}(\Omega; \bs{x}; z)
= \; & \exp\Big(\!- \frac{\mu}{\kappa} p \arg \varphi'(z)\Big) \, 
|\varphi'(z)|^{\tilde{\mathfrak{b}} + \frac{p^2-1}{2\kappa}-\frac{\mu^2}{2\kappa}} \,  
 \Big( |\varphi'(x^1)| \, 
\prod_{j=2}^p |\varphi'(x^j)|\Big)^{\mathfrak{b}} \,   
\nradpartfn{p}{\mu}(\bs{\theta}),
\end{align*}
which coincides with the partition function $\smash{\nradpartfn{p}{\mu}(\Omega; \bs x; z)}$ defined in~\eqref{eqn::Gmu_cov}.
\end{remark}

\begin{lemma}\label{lem::sleakapprhomu_mart}
The law of radial $\SLE_{\kappa}^{\mu}(\bs{\rho})$ in $(\U; \ee^{\ii\bs{\theta}}; 0)$ is the same as the law of radial $\SLE_{\kappa}$ in $(\U; \ee^{\ii\theta^{1}}; 0)$ 
tilted by the following local martingale, up to the first time $\ee^{\ii\theta^{p}}$ or $\ee^{\ii\theta^{2}}$ is disconnected from the origin: 
\begin{align}\label{eqn::slekapparhomu_mart}
M_t(\nradpartfn{1}{\mu; \bs{\rho}}) 
= \; & (g'_t(0))^{\frac{\bar{\rho}(\bar{\rho}+4)}{8\kappa}-\frac{\mu^2}{2\kappa}} \, 
\Big( 
\prod_{j=2}^p (\covmap'_t (\theta^{j}))^{\frac{\rho_j(\rho_j+4-\kappa)}{4 \kappa}} \Big) \, 
\nradpartfn{1}{\mu; \bs{\rho}}(\xi_t, \covmap_t(\theta^{2}), \ldots, \covmap_t(\theta^{p}) ) ,
\end{align}
where $\bs{\rho}=(\rho_2, \ldots, \rho_p)$ and $\bar{\rho}:=\sum_{j=2}^{p} \rho_j$. 
\end{lemma}

\begin{proof}
When $\mu=0$, a direct calculation shows that, up to the first time when $\ee^{\ii\theta^{p}}$ or $\ee^{\ii\theta^{2}}$ is disconnected from the origin, 
the law $\PP$ of radial $\SLE_{\kappa}(\bs{\rho})$ in $(\U; \ee^{\ii\bs{\theta}}; 0)$ 
is the same as the law of 
radial $\SLE_{\kappa}$ in $(\U; \ee^{\ii\theta^{1}}; 0)$ tilted by the local martingale
$M(\nradpartfn{1}{0; \bs{\rho}})$~\cite{Schramm-Wilson:SLE_coordinate_changes}, 
which we denote as $\PPnospiralrho$.
The addition of the spiraling parameter $\mu \in \R$ follows easily via an application of Girsanov's theorem as follows. 
Under $\PPnospiralrho$, we have 
\begin{align*}
\ud \xi_t 
= \; & \sqrt{\kappa} \, \ud B_t + \kappa \, \partial_1 \log \nradpartfn{1}{0; \bs{\rho}}(\xi_t, \covmap_t(\theta^{2}), \ldots, \covmap_t(\theta^{p}) ) \, \ud t 
\\
= \; &  \sqrt{\kappa} \, \ud B_t - \sum_{j=2}^p \frac{\rho_j}{2}\cot \bigg(\frac{\covmap_t(\theta^{j}) - \xi_t}{2}\bigg) \ud t ,
&& [\textnormal{by~\eqref{eqn::pf_radialSLEkapparhomu}}]
\end{align*}
and by the radial Loewner equation~\eqref{eq:single_radial_Loewner_equation_cov} the process
\begin{align*}
L_t := \frac{\mu}{\kappa} \Big(\xi_t+\sum_{j=2}^p \frac{\rho_j}{2} \covmap_t(\theta^{j}) \Big)
\end{align*}
satisfies $\ud L_t=\frac{\mu}{\sqrt{\kappa}}\ud B_t$, 
so $\langle L, L\rangle_t = \frac{\mu^2}{\kappa} t = \frac{\mu^2}{\kappa} \log g'_t(0)$.
Denote by $\PPspiralrho$ the measure $\PPnospiralrho$ tilted by 
\begin{align*}
R_t:=\frac{M_t(\nradpartfn{1}{\mu; \bs{\rho}})}{M_t(\nradpartfn{1}{0; \bs{\rho}})} 
= \exp \Big(L_t-\frac{1}{2}\langle L, L\rangle_t\Big). 
\end{align*}
Girsanov's theorem implies that $\tilde{B}_t=B_t-\frac{\mu}{\sqrt{\kappa}}t$ is standard Brownian motion under $\PPspiralrho$,
and we have 
\begin{align*}
\ud \xi_t=\sqrt{\kappa} \, \ud \tilde{B}_t - \sum_{j=2}^p \frac{\rho_j}{2}\cot\bigg(\frac{\covmap_t(\theta^{j}) - \xi_t}{2}\bigg)\ud t+\mu \, \ud t .
\end{align*}
Therefore, $\PPspiralrho$ is indeed the same as the law of radial $\SLE_{\kappa}^{\mu}(\bs{\rho})$ process (cf.~Eq.~\eqref{eqn::SLEkapparhomu_SDE}). 
\end{proof}

We see from Lemma~\ref{lem::sleakapprhomu_mart} that 
radial $\SLE_{\kappa}^{\mu}(\bs\rho)$ is absolutely continuous with respect to radial $\SLE_{\kappa}$, 
so, as the latter is almost surely generated by a continuous curve, radial $\SLE_{\kappa}^{\mu}(\bs\rho)$ is also almost surely generated by continuous curve --- 
up to the first time when $\ee^{\ii\theta^{p}}$ or $\ee^{\ii\theta^{2}}$is disconnected from the origin. 
However, this does \emph{not} imply transience for radial $\SLE_{\kappa}^{\mu}(\bs\rho)$, 
because the absolute continuity only holds up to a possibly finite stopping time.
Generalizing~\cite{Lawler:Continuity_of_radial_and_two-sided_radial_SLE_at_the_terminal_point}, 
we prove in Appendix~\ref{app:transience} that, when $\kappa\in (0,4]$ and $\rho_j\ge 0$ for all $j$, 
radial $\SLE_{\kappa}^{\mu}(\bs\rho)$ in $(\U; \ee^{\ii\bs{\theta}}; 0)$ is almost surely transient: that is,  $\gamma_t \to 0$ as $t\to\infty$.

\subsection{Local martingales for multiradial SLE}
\label{subsec::multitime_mart_proof}

We use for multiradial Loewner evolutions the same notation as in Sections~\ref{subsec::intro_preli}--\ref{subsec::intro_multiradialSLE} --- see also Figure~\ref{fig::multitime}. 

\begin{proposition}\label{prop::multitime_mart_universal}
Fix $\kappa > 0, \, \aleph\in\R, \, p\ge 2$, and $\bs{\theta}=(\theta^1, \ldots, \theta^p)\in\LX_p$.
For each $j\in \{1,\ldots, p\}$, let $\gamma^{j}$ be radial $\SLE_{\kappa}$ in $(\U; \ee^{\ii\theta^j}; 0)$ 
and let $\mathsf{P}_p$ 
be the probability measure on
$\bs{\gamma}=(\gamma^{1}, \ldots, \gamma^{p})$ 
under which the curves  
are independent. 
We parameterize $\bs{\gamma}$ by $p$-time parameter $\bs{t}$, and let $\bs{\mslitdriv}_{\bs{t}} = (\mslitdriv_{\bs{t}}^{1},\ldots,\mslitdriv_{\bs{t}}^{p})$ be the multi-slit driving function as in~\eqref{eqn::Ntuple_driving}.
For a function $\LG \colon \LX_p\to \R_{>0}$, define 
\begin{align}\label{eqn::multitime_mart_universal}
\begin{split}
M_{\bs{t}}(\LG)
:= \; & \one_{\LE_{\emptyset}(\bs{\gamma}_{\bs{t}})} \, \exp\bigg(\frac{\mathfrak{c}}{2}\blm_{\bs{t}} -\tilde{\mathfrak{b}}\sum_{j=1}^{p}t_j\bigg) 
\; (g_{\bs{t}}'(0))^{\tilde{\mathfrak{b}}-\aleph} 
\Big(\prod_{j=1}^{p} \covmap_{\bs{t},j}'(\xi_{t_j}^{j}) \Big)^{\mathfrak{b}}
\LG(\bs{\mslitdriv}_{\bs{t}})
\\
= \; & \one_{\LE_{\emptyset}(\bs{\gamma}_{\bs{t}})} \, \exp\Big(\frac{\mathfrak{c}}{2}\blm_{\bs{t}}\Big) 
\; (g_{\bs{t}}'(0))^{-\aleph-(p-1)\tilde{\mathfrak{b}}} 
\Big(\prod_{j=1}^{p} \covmap_{\bs{t},j}'(\xi_{t_j}^{j}) \Big)^{\mathfrak{b}}
\Big(\prod_{j=1}^{p} g_{\bs{t},j}'(0) \Big)^{\tilde{\mathfrak{b}}}
\LG(\bs{\mslitdriv}_{\bs{t}}) ,
\end{split}
\end{align}
where $\LE_{\emptyset}(\bs{\gamma}_{\bs{t}}) = \{ \gamma_{[0,t_j]}^{j} \cap \gamma_{[0,t_i]}^{i}=\emptyset, \, \forall i\neq j\}$ is the event that different curves are disjoint, and 
$\blm_{\bs{t}}$ is the unique potential solving the exact differential equation~\eqref{eqn::mt_def} \textnormal{(}see Lemma~\ref{lem::blm_exact}\textnormal{):} 
\begin{align*}
\ud  \blm_{\bs{t}} 
= \sum_{j=1}^{p} \partial_{t_j} \blm_{(t_1,\ldots,t_p)} \ud t_j
= \sum_{j=1}^{p} \Big( \! -\frac{1}{3} (\LS \covmap_{\bs{t},j})(\xi_{t_j}^{j})+\frac{1}{6} \big(1- (\covmap'_{\bs{t},j}(\xi_{t_j}^{j}))^2 \big) \Big) \ud t_j, 
\qquad \blm_{\bs{0}} = 0 .
\end{align*}
Then, the process $M(\LG)$ is a $p$-time parameter local martingale under $\mathsf{P}_p$ if and only if $\LG$ is smooth and satisfies the following system of {\bf\emph{radial BPZ equations}}\textnormal{:}
for each $\bs{\vartheta} := (\vartheta^1, \ldots, \vartheta^p)\in\LX_p$, 
\begin{align}\label{eqn::radialBPZ}
\frac{\kappa}{2} \frac{\partial_j^2 \LG(\bs{\vartheta})}{\LG(\bs{\vartheta})} 
+ \underset{1\leq i\neq j \leq p}{\sum} \, \bigg( \cot \bigg(\frac{\vartheta^{i}-\vartheta^{j}}{2}\bigg)
\frac{\partial_{i} \LG(\bs{\vartheta})}{\LG(\bs{\vartheta})} \, - \, \frac{\mathfrak{b}/2}{ \sin^2 \big(\frac{\vartheta^{i}-\vartheta^{j}}{2}\big)}\bigg)=\aleph, \qquad \textnormal{for all }j\in \{1,\ldots,p\} .
\end{align}
\textnormal{(}The function $\sin^2 \big(\frac{\vartheta^{i}-\vartheta^{j}}{2}\big)$ is a constant multiple of the boundary Poisson kernel in the disc, Eq.~\textnormal{\eqref{eqn::Poissonkernel_disc}}.\textnormal{)}
\end{proposition}

\smallskip

Proposition~\ref{prop::multitime_mart_universal} is a generalization of various special cases appearing in earlier literature:
the case where $\smash{\LG=\nradpartfn{p}{0}}$ as defined in~\eqref{eq:multiradial_partition_function} appeared in~\cite{Healey-Lawler:N_sided_radial_SLE}; 
the case where $\LG$ equals a pure partition function appeared in~\cite{FWW:Multiple_Ising_interfaces_in_annulus_and_2N-sided_radial_SLE}; 
and the case where $\smash{\LG=\nradpartfn{2}{\mu}}$ as defined in~\eqref{eqn:spiralpartitionfunction} with $p=2$ appeared in~\cite{KWW:Commutation_relations_for_two-sided_radial_SLE}. 
For later use, we derive the general version here.  
Moreover, in Section~\ref{subsec::coordinate_change} we address the case of a general endpoint for radial $\SLE_\kappa^\mu$ curves. 
Radial SLEs in $\U$ are not invariant under conformal self-maps of the disc $\U$ moving the origin. 
Instead, one has to take into account the appropriate coordinate change~\cite{Schramm-Wilson:SLE_coordinate_changes}. 
The coordinate change gives a non-trivial contribution to the spiraling rate, as in Equations~(\ref{eqn::pf_radialSLEkapparhomu},~\ref{eqn::multitime_mart_nradial_spiral_z}).

\begin{remark}
Note that $\blm_{\bs{t}} \equiv 0$ when $p=1$ or $\bs{t} = (t,0,\ldots,0)$ (meaning that we only grow one curve): 
in both cases, Eq.~\eqref{eqn::multitime_mart_universal} only contains the $j=1$ term, which equals zero:
namely $g_{\bs{t}} = g_{t}^{1} =: g_{t}$ is just the mapping-out function of the first curve normalized as $\log g_{t}'(0) = t$, and $\covmap_{\bs{t},1}$ is just the identity map. 
\end{remark}

\begin{proof}[Proof of Proposition~\ref{prop::multitime_mart_universal}]
The equality of the first line and the second line in~\eqref{eqn::multitime_mart_universal} follows by definitions of the various conformal maps. To prove the assertion for the second line in~\eqref{eqn::multitime_mart_universal}, let us first compute the variations of the terms appearing in it in the capacity parameterization.
By differentiating~\eqref{eq:LE_multislitg} and using the identity $g_{\bs{t},j} = g_{\bs{t}} \circ (g_{t_j}^{j})^{-1}$ and the radial Loewner equation~\eqref{eq:single_radial_Loewner_equation} for $g_{t_j}^{j}$, we find that
\begin{align}
\label{eqn::multitimemart_aux1}
\frac{\ud g'_{\bs{t}}(0)}{g'_{\bs{t}}(0)} = \sum_{j=1}^p (\covmap_{\bs{t},j}' (\xi_{t_j}^{j}))^2 \, \ud t_j
\qquad \textnormal{and} \qquad 
\frac{\ud g_{\bs{t},j}'(0)}{g_{\bs{t},j}'(0)} = \big( (\covmap_{\bs{t},j}' (\xi_{t_j}^{j}))^2 -1\big) \, \ud t_j 
+ \sum_{i \neq j} (\covmap_{\bs{t},i}' (\xi_{t_i}^{i}))^2 \ud t_i .
\end{align}
Using It\^{o}'s formula, we thus find
\begin{align}
\label{eqn::multitimemart_aux2}
\ud \mslitdriv_{\bs{t}}^{j} = \; & \covmap_{\bs{t},j}' (\xi_{t_j}^{j}) \, \ud \xi_{t_j}^{j} 
- \kappa \mathfrak{b} \, \covmap_{\bs{t},j}'' (\xi_{t_j}^{j}) \, \ud t_j  
+ \sum_{i \neq j} \cot \bigg(\frac{\mslitdriv_{\bs{t}}^{j}-\mslitdriv_{\bs{t}}^i}{2}\bigg) (\covmap_{\bs{t},i}' (\xi_{t_i}^{i}))^2 \, \ud t_i ,
\\
\label{eqn::multitimemart_aux3}
\begin{split}
\textnormal{and} \qquad 
\frac{\ud \covmap_{\bs{t},j}'(\xi_{t_j}^{j})}{\covmap_{\bs{t},j}'(\xi_{t_j}^{j})} 
= \; & \frac{ \covmap_{\bs{t},j}''(\xi_{t_j}^{j})}{\covmap_{\bs{t},j}'(\xi_{t_j}^{j})} \, \ud \xi_{t_j}^{j} 
- \frac{1}{2} \sum_{i \neq j} \frac{(\covmap_{\bs{t},i}' (\xi_{t_i}^{i}))^2}{\sin^2 \big(\frac{\mslitdriv_{\bs{t}}^{j}-\mslitdriv_{\bs{t}}^i}{2}\big)}  \, \ud t_i 
\\
\; & + \bigg( \bigg( \frac{3\kappa-8}{6} \bigg) \frac{ \covmap_{\bs{t},j}'''(\xi_{t_j}^{j})}{\covmap_{\bs{t},j}'(\xi_{t_j}^{j})} + \frac{1}{2} \bigg( \frac{ \covmap_{\bs{t},j}''(\xi_{t_j}^{j})}{\covmap_{\bs{t},j}'(\xi_{t_j}^{j})}\bigg)^2 - \frac{1}{6}\big( (\covmap_{\bs{t},j}' (\xi_{t_j}^{j}))^2 - 1\big)\bigg) \ud t_j.
\end{split}
\end{align}

We first assume that $\LG\in C^2(\LX_p)$ and apply It\^{o}'s formula to $M_{\bs{t}}(\LG)$ to obtain
\begin{align*}
\frac{\ud M_{\bs{t}}(\LG)}{M_{\bs{t}}(\LG)}
= \; & \big(\! -\aleph-(p-1)\tilde{\mathfrak{b}} \big) \frac{\ud g'_{\bs{t}}(0)}{g'_{\bs{t}}(0)} 
+ \sum_{j=1}^p \bigg( \mathfrak{b} \frac{\ud \covmap_{\bs{t},j}'(\xi_{t_j}^{j})}{\covmap_{\bs{t},j}'(\xi_{t_j}^{j})} + \frac{\kappa \mathfrak{b}(\mathfrak{b}-1)}{2} \bigg( \frac{ \covmap_{\bs{t},j}''(\xi_{t_j}^{j})}{\covmap_{\bs{t},j}'(\xi_{t_j}^{j})}\bigg)^2 \ud t_j \bigg)
\\
 \; & + \tilde{\mathfrak{b}} \sum_{j=1}^p \frac{\ud g_{\bs{t},j}'(0)}{g_{\bs{t},j}'(0)} + \frac{\mathfrak{c}}{2} \ud \blm_{\bs{t}}+ \sum_{j=1}^p \bigg( \frac{\partial_j \LG}{\LG} \, \ud \mslitdriv_{\bs{t}}^{j} 
 + \frac{\kappa}{2} \frac{\partial_j^2 \LG}{\LG} (\covmap_{\bs{t},j}' (\xi_{t_j}^{j}))^2 \, \ud t_j \bigg) \\
 \; & + \kappa \mathfrak{b} \sum_{j=1}^p \frac{\partial_j \LG}{\LG} \covmap_{\bs{t},j}''(\xi_{t_j}^{j}) \, \ud t_j. 
\end{align*}
Using~\eqref{eqn::mt_def} and~\eqref{eqn::multitimemart_aux1}--\eqref{eqn::multitimemart_aux3}, we obtain
\begin{align*}
\frac{\ud M_{\bs{t}}(\LG)}{M_{\bs{t}}(\LG)}
= \; & \sum_{j=1}^p \bigg( \frac{\partial_j \LG(\bs{\mslitdriv}_{\bs{t}})}{\LG(\bs{\mslitdriv}_{\bs{t}})} \covmap_{\bs{t},j}'(\xi_{t_j}^{j}) + \mathfrak{b} \frac{\covmap_{\bs{t},j}''(\xi_{t_j}^{j})}{\covmap_{\bs{t},j}'(\xi_{t_j}^{j})} \bigg) \ud \xi_{t_j}^{j}  \\
\; & +\sum_{j=1}^p (\covmap_{\bs{t},j}' (\xi_{t_j}^{j}))^2 \bigg( \frac{\kappa}{2} \frac{\partial_j^2 \LG(\bs{\mslitdriv}_{\bs{t}})}{\LG(\bs{\mslitdriv}_{\bs{t}})} 
+ \sum_{i\neq j} \bigg( \cot \bigg(\frac{\mslitdriv_{\bs{t}}^i-\mslitdriv_{\bs{t}}^j}{2}\bigg) 
\frac{\partial_{i} \LG(\bs{\mslitdriv}_{\bs{t}})}{\LG(\bs{\mslitdriv}_{\bs{t}})}- \frac{ \mathfrak{b}/2}{\sin^2 \big(\frac{\mslitdriv_{\bs{t}}^i-\mslitdriv_{\bs{t}}^j}{2}\big)} \bigg) - \aleph\bigg) \ud t_j. 
\end{align*}
From this expression, we see that $M(\LG)$ is $p$-time parameter local martingale if and only if all the drift terms vanish, which happens precisely when $\LG$ satisfies the system~\eqref{eqn::radialBPZ} of radial BPZ equations. 

To lift the $C^2$-assumption on $\LG$, 
one can prove (like~\cite[Proposition~A.5]{FLPW:Multiple_SLEs_Coulomb_gas_integrals_and_pure_partition_functions}) that
$M(\LG)$ is $p$-time parameter local martingale if and only if $\LG$ satisfies the system~\eqref{eqn::radialBPZ} of radial BPZ equations as a weak solution; 
and since the differential operators in~\eqref{eqn::radialBPZ} are hypoelliptic, 
weak solutions are strong solutions; in particular smooth (see~\cite[Appendix~A]{FLPW:Multiple_SLEs_Coulomb_gas_integrals_and_pure_partition_functions} for the chordal case). 
This finishes the proof. 
\end{proof}

Next, we address the quantity $\blm_{\bs{t}}$ in Eq.~\eqref{eqn::mt_def}. 
Lemma~\ref{lem::blm} gives a formula for it in terms of the Brownian loop measure. 
First, the fact that it is well-defined hinges on the following integrability result.

\begin{lemma} \label{lem::blm_exact}
The differential equation~\eqref{eqn::mt_def} is exact and has a unique solution $\blm_{\bs{t}}$ such that $\blm_{\bs{0}}=0$.
\end{lemma}
\begin{proof}
We need to show that
\begin{align}\label{eq::commutation}
\partial_{t_i} N_{(t_1,\ldots,t_p)}^{j}
= \partial_{t_j} N_{(t_1,\ldots,t_p)}^{i} , \qquad  \textnormal{for all } 1 \leq i\neq j \leq p ,
\end{align}
where we denote the terms in~\eqref{eqn::mt_def} by $N_{\bs{t}}^{j} = N_{(t_1,\ldots,t_p)}^{j}$,
\begin{align*}
\ud  \blm_{\bs{t}} 
= \sum_{j=1}^{p} \underbrace{\Big( \! -\frac{1}{3} (\LS \covmap_{\bs{t},j})(\xi_{t_j}^{j})+\frac{1}{6} \big(1- (\covmap'_{\bs{t},j}(\xi_{t_j}^{j}))^2 \big) \Big)}_{=: \; N_{\bs{t}}^{j}} \ud t_j, 
\qquad \blm_{\bs{0}} = 0 .
\end{align*}
If the balance equations~\eqref{eq::commutation} are satisfied, then there exists a potential $\blm_{\bs{t}} = \blm_{(t_1,\ldots,t_p)}$ such that 
\begin{align*}
\partial_{t_j} \blm_{(t_1,\ldots,t_p)} = N_{(t_1,\ldots,t_p)}^{j} , \qquad  \textnormal{for all } 1 \leq j \leq p .
\end{align*}
$\blm_{\bs{t}}$ is only determined up to an additive constant, which we fix by choosing $\blm_{\bs{0}} = 0$.
A straightforward but tedious computation using~\eqref{eq:LE_multislith} and $\covmap_{\bs{t}} = \covmap_{\bs{t},j} \circ \covmap_{t_j}^{j}$ (involving many cancellations) shows that 
\begin{align*}
\partial_{t_i} N_{(t_1,\ldots,t_p)}^{j} 
= \frac{a_{\bs{t}}^{i} a_{\bs{t}}^{j}}{4 \, \sin^4 \big(\frac{\mslitdriv_{\bs{t}}^{j} - \mslitdriv_{\bs{t}}^{i}}{2}\big)} ,
\qquad \textnormal{where} \qquad a_{\bs{t}}^{j} := \partial_{t_j} \log g_{\bs{t}}'(0) = (\covmap_{\bs{t},j}' (\xi_{t_j}^{j}))^2 , 
\end{align*}
which is symmetric with respect to the exchange $i \leftrightarrow j$, confirming~\eqref{eq::commutation}. 
\end{proof}

\subsection{Brownian loop measure interpretation}
\label{subsec::BLM_interpretation}

\emph{Brownian loop measure} $\blm^\mathrm{loop}$ is a $\sigma$-finite measure on planar unrooted Brownian loops
--- see~\cite{Lawler-Werner:The_Brownian_loop_soup} for its definition and properties
(and~\cite{Lawler:ICM_Conformally_invariant_loop_measures} for a survey). 
While the total mass of $\blm^\mathrm{loop}$ is infinite, the mass on macroscopic loops is finite: 
in particular, if $\Omega$ is a domain and $K_1, K_2 \subset \overline{\Omega}$ are two disjoint compact subsets, 
then the total mass $\blm(\Omega; K_1, K_2)$ of Brownian loops that stay in $\Omega$ and intersect both $K_1$ and $K_2$ is finite.
In general, for $n\ge 2$ non-empty disjoint compact subsets $K_1, \ldots, K_n$ of $\overline{\Omega}$, we denote
\begin{align}\label{eqn::blm_def}
\blm(\Omega; K_1, \ldots, K_n) := \sum_{j=2}^n \blm^{\mathrm{loop}} \big[ \ell\subset\Omega \colon \ell\cap K_i\neq \emptyset\textnormal{ for at least }j \textnormal{ of the }i\in\{1, \ldots, n\} \big]
\; \in \; (0,\infty) .
\end{align}
See~\cite{Lawler:Partition_functions_loop_measure_and_versions_of_SLE} and~\cite{Peltola-Wang:LDP_of_multichordal_SLE_real_rational_functions_and_determinants_of_Laplacians} for more properties and~\cite{Dubedat:Euler_integrals_for_commuting_SLEs, Dubedat:Commutation_relations_for_SLE, Kozdron-Lawler:Configurational_measure_on_mutually_avoiding_SLEs, Peltola-Wu:Global_and_local_multiple_SLEs_and_connection_probabilities_for_level_lines_of_GFF} 
for alternative forms.

\begin{lemma}[cf.~\cite{Healey-Lawler:N_sided_radial_SLE}]\label{lem::blm}
The solution to~\eqref{eqn::mt_def} 
can be described in terms of Brownian loop measure as
\begin{align}\label{eqn::mt_blm}
\blm_{\bs{t}} = \blm \big(\U;\gamma_{[0,t_1]}^{1},\ldots,\gamma_{[0,t_p]}^{p} \big).
\end{align}
Consequently, $\blm_{\bs{t}}$ is finite as long as $\gamma_{[0,t_1]}^{1},\ldots,\gamma_{[0,t_p]}^{p}$ are disjoint. 
\end{lemma}

\begin{proof}
Denote by $\tilde{\blm}_{\bs{t}}$ the right-hand side of~\eqref{eqn::mt_blm}. 
Note that both sides satisfy the same initial value $\tilde{\blm}_{\bs{t}} = \blm_{\bs{t}} = 0$.
Hence, it suffices to calculate $\partial_{t_j} \tilde{\blm}_{\bs{t}}$ for each $j$. 
From~\eqref{eqn::blm_def}, we obtain
\begin{align*}
\; &\blm \big( \U;\gamma_{[0,t_1]}^{1},\ldots,\gamma_{[0,t_p]}^{p} \big) \\
= \; & \sum_{j=1}^{p-1} \blm^{\mathrm{loop}}\big[ \ell\subset\U \colon \ell\cap \gamma_{[0,t_1]}^{1}\neq \emptyset , \textnormal{ and } \ell\cap \gamma_{[0,t_i]}^{i}\neq \emptyset\textnormal{ for at least }j \textnormal{ of the }i\in\{2, \ldots, p\}\big] \\
\; &+ \sum_{j=2}^{p-1} \blm^{\mathrm{loop}}\big[\ell \subset \U \setminus \gamma_{[0,t_1]}^{1} \colon \ell\cap \gamma_{[0,t_i]}^{i} \neq \emptyset\textnormal{ for at least }j \textnormal{ of the }i\in\{2, \ldots, p\}\big] \\
=\; &\blm \big( \U;\gamma_{[0,t_1]}^{1},\bs{\gamma}_{\bs{t}}\setminus \gamma_{[0,t_1]}^{1}\big)
\, + \, \blm \big( \U;\gamma_{[0,t_2]}^{2},\ldots,\gamma_{[0,t_p]}^{p} \big) .
\end{align*}
Using the relation between Brownian loop measure and the Schwarzian derivative $\LS$
(see, e.g., the last two equations at the bottom of~\cite[Page~34]{Jahangoshahi-Lawler:Multiple-paths_SLE_in_multiply_connected_domains}), we obtain
\begin{align*}
\partial_{t_1} \tilde{\blm}_{\bs{t}} 
= \partial_{t_1} \blm \big(\U;\gamma_{[0,t_1]}^{1},\bs{\gamma}_{\bs{t}}\setminus \gamma_{[0,t_1]}^{1}\big) 
= -\frac{1}{3}\LS \covmap_{\bs{t},1}(\xi_{t_1}^{1}) + \frac{1}{6}\big(1 - (\covmap'_{\bs{t},1}(\xi_{t_1}^{1}))^2 \big).
\end{align*}
The variations $\partial_{t_j} \tilde{\blm}_{\bs{t}}$ for $2\le j\le p$ can be calculated in the same way. 
Lemma~\ref{lem::blm_exact} now yields~\eqref{eqn::mt_blm}.
\end{proof}

\subsection{Multiradial SLE with spiral: basic properties}
\label{subsec::multiradial_spiral_basic}

Now, we are ready to consider the process from Definition~\ref{def::multiradialSLEspiral}.

\begin{corollary}\label{cor::SLEspiral_mart}
Assume the same notation as in Proposition~\ref{prop::multitime_mart_universal}. Fix $\mu\in\R$ and recall that $\nradpartfn{p}{\mu}$ is defined in~\eqref{eqn:spiralpartitionfunction}. 
The process $M(\nradpartfn{p}{\mu})$ in~\eqref{eqn::multiradialSLE_mart} is $p$-time parameter local martingale under $\mathsf{P}_p$. 
\end{corollary}

\begin{proof}
This is a consequence of Proposition~\ref{prop::multitime_mart_universal}: 
$\nradpartfn{p}{\mu}$ in~\eqref{eqn:spiralpartitionfunction} satisfies the radial BPZ equations~\eqref{eqn::radialBPZ}:
\begin{align}\label{eqn::radialBPZ_spiral}
\frac{\kappa}{2} \frac{\partial_j^2 \nradpartfn{p}{\mu}(\bs{\vartheta})}{\nradpartfn{p}{\mu}(\bs{\vartheta})} 
+ \underset{1\leq i\neq j \leq p}{\sum} \, \bigg( \cot \bigg(\frac{\vartheta^{i}-\vartheta^{j}}{2}\bigg)
\frac{\partial_{i} \nradpartfn{p}{\mu}(\bs{\vartheta})}{\nradpartfn{p}{\mu}(\bs{\vartheta})} \, - \, \frac{\mathfrak{b}/2}{ \sin^2 \big(\frac{\vartheta^{i}-\vartheta^{j}}{2}\big)}\bigg)=\frac{\mu^2-(p^2-1)}{2\kappa} ,
\end{align}
for all $j$ (here, $\aleph = \frac{\mu^2}{2\kappa} - \frac{p^2-1}{2\kappa} = 2 \Delta(\mathfrak{e}_{0,1/2}) - \Delta(\mathfrak{e}_{0,p/2},\mathfrak{m}_\mu) - \Delta^*(\mathfrak{e}_{0,p/2},\mathfrak{m}_\mu)$ in terms of~\eqref{eqn::spiral_conformal_weight}).
\end{proof}

As in Definition~\ref{def::multiradialSLEspiral}, we denote by $\PPspiral{p}$ the measure obtained by tilting $\mathsf{P}_p$ by $\smash{M_{\bs{t}}(\nradpartfn{p}{\mu})}$, and call it $p$-radial $\SLE_{\kappa}^{\mu}$ in $(\U; \ee^{\ii\bs{\theta}}; 0)$. 
Recall that the collision time of the $p$-tuple of curves $\bs{\gamma}_{\bs{t}}$ is the first time when any two curves in the tuple meet. 
It is clear from the definition that $p$-radial $\SLE_{\kappa}^{\mu}$ is well-defined up to the collision time. 
We now derive the marginal law of a single curve under $\PPspiral{p}$ (see Lemma~\ref{lem::multiradial_marginal}), 
and show that the process almost surely never reaches the collision time when $\kappa\le 4$ (see Lemma~\ref{lem::multiradial_collision}). 

\begin{lemma}\label{lem::multiradial_marginal}
Fix $\kappa > 0, \, \mu\in\R, \, p\ge 2$, and $\bs{\theta}=(\theta^1, \ldots, \theta^p)\in\LX_p$.
Suppose $(\gamma^{1}, \ldots, \gamma^{p})$ is $p$-radial $\SLE_{\kappa}^{\mu}$ in $(\U; \ee^{\ii\bs{\theta}}; 0)$, 
and let $g_{t}$ be the mapping-out function of $\gamma^1$ solving~\eqref{eq:single_radial_Loewner_equation}  with driving function $\xi_t$ and $\covmap_{t}$ the associated covering map solving~\eqref{eq:single_radial_Loewner_equation_cov}.
Then, the law of $\gamma^{1}$ is the same as radial $\SLE_{\kappa}$ in $(\U; \ee^{\ii\theta^1}; 0)$ tilted by the following local martingale, 
up to the first time $\ee^{\ii\theta^p}$ or $\ee^{\ii\theta^2}$ is disconnected from the origin: 
\begin{align}\label{eqn::multiradialSLE_mart_cor}
M_t(\nradpartfn{p}{\mu}) = (g_t'(0))^{\frac{p^2-1}{2\kappa}-\frac{\mu^2}{2\kappa}}
\Big(\prod_{j=2}^{p} \covmap_t'(\theta^{j}) \Big)^{\mathfrak{b}}\nradpartfn{p}{\mu}(\xi_t, \covmap_t(\theta^2), \ldots, \covmap_t(\theta^p)). 
\end{align}
Consequently, $\gamma^{1}$ has the same law as radial $\SLE_{\kappa}^{\mu}(2,\ldots, 2)$ in $(\U; \ee^{\ii\bs{\theta}}; 0)$. 
\end{lemma}

\begin{itemize}
\item Note that the conclusion holds only up to the first time when either of the points $\ee^{\ii\theta^p}$ or $\ee^{\ii\theta^2}$ is disconnected from the origin, 
since this is the time domain of the local martingale~\eqref{eqn::multiradialSLE_mart_cor}. 
\item When $p=1$, we have $M_{t}(\nradpartfn{1}{\mu}) = \exp\big(\frac{\mu}{\kappa} (\xi_{t} - \frac{\mu}{2} t)\big)$, which gives a single radial SLE with spiral. 
\end{itemize}

\begin{proof}
With $\bs{t}=(t, 0, \ldots, 0)$, the local martingale~\eqref{eqn::multiradialSLE_mart} becomes~\eqref{eqn::multiradialSLE_mart_cor} and $g_{\bs{t}} = g_{t}$. Hence, 
the law of $\gamma^{1}$ is the same as radial $\SLE_{\kappa}$ in $(\U; \ee^{\ii\theta^1}; 0)$ tilted by $M_t(\nradpartfn{p}{\mu})$,
and its driving function $\xi_t$ solves 
\begin{align*}
\ud \xi_t = \sqrt{\kappa} \, \ud B_t+\kappa \, \partial_1 \log\nradpartfn{p}{\mu}(\xi_t, \covmap_t(\theta^2), \ldots, \covmap_t(\theta^p)) \, \ud t.
\end{align*}
Comparing with the partition function~\eqref{eqn::pf_radialSLEkapparhomu}, this shows that the law of $\gamma^{1}$ is radial $\SLE_{\kappa}^{\mu}(2,\ldots, 2)$. 
\end{proof}

\begin{lemma}\label{lem::multiradial_collision}
Fix $\kappa\in (0,4], \, \mu\in\R, \, p\ge 2$, and $\bs{\theta}=(\theta^1, \ldots, \theta^p)\in\LX_p$. 
Then, $p$-radial $\SLE_{\kappa}^{\mu}$ curves $\bs{\gamma} = (\gamma^{1}, \ldots, \gamma^{p})$ in $(\U; \ee^{\ii\bs{\theta}}; 0)$ 
almost surely do not reach the collision time in finite time.
\end{lemma}


\begin{proof}
	For $\bs{t}=(t_1, \ldots, t_p)$ before the collision time, 
	by Lemma~\ref{lem::multiradial_marginal} and the domain Markov property of multiradial SLE, 
	the law of $g_{\bs{t}}(\gamma^{1})$ is a radial $\SLE_{\kappa}^{\mu}(2, \ldots, 2)$ in $(\U; \ee^{\ii\bs{\mslitdriv}_{\bs{t}}};0)$. 
	When $\kappa\le 4$, we know that $g_{\bs{t}}(\gamma^{1})$ is a continuous simple curve 
	in $\U$ from $\ee^{\ii\mslitdriv_{\bs{t}}^{1}}$ to $0$ 
	and it almost surely does not hit any point on the boundary except the starting point $\ee^{\ii\mslitdriv_{\bs{t}}^{1}}$ (see Proposition~\ref{prop::radialSLE_avoidboundary}). 
	This property implies that $\gamma^1$ cannot disconnect any other starting points $\ee^{\ii\theta^j}$ from the target point $0$, so the disconnection time in Lemma~\ref{lem::multiradial_marginal} is almost surely infinite and its conclusion holds for all time.
	This shows that, given $\bs{\gamma}_{\bs{t}}$ and with $\bs{t}$ before the collision time, 
	$(\gamma_s^{1})_{s\ge t_1}$ is a continuous simple curve that does not hit $\bs{\gamma}_{\bs{t}}$ nor the boundary $\partial \mathbb{U} \setminus \{\ee^{i\omega_{\mathbf{t}}^1}\}$. 
	This holds for each $(\gamma^{j}_s)_{s\ge t_j}$ with $1\le j\le p$ and for all $\bs{t}$ before the collision time. 
	Therefore, $\bs{\gamma}$ cannot reach the collision time at finite time. 
\end{proof}

\subsection{Multiradial SLE with spiral: a coordinate change}\label{subsec::coordinate_change}

In Definition~\ref{def::multiradialSLEspiral}, we used the multi-time parameter (local) martingale~\eqref{eqn::multiradialSLE_mart} to define $p$-radial $\SLE_{\kappa}^{\mu}$ in $(\U; \ee^{\ii\bs{\theta}}; 0)$. 
In this section, we derive the corresponding martingale for $p$-radial $\SLE_{\kappa}^{\mu}$ in $(\U; \ee^{\ii\bs{\theta}}; z)$, where the target point $z\in\U$ is any interior point in $\U$. 
This process will be useful in Section~\ref{sec::halfwatermelonSLE}. 

\smallskip

The martingale of interest involves the partition function $\nradpartfn{p}{\mu}$ for $(\U; \ee^{\ii\bs{\theta}}; z)$. 
To find an expression for it, let $\varphi(w) = \frac{w-z}{1-z^*w}$ be the conformal self-map of the unit disc such that $\varphi(z)=0$ and $\varphi'(z)= \frac{1}{1-|z|^2}>0$,
which gives the \emph{conformal radius} of $\U$ seen from $z$, denoted $\CR(\U;z)=1-|z|^2$ (see Eq.~\eqref{eqn::CR_H_and_U}).
A slightly tedious computation using the definitions in~\eqref{eq:multiradial_partition_function}--\eqref{eqn::Gmu_cov}  and the identity 
\begin{align}\label{eq:Mob_identity}
|\varphi(u) - \varphi(v)| = \sqrt{|\varphi'(u)| |\varphi'(v)|} \, |u-v|
\end{align}
(which holds for any M\"obius map by a direct calculation) yields
\begin{align}\label{eqn::nradpartfn_mu_z}
\begin{split}
\nradpartfn{p}{\mu}(\U; \ee^{\ii\bs{\theta}}; z) 
= & \; \frac{\LA^{\mu}_{p\mathrm{\textnormal{-}rad}}(\U; \ee^{\ii\bs{\theta}}; z) \,\nradpartfn{p}{\mu}(\U; \ee^{\ii\bs{\theta}}; 0)}{\big( \CR(\U;z) \big)^{\tilde{\mathfrak{b}}+\frac{p^2-1}{2\kappa} - \frac{\mu^2}{2\kappa} - \mathfrak{h}_p} \,\big(\prod_{j=1}^p |z - \ee^{\ii\theta^j}| \big)^{\frac{2}{p} \mathfrak{h}_p}} ,  
\end{split}
\end{align}
where the exponents are 
$\tilde{\mathfrak{b}}+\frac{p^2-1}{2\kappa} - \frac{\mu^2}{2\kappa} - \mathfrak{h}_p 
= \Delta(\mathfrak{e}_{0,p/2},\mathfrak{m}_\mu)+\Delta^*(\mathfrak{e}_{0,p/2},\mathfrak{m}_\mu) - \Delta(\mathfrak{e}_{1,p+1}) 
= -\frac{(\kappa-4-2p)^2}{8\kappa}-\frac{\mu^2}{2\kappa}$
and 
$\frac{2}{p} \mathfrak{h}_p
= \Delta(\mathfrak{e}_{1,p+1}) + \Delta(\mathfrak{e}_{1,2}) -\Delta(\mathfrak{e}_{1,p})
= \frac{2}{\kappa}(p+2)-1$  
in terms of the conformal weights~\eqref{eqn::spiral_conformal_weight} and 
\begin{align}\label{eqn::Kac_conformal_weight}
\mathfrak{h}_p := \Delta(\mathfrak{e}_{1,p+1}) =  \frac{p(p+2)}{\kappa} - \frac{p}{2} ,
\end{align}
and where $\LA^{\mu}_{p\mathrm{\textnormal{-}rad}}(\U; \ee^{\ii\bs{\theta}}; z) := \exp \big( \frac{2\mu}{\kappa} \sum_{j=1}^{p} \arg ( 1-z\ee^{-\ii\theta^j} ) \big)$. 

\begin{lemma}\label{lem::multitime_mart_nradial_spiral_z_RN}
Fix $\kappa>0, \, p\ge 1$, and $\bs{\theta}=(\theta^1, \ldots, \theta^p)\in\LX_p$. 
\begin{itemize}
\item Let $\gamma^{j}$ be radial $\SLE_{\kappa}$ in $(\U; \ee^{\ii\theta^j}; 0)$ for each $1\le j\le p$, 
and let $\mathsf{P}_p$ be the probability measure on $\bs{\gamma}=(\gamma^{1}, \ldots, \gamma^{p})$ under which the curves are independent. 
\item For $z\in\U$, let $\bs{\gamma}=(\gamma^1, \ldots, \gamma^p)\sim\PPspiral{p}(\U; \ee^{\ii\bs{\theta}}; z)$ be $p$-radial $\SLE_{\kappa}^{\mu}$ in $(\U; \ee^{\ii\bs{\theta}}; z)$. 
\end{itemize} 
Viewing it as a $p$-tuple of curves, we parameterize $\bs{\gamma}$ using $p$-time parameter $\bs{t}=(t_1, \ldots, t_p)$, and let $\bs{\mslitdriv}_{\bs{t}} = (\mslitdriv^{1}_{\bs{t}},\ldots,\mslitdriv^{p}_{\bs{t}})$ denote its multi-slit driving function~\eqref{eqn::Ntuple_driving}.
Write also
\begin{align*}
\mathfrak{A}^{\mu}_{\bs{t}}(z) := \exp\Big(\frac{\mathfrak{c}}{2}\blm_{\bs{t}}-\tilde{\mathfrak{b}}\sum_{j=1}^p t_j -\frac{\mu p }{\kappa} \arg g_{\bs{t}}'(z) \Big) .
\end{align*}
Then, the law of $\bs{\gamma}$ under $\PPspiral{p}(\U; \ee^{\ii\bs{\theta}}; z)$ is the same as $\mathsf{P}_p$ tilted by the following $p$-time parameter local martingale, up to the collision time: 
\begin{align}\label{eqn::multitime_mart_nradial_spiral_z}
\; & M_{\bs{t}}(\nradpartfn{p}{\mu}; z)
:= \one_{\LE_{\emptyset}(\bs{\gamma}_{\bs{t}})} \, 
\mathfrak{A}^{\mu}_{\bs{t}}(z) \, |g_{\bs{t}}'(z)|^{\tilde{\mathfrak{b}}+\frac{p^2-1}{2\kappa}-\frac{\mu^2}{2\kappa}} \,\Big(\prod_{j=1}^{p} \covmap_{\bs{t},j}'(\xi_{t_j}^{j}) \Big)^{\mathfrak{b}} 
\, \nradpartfn{p}{\mu}(\U; \ee^{\ii\bs{\mslitdriv}_{\bs{t}}}; g_{\bs{t}}(z)),
\end{align}
where $\LE_{\emptyset}(\bs{\gamma}_{\bs{t}}) = \{ \gamma_{[0,t_j]}^{j} \cap \gamma_{[0,t_i]}^{i}=\emptyset, \, \forall i\neq j\}$ is the event that different curves are disjoint, 
and $\blm_{\bs{t}}$ is the unique potential solving the exact differential equation~\eqref{eqn::mt_def} \textnormal{(}see Lemma~\ref{lem::blm_exact}\textnormal{)}.
\end{lemma}

When $z=0$, the process $M_{\bs{t}}(\nradpartfn{p}{\mu}; 0)$ is exactly the same as~\eqref{eqn::multiradialSLE_mart}.

\begin{proof}
First, for a general $p$-polygon $(\U;\bs{x};z)$, combining the conformal covariance rule~\eqref{eqn::Gmu_cov} of the partition function with Equation~\eqref{eqn::multiradialSLE_mart}, we obtain
\begin{align*} 
\frac{\ud \PPspiral{p}(\U; \bs{x}; z)}{\ud \mathsf{P}_p(\U; \bs{x}; z)} (\bs{\gamma_t}) = \; &  \one_{\LE_{\emptyset}(\bs{\gamma}_{\bs{t}})} \, \exp\Big(\frac{\mathfrak{c}}{2} \blm_{\bs{t}} \Big) \, \frac{\nradpartfn{p}{\mu}(\U \setminus \bs{\gamma}_{[\bs{0},\bs{t}]} ;\gamma_{t_1}^1,\ldots,\gamma_{t_p}^p;z)}{\prod_{j=1}^p \nradpartfn{1}{0}(\U\setminus \gamma_{[0,t_j]} ;\gamma_{t_j}^j;z) } \; \frac{\prod_{j=1}^p \nradpartfn{1}{0}(\U ;x^j;z) }{\nradpartfn{p}{\mu}(\U ;\bs{x} ; z)}.
\end{align*}
Note that, although the boundary of $\U\setminus \bs{\gamma}_{[\bs{0},\bs{t}]}$ is rough at the tips $\gamma_{t_j}^j$, the above ratio 
is well-defined as 
\begin{align*}
 \; & \frac{\nradpartfn{p}{\mu}(\U \setminus \bs{\gamma}_{[\bs{0},\bs{t}]} ;\gamma_{t_1}^1,\ldots,\gamma_{t_p}^p;z)}{\prod_{j=1}^p \nradpartfn{1}{0}(\U\setminus \gamma_{[0,t_j]} ;\gamma_{t_j}^j;z) } 
\\
= \; &  \frac{ \nradpartfn{p}{\mu}(\U;\ee^{\ii \bs{\mslitdriv_t}};g_{\bs{t}}(z)) }{ \prod_{j=1}^p \nradpartfn{1}{0}(\U ;\exp(\ii\xi_{t_j}^j);g_{t_j}^j(z)) } \, \exp\Big(\! -\frac{\mu p}{\kappa} \arg g_{\bs{t}}'(z) \Big) 
\, | g_{\bs{t}}'(z) |^{\tilde{\mathfrak{b}}+\frac{p^2-1}{2\kappa}-\frac{\mu^2}{2\kappa}} \,  \Big( \prod_{j=1}^p \covmap_{\bs{t},j}'(\xi_{t_j}^j) \Big)^{\mathfrak{b}} \, \Big( \prod_{j=1}^p |( g_{t_j}^j )'(z)| \Big)^{-\tilde{\mathfrak{b}}} .
\end{align*}

Second, for two $p$-polygons $(\U;\bs{x};z)$ and $(\U;\bs{x};w)$, combining the conformal covariance rule~\eqref{eqn::Gmu_cov} of the partition function with the well-known coordinate change~\cite[Eq.~(9)]{Schramm-Wilson:SLE_coordinate_changes}, we obtain
\begin{align*} 
\frac{\ud \mathsf{P}_p(\U; \bs{x}; z)}{\ud \mathsf{P}_p(\U; \bs{x}; w)} (\bs{\gamma_t}) =
& \; \frac{\prod_{j=1}^p \nradpartfn{1}{0}(\U\setminus \gamma^j_{[0,t_j]} ;\gamma_{t_j}^j;z)}{\prod_{j=1}^p \nradpartfn{1}{0}(\U\setminus \gamma^j_{[0,t_j]} ;\gamma_{t_j}^j;w) } \; \frac{\prod_{j=1}^p \nradpartfn{1}{0}(\U ;x^j;w) }{\prod_{j=1}^p \nradpartfn{1}{0}(\U ;x^j;z)} ,
\end{align*}
where again, although the boundary of $\U\setminus\gamma^j_{[0,t_j]}$ is rough at the tip $\gamma^j_{t_j}$, the above ratio 
is well-defined as
\begin{align*}
\frac{\nradpartfn{1}{0}(\U\setminus \gamma^j_{[0,t_j]} ;\gamma_{t_j}^j;z)}{\nradpartfn{1}{0}(\U\setminus \gamma^j_{[0,t_j]} ;\gamma_{t_j}^j;w) } 
\; = \,  \bigg| \frac{(g_{t_j}^j)'(z)}{(g_{t_j}^j)'(w)} \bigg|^{\tilde{\mathfrak{b}}}  \, \frac{ \nradpartfn{1}{0}(\U ;\exp(\ii\xi_{t_j}^j);g_{t_j}^j(z))}{ \nradpartfn{1}{0}(\U ;\exp(\ii\xi_{t_j}^j);g_{t_j}^j(w))} .
\end{align*}

In summary, we obtain the desired formulas by taking $w=0$ in
\begin{align} \label{eqn::multitime_mart_nradial_spiral_z_aux3}
\begin{split}
\frac{\ud \PPspiral{p}(\U; \bs{x}; z)}{\ud \mathsf{P}_p(\U; \bs{x}; w)} (\bs{\gamma_t}) 
=   \one_{\LE_{\emptyset}(\bs{\gamma}_{\bs{t}})} \, \exp\Big(\frac{\mathfrak{c}}{2} \blm_{\bs{t}}  \Big) \, \frac{\nradpartfn{p}{\mu}(\U \setminus \bs{\gamma}_{[\bs{0},\bs{t}]} ;\gamma_{t_1}^1,\ldots,\gamma_{t_p}^p;z)}{\prod_{j=1}^p \nradpartfn{1}{0}(\U\setminus \gamma_{[0,t_j]} ;\gamma_{t_j}^j;w) } \; \frac{\prod_{j=1}^p \nradpartfn{1}{0}(\U ;x^j;w) }{\nradpartfn{p}{\mu}(\U ;\bs{x} ; z)}.
\end{split}
\end{align}
Indeed, if we set $w=0$, then the right-hand side of~\eqref{eqn::multitime_mart_nradial_spiral_z_aux3} equals $M_{\bs{t}}(\nradpartfn{p}{\mu}; z)/M_{\bs{t}}(\nradpartfn{p}{\mu}; 0)$. 
\end{proof}

\section{Half-watermelon SLEs}
\label{sec::halfwatermelonSLE}
In this section, we consider half-watermelon SLE processes (recall Definition~\ref{def::halfwatermelonSLE}), 
which will also play a key role in Section~\ref{sec::Multiradial_final} 
in the proofs of the key results, Theorem~\ref{thm::multiradial_resampling} and Proposition~\ref{prop::multiradial_bp}.
We prove that multiradial SLEs (Definition~\ref{def::multiradialSLEspiral}) converge to half-watermelon SLEs when the common endpoint of the curves tends to the boundary (see Theorem~\ref{thm::multiradial_halfwatermelon})
and show that they satisfy a natural boundary perturbation property (see Proposition~\ref{prop::halfwatermelon_bp}).
At an intermediate step, we derive half-watermelon SLEs from a fusion of multiple chordal SLEs (with the so-called rainbow connectivity; see Theorem~\ref{thm::multichordal_halfwatermelon}).

\paragraph*{Partition functions.}  

Throughout, we will write $\bs{x}=(x^1, \ldots, x^n)$ and $\bs{y}=(y^n, \ldots, y^1)$ for convenience. 
For a polygon $(\Omega; \bs u)$, we always assume that the marked points $\bs{u}=(u^1, \ldots, u^m)$ appear in counterclockwise order.
For $u,v \in \partial \Omega$, we denote by $(u \, v)$ the \emph{counterclockwise boundary arc} from $u$ to $v$.

\smallskip

The half-$n$-watermelon $\SLE_\kappa$ partition function\footnote{Here, $\LZ_{\defpatt_n}(\HH;\bs x)$ is the partition function for half-$n$-watermelon $\SLE_{\kappa}$ curves all growing to infinity,
and $\LZfusion{n}(\HH;\bs x, y)$ is the partition function for half-$n$-watermelon $\SLE_{\kappa}$ in $(\HH; \bs{x}, y)$ with target point $y$ on the real line.} is
\begin{align} \label{eqn::LZfusion_H}
\begin{split}
\LZ_{\defpatt_n}(\HH;\bs x) := \; & \Big( \prod_{1\leq i<j\leq n}(x^{j}-x^{i}) \Big)^{2/\kappa} \qquad \textnormal{(with } y = \infty),
\\
\LZfusion{n}(\HH;\bs x, y) := \; & \LZ_{\defpatt_n}(\HH;\bs x) \; \Big( \prod_{j=1}^n(y-x^j) \Big)^{1-\frac{2}{\kappa}(n+2)}  \qquad \textnormal{(with } y > x^n),
\\
\LZfusionInv{n}(\HH;x, \bs y) := \; & \LZ_{\defpatt_n}(\HH;\bs y) \; \Big( \prod_{j=1}^n (y^j-x) \Big)^{1-\frac{2}{\kappa}(n+2)}  \qquad \textnormal{(with } x < y^1),
\end{split}
\end{align} 
where $\frac{2}{\kappa} = \mathfrak{h}_2 - 2\mathfrak{h}_1 = \mathfrak{h}_2 - 2\mathfrak{b}$ 
and $1-\frac{2}{\kappa}(n+2) = \Delta(\mathfrak{e}_{1,n}) - \Delta(\mathfrak{e}_{1,n+1}) - \Delta(\mathfrak{e}_{1,2}) = -\frac{2}{n}\mathfrak{h}_n$ in terms of the Kac conformal weights 
$\mathfrak{h}_n := \Delta(\mathfrak{e}_{1,n+1})$ in~\eqref{eqn::Kac_conformal_weight};  
recall that $\mathfrak{b} = \mathfrak{h}_1= \Delta(\mathfrak{e}_{1,2})$ is the $\SLE_\kappa$ boundary exponent~\eqref{eqn::universal_parameters}). 
As usual, we extend this definition to more general polygons $(\Omega; \bs x, y)$ and $(\Omega; x, \bs{y})$ by
\begin{align} \label{eqn::LZfusion_cov}
\begin{split}
\LZfusion{n}(\Omega;\bs x, y) := \; &  |\varphi'(y)|^{\mathfrak{h}_n} \Big(
\prod_{j=1}^n |\varphi'(x^j)| \Big)^{\mathfrak{b}} \; \LZfusion{n}(\HH;\varphi(\bs x), \varphi(y)) ,
\\
\LZfusionInv{n}(\Omega;x, \bs y) := \; &  |\varphi'(x)|^{\mathfrak{h}_n} \Big( 
\prod_{j=1}^n |\varphi'(y^j)| \Big)^{\mathfrak{b}} \; \LZfusionInv{n}(\HH;\varphi(x), \varphi(\bs y)) ,
\end{split}
\end{align} 
where we write $\varphi(\bs u) := (\varphi(u^1),\ldots,\varphi(u^n))$ and 
$\varphi$ is any conformal map from $\Omega$ onto $\HH$ preserving the order of the marked points: 
$\varphi(x^1)<\cdots<\varphi(x^n)<\varphi(y)$ or $\varphi(x)<\varphi(y^n)<\cdots<\varphi(y^1)$. 

Note that, for half-$n$-watermelon $\SLE_{\kappa}$ $\bs{\eta} = (\eta^{1}, \ldots, \eta^{n}) \sim \QQfusion{n}(\Omega; \bs{x}, y)$, 
the marginal law of each $\eta^{j}$ is chordal $\SLE_{\kappa}(2, \ldots, 2)$ in $\Omega$ from $x^j$ to $y$ with force points $(x^1, \ldots, x^{j-1}; x^{j+1}, \ldots, x^n)$
--- in particular, the curves in half-$n$-watermelon $\SLE_{\kappa}$ are almost surely transient, see~\cite[Section~4]{Miller-Sheffield:Imaginary_geometry1}. 

\bigskip

It was detailed in~\cite[Section~8]{Lawler-Lind:Two-sided_SLE8_over_3_and_infinite_self-avoiding_polygon} how the law of radial $\SLE_\kappa$ 
converges to the law of chordal $\SLE_\kappa$ as the endpoint of the curve approaches the boundary.  
The first goal of this section is to generalize this: 
we show that multiradial $\SLE_{\kappa}^\mu$ converges to half-watermelon $\SLE_{\kappa}$ when the common endpoint of the curves tends to the boundary (Theorem~\ref{thm::multiradial_halfwatermelon}). 
This is a key ingredient in the proof of the resampling property for Theorem~\ref{thm::multiradial_resampling}. 
Recall that $\nradpartfn{n}{0}$ denotes the partition function defined in~(\ref{eq:multiradial_partition_function},~\ref{eqn::multiradial_partition_function_cov}) 
and $\PPnospiral{n}$ the law of $n$-radial $\SLE_{\kappa}$ (without spiral),
while $\PPspiral{n}$ denotes the law of $n$-radial $\SLE_{\kappa}^\mu$ (with spiral).

\begin{theorem}\label{thm::multiradial_halfwatermelon}
Fix $\kappa\in (0,4], \, n\ge 1$, and an $(n+1)$-polygon $(\Omega; \bs{x}, y)$.
\begin{enumerate}
\item \label{item::multiradial_halfwatermelonPF}
The partition function $\nradpartfn{n}{0}$ converges to $\LZfusion{n}$ after proper normalization:
\begin{align}\label{eqn::LGtoLZfusion}
\lim_{\Omega \, \ni \, z\to y} \frac{\nradpartfn{n}{0}(\Omega; \bs{x}; z)}{( \CR(\Omega; z) )^{\mathfrak{h}_n - \tilde{\mathfrak{b}} - \frac{n^2-1}{2\kappa}}}
= \LZfusion{n}(\Omega; \bs{x}, y) ,
\end{align}
where $\CR(\Omega; z) := 1/\varphi'(z)$ is the conformal radius of $\Omega$ seen from $z\in \Omega$, defined in terms of 
the conformal bijection $\varphi \colon \Omega \to \U$ such that $\varphi(z)=0$ and $\varphi'(z)>0$,
and the exponent is 
\begin{align*}
\mathfrak{h}_n - \tilde{\mathfrak{b}} - \frac{n^2-1}{2\kappa}
= \Delta(\mathfrak{e}_{1,n+1})  - 2 \Delta(\mathfrak{e}_{0,n/2})
=\frac{(\kappa-4-2n)^2}{8\kappa}
\end{align*}
in terms of the conformal weights~\textnormal{(\ref{eqn::universal_parameters},~\ref{eqn::spiral_conformal_weight},~\ref{eqn::Kac_conformal_weight})}. 

\item \label{item::multiradial_halfwatermelonMEAS}
Fix $\mu \in \R$. Suppose $U\subset\Omega$ is simply connected, has the points $x^1, \ldots, x^n$ on its boundary, and has a positive distance from $y$. 
Then, as $z\to y$, the law $\PPspiral{n}(\Omega; \bs{x}; z)$ of $n$-radial $\SLE_{\kappa}^\mu$ 
converges weakly to the law $\QQfusion{n}(\Omega; \bs{x}, y)$ of half-$n$-watermelon $\SLE_{\kappa}$, up to the exit time from $U$. 
\end{enumerate}
\end{theorem}

We prove Item~\ref{item::multiradial_halfwatermelonPF} in Section~\ref{subsec::Poissonkernel} (see Lemma~\ref{lem::LGtoLZfusion}) and 
Item~\ref{item::multiradial_halfwatermelonMEAS} in Section~\ref{subsec::multiradial_halfwatermelon} (after Lemma~\ref{lem::rainbowSLE_mart}).

\bigskip

The second goal of this section is to show that multichordal $\SLE_{\kappa}$ converges to half-watermelon $\SLE_{\kappa}$ when the endpoints of the $n$ distinct curves tend to a common point on the boundary (Theorem~\ref{thm::multichordal_halfwatermelon}).

\paragraph*{Multichordal SLE.} 
For a $2$-polygon $(\Omega; x,y)$, we denote by $\chamber(\Omega; x,y)$ the space of continuous simple curves (chords) $\gamma\subset\overline{\Omega}$ from $x$ to $y$ such that $\gamma\cap\partial\Omega = \{x,y\}$. 
In a $2n$-polygon $(\Omega; x^1, \ldots, x^{2n})$, we denote by $\chamberrainbow{n}(\Omega; x^1, \ldots, x^{2n})$ the set of collections $\bs{\eta}=(\eta^{1}, \ldots, \eta^{n})$ of pairwise disjoint curves where $\eta^{j}\in\chambertwo(\Omega; x^j, x^{2n+1-j})$ for each $j$. 
The connectivity of the curves forms the ``rainbow link pattern''
\begin{align}\label{eqn::rainbow}
\rainbowBig_n \; := \; \{\{1,2n\}, \{2, 2n-1\}, \ldots, \{n, n+1\}\}. 
\end{align}

\begin{definition}[Multichordal SLE]\label{def::globalnSLE}
Fix $\kappa\in (0,4]$. Let $(\Omega; x^1, \ldots, x^{2n})$ be a $2n$-polygon, $n\ge 2$. 
The \emph{$n$-chordal} $\SLE_{\kappa}$ associated to the rainbow link pattern $\rainbowBig_n$ 
is the unique probability measure $\QQrainbow{n}(\Omega; x^1, \ldots, x^{2n})$ on the configuration space $\chamberrainbow{n}(\Omega; x^1, \ldots, x^{2n})$ satisfying the \emph{chordal resampling property}:  
for each $j\in \{1, \ldots, n\}$, the conditional law of $\eta^{j}$ given $\{\eta^{i} \colon i\neq j\}$ is chordal $\SLE_{\kappa}$ from $x^j$ to $x^{2n+1-j}$ 
in the connected component $\Omega_j$ of $\Omega\setminus\cup_{i\neq j}\eta^{i}$ having $x^j$ and $x^{2n+1-j}$ on its boundary.
\end{definition}

The existence and uniqueness of multichordal $\SLE_{\kappa}$ for different ranges of $\kappa$ has been widely studied, see~\cite{Miller-Sheffield:Imaginary_geometry1, Miller-Sheffield:Imaginary_geometry2, Peltola-Wu:Global_and_local_multiple_SLEs_and_connection_probabilities_for_level_lines_of_GFF, Wu:Convergence_of_the_critical_planar_ising_interfaces_to_hypergeometric_SLE, BPW:On_the_uniqueness_of_global_multiple_SLEs, AHSY:Conformal_welding_of_quantum_disks_and_multiple_SLE_the_non-simple_case, Zhan:Existence_and_uniqueness_of_nonsimple_multiple_SLE, FLPW:Multiple_SLEs_Coulomb_gas_integrals_and_pure_partition_functions}.  
In particular, the existence and uniqueness holds for all $\kappa\in (0,8)$. 
The chordal Loewner evolution of curves under $\QQrainbow{n}$ can be characterized in terms of the ``rainbow partition functions'' 
$\LZrainbow{n}$ discussed in more detail in Section~\ref{subsec: Proof of limit}. 

\bigskip

We prove the following result in Section~\ref{subsec: Proof of perturbation}. 

\begin{theorem}\label{thm::multichordal_halfwatermelon}
Fix $\kappa\in (0,4], \, n\ge 1$, and a $2n$-polygon $(\Omega; \bs{x}, \bs{y})$ with $\bs{x}=(x^1, \ldots, x^n), \, \bs{y}=(y^n, \ldots, y^1)$, 
and let $y$ and $x$ be two points on the boundary arcs $(x^n x^1)$ and $(y^1 y^n)$, respectively\footnote{Here, recall the arc notation introduced in the beginning of Section~\ref{sec::halfwatermelonSLE}.}.  
Let $H(\Omega; x, y)$ denote the boundary Poisson kernel in $\Omega$ \textnormal{(}see~\textnormal{(\ref{eqn::Poisson_def_H},~\ref{eqn::bPoisson_cov}))}.
\begin{enumerate}
\item The partition function $\LZrainbow{n}$ converges to $\LZfusion{n}$ after proper normalization:
\begin{align}\label{eqn::LZrainbowtoLZfusion}
\lim_{\substack{y^j\to y\\\forall 1\le j\le n}}\frac{\LZrainbow{n}(\Omega; \bs{x}, \bs{y})}{\LZfusionInv{n}(\Omega;x, \bs y)} 
= \frac{A_n}{( H(\Omega; x, y) )^{\mathfrak{h}_n}} \; \LZfusion{n}(\Omega; \bs{x}, y),
\end{align}
where the multiplicative constant $A_n$ is given
in terms of the $q$-integers 
\begin{align}
\qnum{m} = \frac{q^{m}-q^{-m}}{q-q^{-1}} 
\qquad \textnormal{and} \qquad 
\qfact n = \prod_{m=1}^{n}\qnum m ,
\qquad q = \ee^{4 \pi \ii / \kappa} ,
\end{align}
and ratios of Euler Gamma functions $\Gamma(\cdot)$ as\footnote{It follows from~\eqref{eqn::LZrainbowtoLZfusion} that the constant $A_n$ is always finite when $\kappa \in (0,8)$.}
\begin{align} 
\label{eq: fusion cst}
A_n = A_n(\kappa) := \; & 
\frac{\qnum{2}^{n} \qfact{n}}{\qnum{n+1}}  \, 
\Big( \frac{\Gamma(2-8/\kappa)}{\Gamma(1-4/\kappa)^2} \Big)^n \,
\Selberg_{n} , 
\\
\label{eq: Selberg}
\Selberg_{n} = \Selberg_{n}(\kappa) := \; & \frac{1}{n!} \; \prod_{u=1}^{n}
\frac{\Gamma( 1 - \frac{4}{\kappa}(n+1-u)) \; \Gamma( 1 - \frac{4}{\kappa}(n+1-u)) \; \Gamma( 1 + \frac{4}{\kappa}u )}
{\Gamma( 1 + \frac{4}{\kappa}) \; \Gamma( 2 - \frac{4}{\kappa}(n+2-u))} .
\end{align}

\item Suppose $U\subset\Omega$ is simply connected, has the points $x^1, \ldots, x^n$ on its boundary, and has a positive distance from $y$. 
Then, as $y^j\to y$ for all $1\le j\le n$, the law $\QQrainbow{n}(\Omega; \bs{x}, \bs{y})$ of $n$-chordal $\SLE_{\kappa}$ 
converges weakly to the law $\QQfusion{n}(\Omega; \bs{x}, y)$ of half-$n$-watermelon $\SLE_{\kappa}$, up to the exit time from $U$. 
\end{enumerate}
\end{theorem}

It is important to observe that both constants~(\ref{eq: fusion cst},~\ref{eq: Selberg}) have poles only at finitely many rational values (and the number of the poles depends on $n$).
In fact, the constant~\eqref{eq: Selberg} is nothing but the well-known Selberg integral 
(see~\cite[Remark~3.9]{Kytola-Peltola:Conformally_covariant_boundary_correlation_functions_with_quantum_group} and references therein),
which frequently appears in expressions in conformal field theory, Coulomb gas, and random matrix models. 

\bigskip

As a consequence of Theorem~\ref{thm::multichordal_halfwatermelon}, we derive the boundary perturbation property of half-watermelon $\SLE_{\kappa}$ in Section~\ref{subsec: Proof of perturbation}. 
It will be an ingredient in the proof of the boundary perturbation in Proposition~\ref{prop::multiradial_bp}. 

\begin{proposition}[Boundary perturbation]\label{prop::halfwatermelon_bp}
Fix $\kappa\in (0,4], \, n\ge 1$, and an $(n+1)$-polygon $(\Omega; \bs{x}, y)$. 
For half-$n$-watermelon $\SLE_{\kappa}$ curves $\bs{\eta}=(\eta^{1}, \ldots, \eta^{n})$ in $(\Omega; \bs{x}, y)$, denote $\bs{\eta}=\cup_{i=1}^n\eta^{i}$. 
Suppose $K$ is a compact subset of $\overline{\Omega}$ such that $\Omega\setminus K$ is simply connected and $K$ has a positive distance from $\{x^1, \ldots, x^n, y\}$.
Then, the half-$n$-watermelon $\SLE_{\kappa}$ probability measure 
in the smaller polygon $(\Omega\setminus K; \bs{x}; y)$ is absolutely continuous with respect to that in $(\Omega; \bs{x}, y)$, 
with Radon-Nikodym derivative 
\begin{align}\label{eqn::multichordal_bp}
\frac{\LZfusion{n}(\Omega; \bs{x}, y)}{\LZfusion{n}(\Omega\setminus K; \bs{x}, y)} \, \one\{\bs{\eta}\cap K=\emptyset\}\, \exp\Big(\frac{\mathfrak{c}}{2} \blm(\Omega; \bs{\eta}, K)\Big) ,
\end{align} 
where $\blm(\Omega; \bs{\eta}, K)$ is the Brownian loop measure in $\Omega$ of those loops that intersect both $\bs{\eta}$ and $K$. 
\end{proposition}

\subsection{Preliminaries on conformal radius and Poisson kernel}\label{subsec::Poissonkernel}

\paragraph*{Conformal radius.}
Recall that, for a simply connected domain $\Omega\subsetneq\C$ and $z\in \Omega$, the \emph{conformal radius} $\CR(\Omega; z) := 1/\varphi'(z)$ of $\Omega$ seen from $z$ is defined in terms of 
the conformal bijection $\varphi \colon \Omega \to \U$ such that $\varphi(z)=0$ and $\varphi'(z)>0$. It is conformally covariant: for any conformal map $\phi$ on $\Omega$, we have
\begin{align} \label{eq:CR_COV}
\CR(\phi(\Omega); \phi(z))=|\phi'(z)| \, \CR(\Omega; z).
\end{align}
For example, in the upper half-plane $\HH \ni w$ and disc $\U \ni z$, we respectively have
\begin{align}\label{eqn::CR_H_and_U}
\CR(\HH;w)= 2\Im(w)  
\qquad \textnormal{and} \qquad 
\CR(\U;z)=1-|z|^2 .
\end{align}

\begin{lemma} \label{lem::CR_cvg}
Consider a $1$-polygon $(\Omega;x)$ and a subset $K \subset \overline{\Omega}$ such that $\Omega\setminus K$ is simply connected and $K$ has a positive distance from $x$. Then, we have
\begin{align}\label{eqn::CR_cvg}
\lim_{\Omega \, \ni \, z\to x} \frac{\CR(\Omega \setminus K;z)}{\CR(\Omega;z)}=1. 
\end{align}
\end{lemma}

\begin{proof}
By~\eqref{eq:CR_COV}, it suffices to consider $(\Omega;x)=(\HH;0)$ and assume that $\infty\notin K$ and $\dist(0,K)>1$. 
Let $\varphi_K$ be a conformal map from $\HH\setminus K$ onto $\HH$ satisfying $\varphi_K(\infty)=\infty$ and $\varphi_K(0)=0$. 
Then, we have
\begin{align*} 
\frac{\CR(\HH \setminus K;z)}{\CR(\HH;z)} = \; & \frac{\Im(\varphi_K(z))}{|\varphi'_K(z)|\Im(z)},
\\
\varphi_K(z) = \; & \varphi'_K(0) \, z+ \sum_{n=2}^{\infty} \frac{\varphi_K^{(n)}(0)}{n!} \, z^n ,
\qquad z\in \HH .
\end{align*}
Writing $z=r\ee^{\ii\theta}$ in polar coordinates, we find that
\begin{align} \label{eqn::CR_cvg_aux3}
\frac{\CR(\HH \setminus K;z)}{\CR(\HH;z)} 
= \bigg( 1+\sum_{n=2}^{\infty}  \frac{r^{n-1}}{n!} \,\frac{\varphi_K^{(n)}(0)}{\varphi'_K(0)} \, \frac{\sin(n\theta)}{\sin(\theta)} \bigg) \frac{\varphi'_K(0)}{|\varphi'_K(z)|}.
\end{align}
On the one hand, we have $|\sin(n\theta)/\sin(\theta)|\le n$. On the other hand, since $\dist(0,K)>1$, Schwarz reflection guarantees that $\varphi_K$ is a univalent function on the unit disc, 
so de Branges's theorem~\cite{DeBranges:Proof_of_Bieberbach_conjecture} 
gives $|\varphi_K^{(n)}(0)|\le n \, |\varphi'_K(0)|$. 
Combining with~\eqref{eqn::CR_cvg_aux3}, dominated convergence theorem implies~\eqref{eqn::CR_cvg}. 
\end{proof}

\paragraph*{Poisson kernel.} 
For the $1$-polygon $(\HH; x; z)$ with one interior point, the \emph{Poisson kernel} is defined as
\begin{align}\label{eq:PK_HH}
H(\HH; x; z):= \frac{2\Im(z)}{|z-x|^2} , \qquad x\in\R , \; z\in\HH ,
\end{align}
and for a general polygon $(\Omega; x; z)$, we extend its definition via conformal covariance: 
\begin{align}\label{eq:PK_COV}
H(\Omega; x; z) := |\varphi'(x)| \, H(\HH; \varphi(x); \varphi(z)),
\end{align}
where $\varphi$ is any conformal map from $\Omega$ onto $\HH$. 
In particular, for the disc $(\U; \ee^{\ii\theta}, z)$, we have
\begin{align}\label{eqn::Poissonkernel_disc_int}
H(\U; \ee^{\ii\theta}, z) = \frac{1-|z|^2}{|z - \ee^{\ii\theta}|^2} = \frac{\CR(\U;z)}{|z - \ee^{\ii\theta}|^2} .
\end{align}
Conventionally\footnote{It is usually denoted as $H(\Omega; z,x)$, but we adopt the notation $H(\Omega; x; z)$ to match with the notation for $1$-polygons with an interior marked point.}, 
the Poisson kernel is thought of as the Brownian excursion kernel from $z$ to $x$ in $\Omega$.

\begin{lemma} \label{lem::PoissonKerCR_convergence}
Consider a $1$-polygon $(\Omega;x)$ and a subset $K \subset \overline{\Omega}$ such that $\Omega\setminus K$ is simply connected and $K$ has a positive distance from $x$. Then, we have
\begin{align} \label{eqn::PoissonKer_convergence}
\lim_{\Omega \, \ni \, z\to x} \frac{H(\Omega\setminus K;x;z)}{H(\Omega;x;z)}=1.
\end{align}
\end{lemma}

\begin{proof}
By~\eqref{eq:PK_COV}, it suffices to consider $(\Omega;x;z)=(\HH;0;z)$ and assume that $\infty\notin K$. 
Let $\varphi_K$ be a conformal map from $\HH\setminus K$ onto $\HH$ satisfying $\varphi_K(\infty)=\infty$ and $\varphi_K(0)=0$. 
Then, we have
\begin{align} \label{eqn::PoissonKerergence_aux1}
\frac{H(\HH\setminus K;0;z)}{H(\HH;0;z)}
=\frac{\CR(\HH \setminus K;z)}{\CR(\HH;z)} \, \Big| z^2  \, \frac{\varphi'_K(z) \, \varphi'_K(0) }{\varphi_K(z)^2} \Big|.
\end{align}
Plugging~\eqref{eqn::CR_cvg} from Lemma~\ref{lem::CR_cvg} into~\eqref{eqn::PoissonKerergence_aux1} we obtain~\eqref{eqn::PoissonKer_convergence}.
\end{proof}

\paragraph*{Boundary Poisson kernel.} 
For the $2$-polygon $(\HH; x,y)$, the \emph{boundary Poisson kernel} is defined as 
\begin{align}\label{eqn::Poisson_def_H}
H(\HH; x,y) := \frac{1}{(y-x)^2},\quad x,y \in\R ,
\end{align}
and for a general polygon $(\Omega; x,y)$, we extend its definition via conformal covariance: 
\begin{align}\label{eqn::bPoisson_cov}
H(\Omega; x,y) := |\varphi'(y)| \, |\varphi'(x)| \, H(\HH; \varphi(x),\varphi(y)), 
\end{align}
where $\varphi$ is any conformal map from $\Omega$ onto $\HH$. 
In particular, for the disc $(\U; \ee^{\ii\theta}, \ee^{\ii\vartheta})$, we have
\begin{align}\label{eqn::Poissonkernel_disc}
H(\U; \ee^{\ii\theta}, \ee^{\ii\vartheta}) = \frac{1}{4\sin^2\big(\frac{\vartheta-\theta}{2}\big)}
= \frac{1}{|\ee^{\ii\vartheta} - \ee^{\ii\theta}|^2} .
\end{align}

The relation between Poisson kernel and boundary Poisson kernel is given by the following lemma. 
\begin{lemma}\label{lem::Poissonkernel_cvg}
For a $2$-polygon $(\Omega; x,y)$, we have
\begin{align}\label{eqn::Poissonkernel_cvg}
\lim_{\Omega \, \ni \, z\to y}  \frac{H(\Omega;x;z)}{\CR(\Omega; z) \, H(\Omega; x,y)}=1. 
\end{align}
\end{lemma}

\begin{proof}
Let $\varphi$ be any conformal map from $\Omega$ onto $\HH$. Then, we have
\begin{align*}
\frac{H(\Omega; x;z)}{\CR(\Omega; z) \, H(\Omega; x,y)}
= \; & \frac{H(\HH; \varphi(x);\varphi(z))}{\CR(\Omega; z) \, |\varphi'(y)| \, H(\HH; \varphi(x),\varphi(y))}
=\frac{2\Im(\varphi(z)) \, |\varphi(y)-\varphi(x)|^2}{\CR(\Omega; z) \, |\varphi'(y)| \, |\varphi(z)-\varphi(x)|^2}. 
\end{align*}
Now, because $|\varphi'(z)| \, \CR(\Omega; z)=\CR(\HH; \varphi(z)) = 2\Im(\varphi(z))$, we obtain
\begin{align*}
\frac{H(\Omega; x;z)}{\CR(\Omega; z) \, H(\Omega; x,y)}
=\frac{|\varphi'(z)| \, |\varphi(y)-\varphi(x)|^2}{|\varphi'(y)| \, |\varphi(z)-\varphi(x)|^2} \longrightarrow 1, \qquad \textnormal{as }z\to y \textnormal{ in }\Omega.
\end{align*}
\end{proof}

The ratio of the partition functions $\nradpartfn{n}{0}$ and $\LZfusion{n}$ can be written in terms of Poisson kernels and conformal radius. 
This yields Item~\ref{item::multiradial_halfwatermelonPF} of Theorem~\ref{thm::multiradial_halfwatermelon}. 

\begin{lemma}\label{lem::LGtoLZfusion}
For an $(n+1)$-polygon $(\Omega; \bs x, y)$, $n\ge 1$, and $z \in \Omega$, we have 
\begin{align}\label{eqn::LGandLZfusion}
\frac{1}{( \CR(\Omega; z) )^{\mathfrak{h}_n - \tilde{\mathfrak{b}} - \frac{n^2-1}{2\kappa}}}
 \, \frac{\nradpartfn{n}{0}(\Omega; \bs x; z)}{\LZfusion{n}(\Omega; \bs x, y)} 
= \bigg(\prod_{j=1}^n \frac{H(\Omega; x^j; z)}{\CR(\Omega; z) \, H(\Omega; x^j, y)}\bigg)^{\frac{1}{n}\mathfrak{h}_n} .
\end{align}
Consequently, the convergence~\eqref{eqn::LGtoLZfusion} holds as $z \to y$. 
\end{lemma}

\begin{proof}
By conformal covariance, it suffices to consider the case where $\Omega = \HH$ and $x^1<\cdots<x^n<y$ and $z = \ii r$ with $r > 0$, so $\CR(\HH; \ii r) = 2r$ as in~\eqref{eqn::CR_H_and_U}.
Let $\varphi(z) = \frac{z-\ii r}{z+\ii r}$ be the conformal map from $\HH$ onto $\U$ such that $\varphi(\ii r)=0$ and $\varphi'(z)= \frac{2\ii r}{(z+\ii r)^2}$. 
A slightly tedious computation using the definitions in~(\ref{eq:multiradial_partition_function},~\ref{eqn::multiradial_partition_function_cov}) with $\mu=0$ 
and the identity~\eqref{eq:Mob_identity} yields
\begin{align*}
\nradpartfn{n}{0}(\HH; \bs x; \ii r) 
= \; & ( \CR(\HH; \ii r) )^{\mathfrak{h}_n - \tilde{\mathfrak{b}} - \frac{n^2-1}{2\kappa}} 
\;\Big( \underset{1\leq i < j \leq n}{\prod} \, |x^j - x^i| \Big)^{2/\kappa} 
\Big(\prod_{j=1}^n |\ii r - x^j| \Big)^{1-\frac{2}{\kappa}(n+2)}
\\
= \; & ( \CR(\HH; \ii r) )^{\mathfrak{h}_n - \tilde{\mathfrak{b}} - \frac{n^2-1}{2\kappa}} \, 
\LZfusion{n}(\HH; \bs x, y) \; 
 \bigg(\prod_{j=1}^n \frac{|\ii r - x^j|}{ |y-x^j|}  \bigg)^{1-\frac{2}{\kappa}(n+2)} .
&& \textnormal{[by~\eqref{eqn::LZfusion_H}]}
\end{align*}
where $1-\frac{2}{\kappa}(n+2) = -\frac{2}{n}\mathfrak{h}_n$. 
The asserted identity follows by noting that $H(\HH; x; \ii r) = \frac{\CR(\HH; \ii r)}{|\ii r-x|^2}$
and $H(\HH; x,y) = \frac{1}{|y-x|^2}$  by~(\ref{eq:PK_HH},~\ref{eqn::Poisson_def_H}).
The convergence~\eqref{eqn::LGtoLZfusion} follows from~(\ref{eqn::Poissonkernel_cvg},~\ref{eqn::LGandLZfusion}).
\end{proof}

\subsection{From multiradial SLE to half-watermelon SLE --- proof of Theorem~\ref{thm::multiradial_halfwatermelon}}
\label{subsec::multiradial_halfwatermelon}

Let us first investigate $\LZfusion{n}$ in the unit disc $(\U; \ee^{\ii\bs{\theta}}, \ee^{\ii\vartheta})$ with $\ee^{\ii\bs{\theta}} = (\ee^{\ii\theta^1},\ldots, \ee^{\ii\theta^n})$.
Using the definitions~(\ref{eqn::LZfusion_H},~\ref{eqn::LZfusion_cov}), a computation using a conformal map from $\U$ onto $\HH$ shows that  
\begin{align} \label{LZfusion_U}
\LZfusion{n}(\bs{\theta}, \vartheta)
\; = \;& \LZfusion{n}(\U;\ee^{\ii\bs{\theta}}, \ee^{\ii\vartheta}) 
\;= \; \bigg( \prod_{1\leq i<j\leq n} |\ee^{\ii\theta^j} - \ee^{\ii\theta^i}| \bigg)^{2/\kappa}  \; \bigg( \prod_{\ell=1}^n |\ee^{\ii\vartheta} - \ee^{\ii\theta^\ell}| \bigg)^{1-\frac{2}{\kappa}(n+2)} 
\\
\;= \; & 2^{\frac{n(\kappa-n-5)}{\kappa}} \bigg( \underset{1\leq i < j \leq n}{\prod} \, \sin \Big(\frac{\theta^j - \theta^i}{2}\Big) \bigg)^{2/\kappa} 
\bigg( \prod_{\ell=1}^n \sin\Big(\frac{\vartheta - \theta^\ell}{2}\Big) \bigg)^{1-\frac{2}{\kappa}(n+2)} 
, 
\qquad (\bs{\theta}, \vartheta) \in \LX_{n+1} . 
\nonumber
\end{align}
The partition function $\LZfusion{n} \colon \LX_{n+1}\to \R$ satisfies the following system of radial BPZ equations: 
\begin{align}\label{eqn::radialBPZ_fusion}
\begin{split}
\frac{\kappa}{2} \frac{\partial_j^2 \LZfusion{n}(\bs{\theta}, \vartheta)}{\LZfusion{n}(\bs{\theta}, \vartheta)} 
+ \underset{1\leq \ell\neq j \leq n}{\sum} \,\; &  \bigg( \cot \bigg(\frac{\theta^{\ell}-\theta^{j}}{2}\bigg)
\frac{\partial_{\ell} \LZfusion{n}(\bs{\theta}, \vartheta)}{\LZfusion{n}(\bs{\theta}, \vartheta)} \, - \, \frac{\mathfrak{b}/2}{ \sin^2 \big(\frac{\theta^{\ell}-\theta^{j}}{2}\big)}\bigg)
\\
\; & + \cot \bigg(\frac{\vartheta-\theta^{j}}{2}\bigg)
\frac{\partial_{n+1} \LZfusion{n}(\bs{\theta}, \vartheta)}{\LZfusion{n}(\bs{\theta}, \vartheta)} \, - \, \frac{\mathfrak{h}_n/2}{ \sin^2 \big(\frac{\vartheta-\theta^{j}}{2}\big)} 
\; = \; \tilde{\mathfrak{b}} ,
\end{split}
\end{align}
for all $j \in \{1,\ldots,n\}$, 
associated to the evolution of each starting angle $\theta^j$.  
Though we will not need it, let us point out that with respect to the angular coordinate $\vartheta$ of the target point $\ee^{\ii\vartheta}$ (recall from Definition~\ref{def::halfwatermelonSLE} the target point of half-watermelon $\SLE_{\kappa}$ in $(\U;\ee^{\ii\bs{\theta}},\ee^{\ii\vartheta})$ is $\ee^{\ii\vartheta}$), the partition function $\LZfusion{n} \colon \LX_{n+1}\to \R$ satisfies a PDE of order $n+1$ 
(this follows from~\cite[Theorem~5.3, Part~(1)]{Peltola:Basis_for_solutions_of_BSA_PDEs_with_particular_asymptotic_properties} and Eq.~\eqref{eq:integral formula for fused ppf}).

\begin{lemma}\label{lem::rainbowSLE_mart}
Fix $\kappa\in (0,4], \, n\ge 1$, and $(\bs{\theta}, \vartheta) = (\theta^1, \ldots, \theta^{n}, \vartheta)\in\LX_{n+1}$. 
\begin{itemize}
\item Let $\eta^{j}$ be radial $\SLE_{\kappa}$ in $(\U; \ee^{\ii\theta^j}; 0)$ for each $1\le j\le n$, 
and let $\mathsf{P}_n$ be the probability measure on $\bs{\eta}=(\eta^{1}, \ldots, \eta^{n})$ under which the curves are independent. 

\item Let $\bs{\eta}=(\eta^{1}, \ldots, \eta^{n})\sim\QQfusion{n}$ be half-$n$-watermelon $\SLE_{\kappa}$ in $(\U; \ee^{\ii\bs{\theta}}, \ee^{\ii\vartheta})$. 
\end{itemize} 
Viewing it as an $n$-tuple of curves, we parameterize $\bs{\eta}$ using $n$-time-parameter $\bs{t}=(t_1, \ldots, t_n)$, 
and let $\bs{\mslitdriv}_{\bs{t}} = (\mslitdriv^{1}_{\bs{t}},\ldots,\mslitdriv^{n}_{\bs{t}})$ denote its multi-slit driving function~\eqref{eqn::Ntuple_driving}.
Then, the law of $\bs{\eta}$ under $\QQfusion{n}$ is the same as $\mathsf{P}_n$ tilted by the following $n$-time-parameter local martingale, up to the collision time: 
\begin{align*}
M_{\bs{t}}(\LZfusion{n}) 
= \; & \one_{\LE_{\emptyset}(\bs{\eta}_{\bs{t}})} \, \exp\Big(\frac{\mathfrak{c}}{2}\blm_{\bs{t}}\Big) \, 
(g_{\bs{t}}'(0))^{-n\tilde{\mathfrak{b}}} \, ( \covmap_{\bs{t}}'(\vartheta) )^{\mathfrak{h}_n}
\Big(\prod_{j=1}^{n} \covmap_{\bs{t},j}'(\xi_{t_j}^{j}) \Big)^{\mathfrak{b}}
\Big(\prod_{j=1}^{n} g_{\bs{t},j}'(0) \Big)^{\tilde{\mathfrak{b}}}
\,\LZfusion{n}(\bs{\mslitdriv}_{\bs{t}}, \covmap_{\bs{t}}(\vartheta)) ,
\end{align*}
where $\LE_{\emptyset}(\bs{\eta}_{\bs{t}}) = \{ \eta_{[0,t_j]}^{j} \cap \eta_{[0,t_i]}^{i}=\emptyset, \, \forall i\neq j\}$ is the event that different curves are disjoint, 
and $\blm_{\bs{t}}$ is the unique potential solving the exact differential equation~\eqref{eqn::mt_def} with $p=n$ \textnormal{(}see Lemma~\ref{lem::blm_exact}\textnormal{)}.
\end{lemma}

\begin{proof}
The fact that $M(\LZfusion{n})$ is an $n$-time-parameter local martingale can be checked using similar calculations as in the proof of Proposition~\ref{prop::multitime_mart_universal}.
Indeed, using the identities~\eqref{eqn::multitimemart_aux1}--\eqref{eqn::multitimemart_aux3} and
the radial Loewner equation~\eqref{eq:LE_multislith} for the multi-slit covering map $\covmap_{\bs{t}} = \covmap_{\bs{t},j} \circ \covmap_{t_j}^{j}$, which yields
\begin{align}\label{eqn::multitimemart_watermelon_aux1}
\frac{ \ud \covmap'_{\bs{t}} ( \vartheta ) }{\covmap'_{\bs{t}} ( \vartheta )} 
= -\frac{1}{2} \sum_{j=1}^{n}  \csc^2 \bigg(\frac{\covmap_{\bs{t}}(\vartheta)-\mslitdriv_{\bs{t}}^j}{2}\bigg) (\covmap_{\bs{t},j}' (\xi_{t_j}^{j}))^2 \, \ud t_j ,
\end{align}
after applying It\^{o}'s formula to $M_{\bs{t}}(\LZfusion{n})$ and observing that 
$\frac{n(2n+4-\kappa)}{2\kappa} = \mathfrak{h}_n$, 
we obtain
\begin{align*}
\frac{\ud M_{\bs{t}}(\LZfusion{n})}{M_{\bs{t}}(\LZfusion{n})}
= \; & -n \mathfrak{b} \, \frac{\ud g'_{\bs{t}}(0)}{g'_{\bs{t}}(0)} 
+ \sum_{j=1}^n \bigg( \mathfrak{b} \, \frac{\ud \covmap_{\bs{t},j}'(\xi_{t_j}^{j})}{\covmap_{\bs{t},j}'(\xi_{t_j}^{j})} + \frac{\kappa \mathfrak{b}(\mathfrak{b}-1)}{2} \, \Big( \frac{ \covmap_{\bs{t},j}''(\xi_{t_j}^{j})}{\covmap_{\bs{t},j}'(\xi_{t_j}^{j})} \Big)^2 \ud t_j \bigg) 
+ \tilde{\mathfrak{b}} \sum_{j=1}^n \frac{\ud g_{\bs{t},j}'(0)}{g_{\bs{t},j}'(0)} \\
\; & + 
\mathfrak{h}_n \frac{\ud \covmap'_{\bs{t}}(\vartheta)}{\covmap'_{\bs{t}}(\vartheta)}+ \frac{\mathfrak{c}}{2} \ud \blm_{\bs{t}} \\
\; & + \sum_{j=1}^n \bigg( \frac{\partial_j \LZfusion{n}(\bs{\mslitdriv}_{\bs{t}}, \covmap_{\bs{t}}(\vartheta))}{\LZfusion{n}(\bs{\mslitdriv}_{\bs{t}}, \covmap_{\bs{t}}(\vartheta))} \, \ud \bs{\mslitdriv}_{\bs{t}}^{j} 
+ \frac{\kappa}{2} \frac{\partial_j^2 \LZfusion{n}(\bs{\mslitdriv}_{\bs{t}}, \covmap_{\bs{t}}(\vartheta))}{\LZfusion{n}(\bs{\mslitdriv}_{\bs{t}}, \covmap_{\bs{t}}(\vartheta))} \, (\covmap_{\bs{t},j}'(\xi_{t_j}^{j}))^2 \ud t_j \bigg) \\
\; & + \frac{\partial_{n+1} \LZfusion{n}(\bs{\mslitdriv}_{\bs{t}}, \covmap_{\bs{t}}(\vartheta))}{\LZfusion{n}(\bs{\mslitdriv}_{\bs{t}}, \covmap_{\bs{t}}(\vartheta))} \, \ud \covmap_{\bs{t}}(\vartheta) 
+ \kappa \mathfrak{b} \sum_{j=1}^n \frac{\partial_j \LZfusion{n}(\bs{\mslitdriv}_{\bs{t}}, \covmap_{\bs{t}}(\vartheta))}{\LZfusion{n}(\bs{\mslitdriv}_{\bs{t}}, \covmap_{\bs{t}}(\vartheta))} \, \covmap_{\bs{t},j}''(\xi_{t_j}^{j}) \, \ud t_j .
\end{align*}
Using~\eqref{eqn::mt_def}, \eqref{eqn::multitimemart_aux1}--\eqref{eqn::multitimemart_aux3} and~\eqref{eqn::multitimemart_watermelon_aux1} and writing $a_{\bs{t}}^{j} = (\covmap_{\bs{t},j}' (\xi_{t_j}^{j}))^2$, we thus find that
\begin{align*}
\frac{\ud M_{\bs{t}}(\LZfusion{n})}{M_{\bs{t}}(\LZfusion{n})}
= \; & \sum_{j=1}^n \bigg( \frac{\partial_j \LZfusion{n}(\bs{\mslitdriv}_{\bs{t}}, \covmap_{\bs{t}}(\vartheta))}{\LZfusion{n}(\bs{\mslitdriv}_{\bs{t}}, \covmap_{\bs{t}}(\vartheta))} \covmap_{\bs{t},j}'(\xi_{t_j}^{j}) 
+ \mathfrak{b} \, \frac{\covmap_{\bs{t},j}''(\xi_{t_j}^{j})}{\covmap_{\bs{t},j}'(\xi_{t_j}^{j})} \bigg) \, \ud \xi_{t_j}^{j}  
\\
\; & + \frac{\kappa}{2} \sum_{j=1}^n 
\frac{\partial_j^2 \LZfusion{n}(\bs{\mslitdriv}_{\bs{t}}, \covmap_{\bs{t}}(\vartheta))}{\LZfusion{n}(\bs{\mslitdriv}_{\bs{t}}, \covmap_{\bs{t}}(\vartheta))} \, a_{\bs{t}}^{j} \, \ud t_j
\\
\; & + \sum_{j=1}^n 
\underset{1\leq i\neq j \leq n}{\sum} \bigg( \cot \bigg(\frac{\mslitdriv_{\bs{t}}^{i}-\mslitdriv_{\bs{t}}^{j}}{2}\bigg)
\frac{\partial_{i} \LZfusion{n}(\bs{\mslitdriv}_{\bs{t}}, \covmap_{\bs{t}}(\vartheta))}{\LZfusion{n}(\bs{\mslitdriv}_{\bs{t}}, \covmap_{\bs{t}}(\vartheta))} \, - \, \frac{\mathfrak{b}/2}{ \sin^2 \big(\frac{\theta^{i}-\theta^{j}}{2}\big)}\bigg)\, a_{\bs{t}}^{j} \, \ud t_j
\\
\; & + \sum_{j=1}^n \bigg( \cot \bigg(\frac{\covmap_{\bs{t}}(\vartheta)-\mslitdriv_{\bs{t}}^{j}}{2}\bigg)
\frac{\partial_{n+1} \LZfusion{n}(\bs{\mslitdriv}_{\bs{t}}, \covmap_{\bs{t}}(\vartheta))}{\LZfusion{n}(\bs{\mslitdriv}_{\bs{t}}, \covmap_{\bs{t}}(\vartheta))} \, - \, \frac{\mathfrak{h}_n/2}{ \sin^2 \big(\frac{\covmap_{\bs{t}}(\vartheta)-\mslitdriv_{\bs{t}}^{j}}{2}\big)} - \tilde{\mathfrak{b}}
\bigg)\,  a_{\bs{t}}^{j} \, \ud t_j
\\
= \; & \sum_{j=1}^n \bigg( \frac{\partial_j \LZfusion{n}(\bs{\mslitdriv}_{\bs{t}}, \covmap_{\bs{t}}(\vartheta))}{\LZfusion{n}(\bs{\mslitdriv}_{\bs{t}}, \covmap_{\bs{t}}(\vartheta))} \covmap_{\bs{t},j}'(\xi_{t_j}^{j}) 
+ \mathfrak{b} \, \frac{\covmap_{\bs{t},j}''(\xi_{t_j}^{j})}{\covmap_{\bs{t},j}'(\xi_{t_j}^{j})} \bigg) \, \ud \xi_{t_j}^{j}  .
&& \textnormal{[by~\eqref{eqn::radialBPZ_fusion}]}
\end{align*}
Thus, $M(\LZfusion{n})$ is $n$-time-parameter local martingale. 
It remains to show that  the probability measure $\PP(\LZfusion{n})$ obtained by tilting $\mathsf{P}_n$ by $M_{\bs{t}}(\LZfusion{n})$ is the same as $\QQfusion{n}$. 
Girsanov's theorem yields
\begin{align*}
\ud \bs{\mslitdriv}_{\bs{t}}^{j} 
= \; & (\kappa a_{\bs{t}}^{j})^{1/2} \, \ud \tilde{B}^{j}_{t_j} + \kappa \partial_j \log\LZfusion{n} (\bs{\mslitdriv}_{\bs{t}}, \covmap_{\bs{t}}(\vartheta))\, a_{\bs{t}}^{j} \, \ud t_j
- \underset{1\leq i\neq j \leq n}{\sum} \cot \bigg(\frac{\mslitdriv_{\bs{t}}^{i}-\mslitdriv_{\bs{t}}^{j}}{2}\bigg)\, a_{\bs{t}}^{i} \, \ud t_i
\\
= \; & (\kappa a_{\bs{t}}^{j})^{1/2} \, \ud \tilde{B}^{j}_{t_j} - \Big( \big( \tfrac{\kappa}{2} - n - 2 \big) \cot \bigg(\frac{\covmap_{\bs{t}}(\vartheta)-\mslitdriv_{\bs{t}}^{j}}{2}\bigg)
+ \underset{1\leq i\neq j \leq n}{\sum} \cot \bigg(\frac{\mslitdriv_{\bs{t}}^{i}-\mslitdriv_{\bs{t}}^{j}}{2}\bigg) \Big)\, a_{\bs{t}}^{j} \, \ud t_j
\\
\; & - \underset{1\leq i\neq j \leq n}{\sum} \cot \bigg(\frac{\mslitdriv_{\bs{t}}^{i}-\mslitdriv_{\bs{t}}^{j}}{2}\bigg)\, a_{\bs{t}}^{i} \, \ud t_i ,
\end{align*}
where $\tilde{B}^{1}_{t_1},\ldots,\tilde{B}^{n}_{t_n}$ are independent $\PP(\LZfusion{n})$-Brownian motions. 
Using the coordinate changes from~\cite{Schramm-Wilson:SLE_coordinate_changes} and accounting for the variation of the capacity parameterization, we see that under the measure $\PP(\LZfusion{n})$,
\begin{itemize}
\item the marginal law of $\eta^{1}_{[0,t_1]}$ is radial $\SLE_{\kappa}(2,\ldots,2,\kappa-4-2n)$ in 
$(\U; \ee^{\ii\bs{\theta}}, \ee^{\ii\vartheta};0)$, which is the same as the law of chordal $\SLE_{\kappa}(2,\ldots,2)$ in $\U$ from $\ee^{\ii\theta^1}$ to $\ee^{\ii\vartheta}$ with force points $(\ee^{\ii\theta^2}, \ldots, \ee^{\ii\theta^n})$;

\item for each $j\in \{1,2,\ldots, n-1\}$, conditionally on the curves $(\eta^{1}_{[0,t_1]},\ldots,\eta^{j}_{[0,t_j]})$, the law of the $(j+1)$:st curve $\eta^{j+1}_{[0,t_{j+1}]}$ is radial 
$\SLE_{\kappa}(2,\ldots,2,\kappa-4-2n,2,\ldots,2)$ in the polygon 
\begin{align*}
\big( \U\setminus \bigcup_{i=1}^j \eta^{i}_{[0,t_{i}]}; \ee^{\ii\theta^{j+1}},\ee^{\ii\theta^{j+2}}, \ldots, \ee^{\ii\theta^n}, \ee^{\ii\vartheta}, \eta_{t_1}^{1}, \ldots, \eta_{t_{j}}^{j};0 \big) ,
\end{align*}
which is the same as the law of chordal $\SLE_{\kappa}(2,\ldots,2)$ in $\U\setminus \cup_{i=1}^j \eta^{i}_{[0,t_{i}]}$ from $\ee^{\ii\theta^{j+1}}$ to 
$\ee^{\ii\vartheta}$ with force points $(\ee^{\ii\theta^{j+2}}, \ldots, \ee^{\ii\theta^n}, \eta_{t_1}^{1}, \ldots, \eta_{t_{j}}^{j})$. 
\end{itemize}
Since on the one hand, the law $\QQfusion{n}$ of half-$n$-watermelon $\SLE_{\kappa}$ in  $(\U; \ee^{\ii\bs{\theta}}, \ee^{\ii\vartheta})$ is unique, 
and on the other hand it coincides with the law of $\bs{\eta}$ under $\PP(\LZfusion{n})$, 
we conclude that $\PP(\LZfusion{n}) = \QQfusion{n}$. 
\end{proof}

We are now ready to finish the proof of Theorem~\ref{thm::multiradial_halfwatermelon}.
Item~\ref{item::multiradial_halfwatermelonPF} was already proved in Lemma~\ref{lem::LGtoLZfusion}, 
so it remains to prove Item~\ref{item::multiradial_halfwatermelonMEAS}: the convergence of multiradial $\SLE_\kappa^\mu$ to half-watermelon $\SLE_\kappa$. 

\begin{proof}[Proof of Theorem~\ref{thm::multiradial_halfwatermelon}]
By conformal invariance, it suffices to consider the polygon $(\Omega; \bs{x}, y) = (\U; \ee^{\ii\bs{\theta}},1)$ with angles $(\theta^1, \ldots, \theta^n)\in\LX_n$ such that 
$\arg(y)-2\pi = -2\pi<\theta^1<\theta^n < 0 = \arg(y)$ (see Eq.~\eqref{eqn::torus}). We keep writing $y=1=\ee^0$ for clarity. 

Combining Lemma~\ref{lem::rainbowSLE_mart} with Lemma~\ref{lem::multitime_mart_nradial_spiral_z_RN} for $p=n$,
we see that the law $\PPspiral{n}(\U; \bs{x}; z)$ of $n$-radial $\SLE_{\kappa}^\mu$ is the same as 
the law $\QQfusion{n}(\U; \bs{x}, y)$ of half-$n$-watermelon $\SLE_{\kappa}$ tilted by the local martingale
\begin{align*}
\; & N_{\bs{t}}^{\mu}(z;y) 
:= \frac{M_{\bs{t}}(\nradpartfn{n}{\mu}; z)}{M_{\bs{t}}(\LZfusion{n})} 
= \; N_{\bs{t}}(z;y) \, \frac{M_{\bs{t}}(\nradpartfn{n}{\mu}; z)}{M_{\bs{t}}(\nradpartfn{n}{0}; z)} ,
\\
\textnormal{where} \qquad
\; & N_{\bs{t}}(z;y) 
:= \frac{M_{\bs{t}}(\nradpartfn{n}{0}; z)}{M_{\bs{t}}(\LZfusion{n})} 
\, = \, \frac{|g_{\bs{t}}'(z)|^{\tilde{\mathfrak{b}}+\frac{n^2-1}{2\kappa}} \; \nradpartfn{n}{0}(\U; \ee^{\ii\bs{\mslitdriv}_{\bs{t}}}; g_{\bs{t}}(z))}{|g_{\bs{t}}'(y)|^{\mathfrak{h}_n} \; \LZfusion{n}(\U; \ee^{\ii\bs{\mslitdriv}_{\bs{t}}}; g_{\bs{t}}(y))}
\\
= \; & \frac{|g_{\bs{t}}'(z)|^{\tilde{\mathfrak{b}}+\frac{n^2-1}{2\kappa}}}{|g_{\bs{t}}'(y)|^{\mathfrak{h}_n}}
\; ( \CR(\U; g_{\bs{t}}(z)) )^{\mathfrak{h}_n - \tilde{\mathfrak{b}} - \frac{n^2-1}{2\kappa}} \; 
\bigg(\prod_{j=1}^n \frac{H(\U; \ee^{\ii\mslitdriv^j_{\bs{t}}}; g_{\bs{t}}(z))}{\CR(\U; g_{\bs{t}}(z)) \, H(\U; \ee^{\ii\mslitdriv^j_{\bs{t}}}, g_{\bs{t}}(y))}\bigg)^{\frac{1}{n} \mathfrak{h}_n} 
&& \textnormal{[by~\eqref{eqn::LGandLZfusion}]}
\\
= \; & \; ( \CR(\U\setminus\bs{\eta}_{\bs{t}}; z) )^{\mathfrak{h}_n - \tilde{\mathfrak{b}} - \frac{n^2-1}{2\kappa}} \; 
\bigg(\prod_{j=1}^n \underbrace{\frac{|g_{\bs{t}}'(z)|}{|g_{\bs{t}}'(y)|} \, \frac{H(\U; \ee^{\ii\mslitdriv^j_{\bs{t}}}; g_{\bs{t}}(z))}{\CR(\U; g_{\bs{t}}(z))\, H(\U; \ee^{\ii\mslitdriv^j_{\bs{t}}}, g_{\bs{t}}(y))}}_{=: \, Z_{\bs{t}}^{j}(z;y)} \bigg)^{\frac{1}{n} \mathfrak{h}_n} ,
&& \textnormal{[by~\eqref{eq:CR_COV}]}
\end{align*} 
and writing 
\begin{align*}
I_{\bs{t}}^{j}(z) := \mslitdriv_{\bs{t}}^{j} -\arg g_{\bs{t}}'(z) + 2 \arg (1-g_{\bs{t}}(z) \ee^{-\ii \mslitdriv_{\bs{t}}^j} \big) , \quad j \in \{1,\ldots,n\} ,
\end{align*} 
and using~(\ref{eqn::nradpartfn_mu_z},~\ref{eqn::multitime_mart_nradial_spiral_z}) for $p=n$ we have
\begin{align*}
\frac{M_{\bs{t}}(\nradpartfn{n}{\mu}; z)}{M_{\bs{t}}(\nradpartfn{n}{0}; z)} 
= \; & \exp \Big( \frac{\mu}{\kappa}\sum_{j=1}^p I_{\bs{t}}^{j}(z) \Big)\,
\bigg( \frac{1-|g_{\bs{t}}(z)|^2}{|g_{\bs{t}}'(z)|} \bigg)^{\frac{\mu^2}{2\kappa}}  
\\
= \; & \exp \Big( \frac{\mu}{\kappa}\sum_{j=1}^p I_{\bs{t}}^{j}(z) \Big)\,
(\CR(\U\setminus\bs{\eta}_{\bs{t}}; z))^{\frac{\mu^2}{2\kappa}} .
&& \textnormal{[by~\eqref{eq:CR_COV}]}
\end{align*}

{\bf Step~1.}
We will first show that almost surely (a.s.), the associated Radon-Nikodym derivative 
\begin{align}\label{eqn::multiradial_spiral_halfwatermelon_aux0}
\frac{\ud \PPspiral{n}(\U; \bs{x}; z)}{\ud \QQfusion{n}(\U; \bs{x}, y)} (\bs{\eta}_{\bs{t}})
\; = \; \frac{N_{\bs{t}}^{\mu}(z;y) }{N_{\bs{0}}^{\mu}(z;y) } 
\; = \; \frac{N_{\bs{t}}(z;y)}{N_{\bs{0}}(z;y)}
 \exp\Big( \frac{\mu}{\kappa} \sum_{j=1}^n \big( I_{\bs{t}}^{j}(z)-I_{\bs{0}}^{j}(z) \big) \Big)\bigg(\frac{\CR(\U\setminus\bs{\eta}_{\bs{t}}; z)}{\CR(\U; z)}\bigg)^{\frac{\mu^2}{2\kappa}} 
\end{align}
converges to $1$ as $z \to y$, for any $\bs{t}$ before the exit time from $U$.
Lemmas~\ref{lem::CR_cvg}~\&~\ref{lem::Poissonkernel_cvg} yield the a.s.~limits
\begin{align*}
\lim_{z\to y}\frac{\CR(\U\setminus\bs{\eta}_{\bs{t}}; z)}{\CR(\U; z)}=1,
\qquad  \lim_{z\to y}Z_{\bs{0}}^j(z;y)=1, 
\qquad \textnormal{and} \qquad 
\lim_{z\to y} Z_{\bs{t}}^j(z;y)=1 , \quad j \in \{1,\ldots,n\} ,
\end{align*}
which readily implies that the part without spiral converges:
\begin{align}\label{eqn::multiradial_halfwatermelon_RN}
\frac{N_{\bs{t}}(z;y) }{N_{\bs{0}}(z;y) } 
= \Big(\frac{\CR(\U\setminus\bs{\eta}_{\bs{t}}; z)}{\CR(\U; z)}\Big)^{\mathfrak{h}_n - \tilde{\mathfrak{b}} - \frac{n^2-1}{2\kappa}} \; 
\Big(\prod_{j=1}^n \frac{Z_{\bs{t}}^{j}(z;y)}{Z_{\bs{0}}^{j}(z;y)} \Big)^{\frac{1}{n} \mathfrak{h}_n} 
\quad \overset{z\to y}{\longrightarrow} \quad 1 ,
\end{align}
for any $\bs{t}$ before the exit time from $U$. 
The last factor in the product in~\eqref{eqn::multiradial_spiral_halfwatermelon_aux0} converges similarly. 
In the exponential in~\eqref{eqn::multiradial_spiral_halfwatermelon_aux0}, 
for any $\bs{t}$ before the exit time from $U$, we have for each $j$ the a.s.~limit 
\begin{align}\label{eq:limit_of_arg_I}
I_{\bs{t}}^{j}(z) 
\quad \overset{z \to y = 1}{\longrightarrow} \quad 
I_{\bs{t}}^{j}(1) 
= \; & \mslitdriv_{\bs{t}}^{j} - \arg g_{\bs{t}}'(1) + 2 \arg (1-g_{\bs{t}}(1) \ee^{-\ii \mslitdriv_{\bs{t}}^j} \big)
\; = \; -\pi ,
\end{align}
writing $g_{\bs{t}}(1) = \exp( \ii \covmap_{\bs{t}}(0) )$ and 
$g_{\bs{t}}'(1) = g_{\bs{t}}(1) \, \covmap_{\bs{t}}'(0) = \exp( \ii \covmap_{\bs{t}}(0) ) \, \covmap_{\bs{t}}'(0)$, so
\begin{align}\label{eq:arg_g}
\arg g_{\bs{t}}(1) = \covmap_{\bs{t}}(0)
\qquad \textnormal{and} \qquad 
\arg g_{\bs{t}}'(1) = \covmap_{\bs{t}}(0) .
\end{align}
Hence, the exponential in~\eqref{eqn::multiradial_spiral_halfwatermelon_aux0} also converges a.s.~to $1$ as $z \to y=1$. 

\medskip

{\bf Step~2.}
We show that in fact, for any $\bs{t}$ before the exit time from $U$, 
the Radon-Nikodym derivative $N_{\bs{t}}^{\mu}(z;y)/N_{\bs{0}}^{\mu}(z;y)$ stays uniformly bounded as $z\to y$. 
To this end, since $\CR(\U\setminus\bs{\eta}_{\bs{t}}; z)\le \CR(\U; z)$, 
\begin{align}\label{eqn::muliradial_halfwatermelon_RN_control1}
\bigg(\frac{\CR(\U\setminus\bs{\eta}_{\bs{t}}; z)}{\CR(\U; z)}\bigg)^{\frac{\mu^2}{2\kappa}} \leq 1
\qquad \textnormal{and} \qquad
0 \le \frac{N_{\bs{t}}(z;y) }{N_{\bs{0}}(z;y) } \le \Big(\prod_{j=1}^n \frac{Z_{\bs{t}}^{j}(z;y)}{Z_{\bs{0}}^{j}(z;y)}\Big)^{\frac{1}{n} \mathfrak{h}_n} .
\end{align}
Explicitly, using Poisson kernels and conformal radius in the unit disc, $\smash{\frac{H(\U; z, x)}{\CR(\U; z)H(\U; y, x)}=\frac{|y-x|^2}{|z-x|^2}}$, so
\begin{align*}
Z_{\bs{t}}^{j}(z;y)
= \frac{|g_{\bs{t}}'(z)|}{|g_{\bs{t}}'(y)|} \; \frac{|g_{\bs{t}}(y) - \ee^{\ii\mslitdriv^j_{\bs{t}}}|^2}{|g_{\bs{t}}(z) - \ee^{\ii\mslitdriv^j_{\bs{t}}}|^2} , \qquad j \in \{1,\ldots,n\} .
\end{align*}

Recall now that $U$ has a positive distance from $y$. 
Pick $\delta>0$ so that $B(y,16\delta)\cap U=\emptyset$. 
We will use the following estimates, which involve some universal constants (also independent of the number of curves) indicated in the notation ``$\lesssim$'': 
\begin{enumerate}
\item By Schwarz reflection, $g_{\bs{t}}$ is a conformal map in a neighborhood of $y$ (which depends on $U$), 
for any $\bs{t}$ before the exit time from $U$. 
Thus, Koebe's one-quarter theorem (e.g.~\cite[Theorem~2.3]{Duren:Univalent_functions}) 
guarantees that $B(g_{\bs{t}}(y), 4\delta|g_{\bs{t}}'(y)|) \subset g_{\bs{t}}(B(y,16\delta))$, and in particular, 
\begin{align}\label{eqn::muliradial_halfwatermelon_RN_aux1}
4\delta|g_{\bs{t}}'(y)|\le |g_{\bs{t}}(y) - \ee^{\ii\mslitdriv^j_{\bs{t}}}| , \qquad j \in \{1,\ldots,n\} .
\end{align}

\item The growth theorem (e.g.~\cite[Theorem~2.6]{Duren:Univalent_functions}) 
implies that
\begin{align}\label{eqn::muliradial_halfwatermelon_RN_aux2}
|g_{\bs{t}}(z)-g_{\bs{t}}(y)|\le \frac{256}{225} \, |g_{\bs{t}}'(y)| \, |z-y|, \quad\textnormal{for any }z\in B(y,\delta).
\end{align}

\item The distortion theorem (e.g.~\cite[Theorem~2.5]{Duren:Univalent_functions})
implies that
\begin{align}\label{eqn::muliradial_halfwatermelon_RN_aux3}
|g_{\bs{t}}'(z)| \lesssim |g_{\bs{t}}'(y)|, \quad\textnormal{for any }z\in B(y,\delta).
\end{align}
\end{enumerate}
We can thus estimate $\smash{Z_{\bs{t}}^{j}(z;y)}$ for each $j \in \{1,\ldots,n\}$ and 
for any $z\in B(y,\delta)$ as 
\begin{align} 
Z_{\bs{t}}^{j}(z;y)
\lesssim \; &  \frac{|g_{\bs{t}}(y) - \ee^{\ii\mslitdriv^j_{\bs{t}}}|^2}{|g_{\bs{t}}(z) - \ee^{\ii\mslitdriv^j_{\bs{t}}}|^2} 
&& \textnormal{[by~\eqref{eqn::muliradial_halfwatermelon_RN_aux3}]} \notag
\\
\lesssim \; & 
\bigg( 1 - \frac{256}{225} \frac{|g_{\bs{t}}'(y)| \, |z-y|}{|g_{\bs{t}}(y) - \ee^{\ii\mslitdriv^j_{\bs{t}}}|} \bigg)^{-2}
&& \textnormal{[by~\eqref{eqn::muliradial_halfwatermelon_RN_aux2}]} \notag
\\
\lesssim \; & 1 .
&& \textnormal{[by~\eqref{eqn::muliradial_halfwatermelon_RN_aux1}, since $|z-y| < \delta$]} \label{eqn::muliradial_halfwatermelon_RN_control2}
\end{align}
This implies that $N_{\bs{t}}(z;y) /N_{\bs{0}}(z;y)$ stays uniformly bounded as $z \to y$ for any $\bs{t}$ before the exit time from~$U$. 
It remains to control the exponential in~\eqref{eqn::multiradial_spiral_halfwatermelon_aux0}. 
To this end, consider the univalent map $f_{\bs{t}} \colon \U \to \C$ with $f(0)=0$ and $f'(0)=1$ 
defined as 
\begin{align*}
f_{\bs{t}}(w) := \frac{g_{\bs{t}}(\delta w+1)-g_{\bs{t}}(1)}{\delta g_{\bs{t}}'(1)}, \qquad w\in \U.
\end{align*}
By~\cite[Theorem~3.5]{Duren:Univalent_functions} and the rotation theorem~\cite[Theorem~3.7]{Duren:Univalent_functions}, we have
\begin{align*}
\Big| \arg \Big(\frac{f_{\bs{t}}(w)}{w}\Big) \Big| \le \; & \log \bigg( \frac{1+\tfrac{1}{100}}{1-\tfrac{1}{100}} \bigg) < 1 , 
\\
| \arg f_{\bs{t}}'(w) | \le \; & 4\arcsin(\tfrac{1}{100}) <1 , \qquad  \textnormal{for any } w\in B(0,\tfrac{1}{100}),
\end{align*}
which implies that for any point $z=1+\delta w\in B(1,\delta/100)$, we have
\begin{align} \label{eqn::muliradial_halfwatermelon_RN_aux4}
|\arg(g_{\bs{t}}(z)-g_{\bs{t}}(1))-\arg g_{\bs{t}}'(1)-\arg (z-1)| <1 
\qquad \textnormal{and} \qquad 
|\arg g_{\bs{t}}'(z)-\arg g_{\bs{t}}'(1)|<1.
\end{align}
Using~\eqref{eq:arg_g}, we see that
\begin{align} 
\label{eqn::muliradial_halfwatermelon_RN_aux5}
|\mslitdriv_{\bs{t}}^j- \arg g_{\bs{t}}'(1)| < \; &  2\pi , \\
\label{eqn::muliradial_halfwatermelon_RN_aux6}
|\arg (\ee^{\ii \mslitdriv_{\bs{t}}^j} - g_{\bs{t}}(z)) - \arg (g_{\bs{t}}(1) - g_{\bs{t}}(z))| < \; & 2\pi.
\end{align}
Combining these estimates, we obtain
\begin{align} 
\nonumber 
|I_{\bs{t}}^j(z)| = \; & \big| \mslitdriv_{\bs{t}}^j + \arg g_{\bs{t}}'(z) - 2\arg(\ee^{\ii\mslitdriv_{\bs{t}}^j}-g_{\bs{t}}(z)) \big| \\
\nonumber
\le \; & \underbrace{|\mslitdriv_{\bs{t}}^j-\arg g_{\bs{t}}'(1)|}_{\le \; 2\pi {\textnormal{ [by~\eqref{eqn::muliradial_halfwatermelon_RN_aux5}]}}} 
\; + \; \underbrace{|\arg g_{\bs{t}}'(z)-\arg g_{\bs{t}}'(1)|}_{\le \; 1 {\textnormal{ [by~\eqref{eqn::muliradial_halfwatermelon_RN_aux4}]}}} 
\; + \; 2\underbrace{ |\arg g_{\bs{t}}'(1)- \arg (g_{\bs{t}}(z)-g_{\bs{t}}(1))|}_{\le \; 1+2\pi \textnormal{ [by~\eqref{eqn::muliradial_halfwatermelon_RN_aux4}]}} \\
\nonumber 
& \; + 2 \underbrace{|\arg (g_{\bs{t}}(z)-g_{\bs{t}}(1))-\arg (g_{\bs{t}}(1)-g_{\bs{t}}(z))|}_{= \; \pi} 
\; + \; 2 \underbrace{|\arg (g_{\bs{t}}(1)-g_{\bs{t}}(z))-\arg(\ee^{\ii\mslitdriv_{\bs{t}}^j}-g_{\bs{t}}(z))|}_{< \; 2\pi \textnormal{ [by~\eqref{eqn::muliradial_halfwatermelon_RN_aux6}]}} \\
\label{eqn::muliradial_halfwatermelon_RN_control3}
\le \; & 12\pi + 3, \quad \textnormal{for any } z\in B(1,\delta/100).
\end{align}
Plugging~(\ref{eqn::muliradial_halfwatermelon_RN_control1},~\ref{eqn::muliradial_halfwatermelon_RN_control2},~\ref{eqn::muliradial_halfwatermelon_RN_control3}) into~\eqref{eqn::multiradial_spiral_halfwatermelon_aux0}, 
we conclude that the Radon-nikodym derivative 
$N^{\mu}_{\bs{t}}(z;y) /N_{\bs{0}}^{\mu}(z;y)$ stays uniformly bounded as $z \to y=1$ for any $\bs{t}$ before the exit time from~$U$.
This concludes the proof, since the above two steps readily imply that $\PPspiral{n}(\U; \bs{x}; z)$ converges to $\QQfusion{n}(\U; \bs{x}, y)$ for any such $\bs{t}$. 
\end{proof}

\subsection{Fusion of rainbow partition functions}
\label{subsec: Proof of limit}

Let us now return to investigating multichordal SLEs with rainbow connectivity (Definition~\ref{def::globalnSLE}). 
The chordal Loewner evolution of curves under $\QQrainbow{n}$ can be characterized in terms of the rainbow partition functions $\{\LZrainbow{n}\}_{n\ge 0}$,
\begin{align*}
\LZrainbow{n}(\HH; \cdot) \colon \{ \bs u = (u^1,\ldots,u^{2n})\in\R^{2n} \,|\, u^1<\cdots<u^{2n} \} \to (0,\infty) ,
\end{align*}
which are defined recursively 
via the following four properties, motivated by CFT (and which uniquely determine them)~\cite{Flores-Kleban:Solution_space_for_system_of_null-state_PDE2, Flores-Kleban:Solution_space_for_system_of_null-state_PDE3, Kytola-Peltola:Pure_partition_functions_of_multiple_SLEs, Peltola-Wu:Global_and_local_multiple_SLEs_and_connection_probabilities_for_level_lines_of_GFF, Wu:Convergence_of_the_critical_planar_ising_interfaces_to_hypergeometric_SLE, AHSY:Conformal_welding_of_quantum_disks_and_multiple_SLE_the_non-simple_case, FLPW:Multiple_SLEs_Coulomb_gas_integrals_and_pure_partition_functions}:
\begin{itemize}
\item[(PDE)] {\bf\emph{Chordal BPZ equations}}\textnormal{:} 
\begin{align}\label{eqn::chordalBPZ}
\bigg(
\frac{\kappa}{2} \partial_j^2
+  \underset{1\leq i\neq j \leq 2n}{\sum} \,  \bigg( \frac{2\partial_{i}}{u^{i}-u^{j}}
- \frac{2\mathfrak{b}}{(u^{i}-u^{j})^{2}} \bigg) \bigg)
\LZrainbow{n}(\HH; \bs u) =  0 , \qquad \textnormal{for all }j\in \{1,\ldots,2n\} .
\end{align}

\item[(COV)] {\bf\emph{M\"{o}bius covariance}}\textnormal{:} 
for all M\"obius maps $\varphi \colon \HH \to \HH$ such that $\varphi(u^{1}) < \cdots < \varphi(u^{2n})$, we have
\begin{align}\label{eqn::chordalCOV}
\LZrainbow{n}(\HH; u^{1},\ldots,u^{2n}) = 
\Big( \prod_{j=1}^{2n} \varphi'(u^{j}) \Big)^{\mathfrak{b}} 
\, \LZrainbow{n}(\HH; \varphi(u^{1}),\ldots,\varphi(u^{2n})).
\end{align}

\item[(ASY)] {\bf\emph{Asymptotics}}\textnormal{:} 
with $\LZ_{\emptyset} \equiv 1$ for the empty link pattern $\emptyset \in \LP_0$, the collection $\{\LZrainbow{n}\}_{n\ge 0} $ satisfies the following recursive asymptotics property. Fix $n \ge 1$ and $j \in \{1,2, \ldots, 2n-1 \}$. 
Then, we have
\begin{align}\label{eqn::chordalrainbowASY}
\lim_{u^j,u^{j+1}\to\xi} \frac{\LZrainbow{n}(\HH; u^1,\ldots, u^{2n})}{ (u^{j+1}-u^j)^{-2 \mathfrak{b} } }
= 
\begin{cases}
\LZrainbow{n-1}(\HH; u^1, \ldots, u^{j-1}, u^{j+2}, \ldots, u^{2n}), 
& \quad j = n, \\
0 ,
& \quad j \neq n ,
\end{cases}
\end{align}
where $\xi \in (u^{j-1}, u^{j+2})$ (with the convention that $u^0 = -\infty$ and $u^{2n+1} = +\infty$). 

\item[(PLB)] {\bf\emph{Power-law bound}}\textnormal{:} 
there exist constants $C>0$ and $r>0$ such that for all $n \geq 1$, we have
\begin{align}\label{eqn::PPF_PLB_weak}
\LZrainbow{n}(\HH; u^1,\ldots,u^{2n}) \le  \; & C\prod_{1\le i<j\le 2n}(u^j-u^i)^{\nu^{ij}(r)}, 
\qquad \textnormal{for all } u^1<\cdots<u^{2n} ,
\\
\nonumber
\textnormal{where } \quad
\nu^{ij}(r) := \; &
\begin{cases}
r , & \textnormal{if }|u^j-u^i|>1,\\
-r , & \textnormal{if }|u^j-u^i|\le 1.
\end{cases}
\end{align}
\end{itemize}

The goal of this section is to prove that $\LZrainbow{n}$ converges to $\LZfusion{n}$ 
after proper normalization. 
To this end, we utilize the axiomatic definition of $\LZrainbow{n}$
via the properties PDE~\eqref{eqn::chordalBPZ}, COV~\eqref{eqn::chordalCOV}, ASY~\eqref{eqn::chordalrainbowASY}, and PLB~\eqref{eqn::PPF_PLB_weak} (crucially relying on the uniqueness result~\cite[Lemma~1]{Flores-Kleban:Solution_space_for_system_of_null-state_PDE2}). 
We will also need the Coulomb gas integral description of the partition functions developed in~\cite{Kytola-Peltola:Pure_partition_functions_of_multiple_SLEs, Kytola-Peltola:Conformally_covariant_boundary_correlation_functions_with_quantum_group, Peltola:Basis_for_solutions_of_BSA_PDEs_with_particular_asymptotic_properties}.

\begin{proposition} \label{prop: block limit}
Fix $\kappa \in (0,8)$. 
Write $\bs x = (x^1,\ldots,x^{n})$ and $\bs y = (y^{n},\ldots,y^1)$. For any $y > x^n$, we~have
\begin{align}\label{eq: block limit}
\lim_{\substack{y^j\to y\\\forall 1\le j\le n}}
\frac{\LZrainbow{n}(\HH;\bs x, \bs y)}{\LZ_{\defpatt_n}(\HH;\bs y)} 
= A_n \, \LZfusion{n}(\HH;\bs x, y) ,
\end{align}
where $\LZ_{\defpatt_n}$ and $\LZfusion{n}$ are defined in~\eqref{eqn::LZfusion_H} and the multiplicative constant $A_n$ is defined in~\eqref{eq: fusion cst}.
\end{proposition}

\begin{figure}[H]
\begin{center}
\includegraphics[width=.75\textwidth]{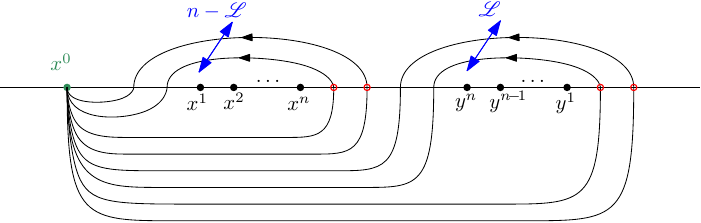}
\end{center}
\caption{\label{fig:Gammal}
Illustration of the integration contours $\contour_\nnofloops$ in Lemma~\ref{lem: rainbow ppf}. 
The red circles indicate the branch choice of the integrand function $\bs w \mapsto f^{(\nnofloops)}(\bs x, \bs y;\bs w)$
defined in~\eqref{eq: integrand} for fixed $x^1 < \cdots < x^{n} < y^{n} < \cdots < y^1$, 
so that it is real and positive when 
$x^{n} < \Re(w^1) < \cdots < \Re(w^{n-\nnofloops}) < y^{n}$ and
$y^1 < \Re(w^{n-\nnofloops+1}) < \cdots < \Re(w^n)$. 
By~\cite[Theorem~4.17]{Kytola-Peltola:Conformally_covariant_boundary_correlation_functions_with_quantum_group}, 
while the integration $\contour_\nnofloops$ depends on an auxiliary variable $x^0 < x^1$, 
the function defined in Eq.~\eqref{eq: rainbow ppf} 
as a linear combination of integrals involving $\contour_\nnofloops$, for $\nnofloops = 0,1,\ldots,n$, 
is independent of the choice of $x^0$.}
\end{figure}

In the upper half-plane $\HH$, we have the following explicit formula for the rainbow partition function.

\begin{lemma}\label{lem: rainbow ppf}
For $\kappa \in (0,8)$, we have
\begin{align} \label{eq: rainbow ppf}
\LZrainbow{n}(\HH;\bs x, \bs y)
= C_n \sum_{\nnofloops=0}^{n}(-1)^{\nnofloops} q^{\nnofloops(n-\nnofloops-1)}
\int_{\contour_\nnofloops} f^{(\nnofloops)}(\bs x, \bs y;\bs w) \, \ud \bs{w} 
, \qquad \bs w = (w^1,\ldots,w^{\nnofloops}) ,
\end{align}
where both sides are continuous in $\kappa$, the multiplicative constant is\footnote{Observe that the constant $C_n$ has divergences, which cancel out in a nontrivial manner with zeroes of the integral.} 
\begin{align*}
C_n = C_n(\kappa) := \frac{1}{(q^{-2}-1)^{n}}\frac{\qnum{2}^{n}}{\qfact{n+1}}  \, 
\Big( \frac{\Gamma(2-8/\kappa)}{\Gamma(1-4/\kappa)^2} \Big)^n , 
\qquad q = \ee^{4 \pi \ii / \kappa} ,
\end{align*}
and where each $\contour_\nnofloops \ni \bs w$ is an integration contour depicted in Figure~\ref{fig:Gammal}, and 
the integrand $f^{(\nnofloops)}(\bs x, \bs y;\bs w)$ associated to it is a branch of the multivalued function
\begin{align} \label{eq: integrand}
\begin{split}
f(\bs x, \bs y;\bs w)
:=\; & 
\Big( \prod_{1\leq i<j\leq n}(x^{j}-x^{i}) \Big)^{2/\kappa}
\Big( \prod_{1\leq i<j\leq n}(y^{j}-y^{i}) \Big)^{2/\kappa}
\Big( \prod_{1\leq i,j\leq n}(y^{j}-x^{i}) \Big)^{2/\kappa}
\\
\; & \times
\Big( \prod_{\substack{1\leq i\leq n\\1\leq r\leq n}}(w^{r}-x^{i}) \Big)^{-4/\kappa}
\Big( \prod_{\substack{1\leq i\leq n\\1\leq r\leq n}} (w^{r}-y^{i}) \Big)^{-4/\kappa}
\Big( \prod_{1\leq r<s\leq n}(w^{s}-w^{r}) \Big)^{8/\kappa} 
\end{split}
\end{align}
chosen to be real and positive at the point indicated by the red circles in Figure~\ref{fig:Gammal}. 
\end{lemma}

\begin{proof}
Let us first observe that the constant $C_n$ and each product factor in it has poles only at finitely many rational values (and the number of the poles depends on $n$);
we denote by $Q_n \subset \QQ$ the set of all of these singularities. 
For fixed $\bs x, \bs y$, the left-hand side of the asserted identity~\eqref{eq: rainbow ppf} is continuous in $\kappa \in (0,8)$ by~\cite[Theorem~1.10]{FLPW:Multiple_SLEs_Coulomb_gas_integrals_and_pure_partition_functions},
while the right-hand side of~\eqref{eq: rainbow ppf} is analytic in $\kappa \in \C \setminus Q_n$.  
It~therefore suffices to prove that 
\begin{align} \label{eq: rainbow ppf nonrat}
\LZrainbow{n}(\HH;\bs x, \bs y)
= \Big( \frac{\Gamma(2-8/\kappa)}{\Gamma(1-4/\kappa)^2} \Big)^n \, \Frainbow{n}(\bs x, \bs y) , \qquad  \kappa\in(0,8)\setminus Q_n, 
\end{align}
where\footnote{Note that the constant $\frac{\Gamma(2-8/\kappa)}{\Gamma(1-4/\kappa)^2} = \frac{1}{\Selberg_{1}}$ is a reciprocal of a Selberg integral~\eqref{eq: Selberg}.}
\begin{align*} 
\Frainbow{n}(\bs x, \bs y)
:= \frac{1}{(q^{-2}-1)^{n}}\frac{\qnum{2}^{n}}{\qfact{n+1}}\,
\sum_{\nnofloops=0}^{n}(-1)^{\nnofloops} q^{\nnofloops(n-\nnofloops-1)}
\int_{\contour_\nnofloops} f^{(\nnofloops)}(\bs x, \bs y;\bs w) \, \ud \bs{w} .
\end{align*}
We can use the well-known uniqueness property,~\cite[Lemma~1]{Flores-Kleban:Solution_space_for_system_of_null-state_PDE2}: 
in order to verify the equality~\eqref{eq: rainbow ppf nonrat}, it suffices to check that the left/right hand sides of~\eqref{eq: rainbow ppf nonrat} both 
satisfy PDE~\eqref{eqn::chordalBPZ}, COV~\eqref{eqn::chordalCOV}, ASY~\eqref{eqn::chordalrainbowASY}, and PLB~\eqref{eqn::PPF_PLB_weak}. 
On the one hand, $\{\LZrainbow{n}\}_{n\ge 0}$ satisfy these properties by definition.
On the other hand, 
let us summarize why the functions appearing on the right-hand side of~\eqref{eq: rainbow ppf nonrat} satisfy the same properties.
First of all,~\cite[Section~4.1 with {$d_j=2$} for all {$j$}]{Kytola-Peltola:Conformally_covariant_boundary_correlation_functions_with_quantum_group} 
gives a linear map from a certain complex vector space
(namely, a tensor product representation $\mathsf{M}^{\otimes (2n)}$ of two-dimensional type-one simple modules $\mathsf{M} = \C^2$ for the quantum group $U_q(\mathfrak{sl}_2(\C))$, with $q = \exp(4\pi\ii/\kappa)$), 
to a certain space of smooth functions of the points $(\bs x, \bs y)$ (such that to each point, a copy of the simple $U_q(\mathfrak{sl}_2(\C))$-module $\mathsf{M}$ is attached),
which satisfy PLB~\eqref{eqn::PPF_PLB_weak} by construction. Under this linear map, the image functions are obtained as Coulomb gas type contour integrals as the one on the right-hand side of~\eqref{eq: rainbow ppf nonrat},
and they satisfy PDE~\eqref{eqn::chordalBPZ} and COV~\eqref{eqn::chordalCOV} by~\cite[Theorem~4.17~(PDE, COV) with {$d_j=2$} for all {$j$}, and {$d=1$}]{Kytola-Peltola:Conformally_covariant_boundary_correlation_functions_with_quantum_group}.
Moreover, the properties ASY~\eqref{eqn::chordalrainbowASY} are obtained by combining~\cite[Proposition~3.3]{Kytola-Peltola:Pure_partition_functions_of_multiple_SLEs}, which gives specific projection properties for the vectors in the $U_q(\mathfrak{sl}_2(\C))$-module
corresponding to the right-hand side of~\eqref{eq: rainbow ppf nonrat} (with the specific constants chosen as in~\eqref{eq: rainbow ppf nonrat}), 
with~\cite[Theorem~4.17~(ASY) with {$d_j=2$} for all {$j$}, and {$d=1$}]{Kytola-Peltola:Conformally_covariant_boundary_correlation_functions_with_quantum_group}, 
which says that these projection properties become the desired asymptotic properties for the corresponding functions.
The asserted formula~\eqref{eq: rainbow ppf} now follows.
\end{proof}

\begin{remark}
It is not trivial to see that the right-hand side of~\eqref{eq: rainbow ppf} is real-valued. 
While the $q$-integers are real-valued (and can be written in terms of trigonometric functions),
the factor $(q^{-2}-1)^{n}$ is obviously not real-valued in general, neither are the coefficients in the sum over $\nnofloops$.
However, as the integrand function~\eqref{eq: integrand} is complex-valued (and multivalued), 
the phases from these factors cancel out, rendering the right-hand side of~\eqref{eq: rainbow ppf} real. 
(In fact, it is moreover positive, as the left-hand side is.)
\end{remark}

\begin{example}
When $n=1$, we have $\qnum{2} = q+q^{-1} = 2 \cos(4 \pi/\kappa)$ for $q = \ee^{4 \pi \ii / \kappa}$, so we find
\begin{align*}
\LZrainbow{n}(\HH;x,y)
= \; & \frac{q}{(q^{-1}-q)}\,
\frac{\Gamma(2-8/\kappa)}{\Gamma(1-4/\kappa)^2} \, 
\bigg( 
\int_{\contour_0} f^{(0)}(x,y;w) \, \ud w
\; - \; q^{-1} \int_{\contour_1} f^{(1)}(x,y;w) \, \ud w
\bigg) 
\\
= \; & \frac{1}{(q-q^{-1})}\,
\frac{\Gamma(2-8/\kappa)}{\Gamma(1-4/\kappa)^2} \, 
(q-q^{-1}) \, 
\int_{x}^{y} |f(x,y;w)| \, \ud w
\\
= \; & \frac{\Gamma(2-8/\kappa)}{\Gamma(1-4/\kappa)^2} \, 
\int_{x}^{y} |f(x,y;w)| \, \ud w
\; = \; (y-x)^{-2\mathfrak{b}} ,
\qquad \mathfrak{b} = \frac{6-\kappa}{2\kappa} ,
\end{align*}
by evaluating the integral via contour deformation (keeping careful track of phase choices) as in~\cite{Kytola-Peltola:Conformally_covariant_boundary_correlation_functions_with_quantum_group}.
The integral on the last line can be evaluated in terms of Euler's Beta function, that also equals a ratio of Gamma functions. 
It should be understood in the sense of Pochhammer contours for $\kappa \leq 4$, 
when it is not convergent as a line integral (see~\cite{FLPW:Multiple_SLEs_Coulomb_gas_integrals_and_pure_partition_functions}). 
Note that both sides of the end result are continuous in $\kappa \in (0,8)$, but the constant $\frac{\Gamma(2-8/\kappa)}{\Gamma(1-4/\kappa)^2}$ has divergences in the interval $\kappa \in (0,4)$, which cancel out with the effect of changing the line integral to a Pochhammer integral.
\end{example}

\begin{proof}[Proof of Proposition~\ref{prop: block limit}]
Using the formula~\eqref{eq: rainbow ppf} from Lemma~\ref{lem: rainbow ppf}, we find that for each $y > x^n$, 
\begin{align} \label{eq:integral formula for fused ppf}
\begin{split}
\lim_{\substack{y^j\to y\\\forall 1\le j\le n}} 
\frac{\LZrainbow{n}(\bs x, \bs y)}{\LZ_{\defpatt_n}(\HH;\bs y)} 
= \; & C_n \sum_{\nnofloops=0}^{n}(-1)^{\nnofloops} q^{\nnofloops(n-\nnofloops-1)}\int_{\hat{\contour}_\nnofloops} \hat{f}^{(\nnofloops)}(\bs x, y;\bs w) \, \ud \bs{w} 
\\
= \; & 
\Big( \frac{\Gamma(2-8/\kappa)}{\Gamma(1-4/\kappa)^2} \Big)^n \, \Ffusion{n}(\bs x, y) , 
\end{split}
\end{align}
where 
each $\hat{\contour}_\nnofloops \ni \bs w$ is an integration contour depicted in Figure~\ref{fig:Gammalhat}, 
and the integrand $\hat{f}^{(\nnofloops)}(\bs x, y;\bs w)$ associated to it is a branch of the multivalued function
\begin{align*}
\hat{f}(\bs x, y;\bs w)
:=\; & \prod_{1\leq i<j\leq n}(x^{j}-x^{i})^{\frac{2}{\kappa}}
\prod_{1\leq \ell \leq n}(y-x^\ell)^{\frac{2n}{\kappa}}
\prod_{\substack{1\leq i\leq n\\1\leq r\leq n}}
(w^{r}-x^{i})^{-\frac{4}{\kappa}}
\prod_{\substack{1\leq r\leq n}}
(w^{r}-y)^{-\frac{4n}{\kappa}}
\prod_{1\leq r<s\leq n}(w^{s}-w^{r})^{\frac{8}{\kappa}} 
\end{align*}
chosen to be real and positive at the point indicated by the red circles in Figure~\ref{fig:Gammalhat}; 
and for later use, we also include in~\eqref{eq:integral formula for fused ppf} an expression involving the function $\smash{\Ffusion{n}}$ analogous to that in Lemma~\ref{lem: rainbow ppf},
\begin{align}
\label{eqn::Ffusion}
\Ffusion{n}(\bs x, y) :=\; & 
\frac{1}{(q^{-2}-1)^{n}}\frac{\qnum{2}^{n}}{\qfact{n+1}}\,
\sum_{\nnofloops=0}^{n}(-1)^{\nnofloops} q^{\nnofloops(n-\nnofloops-1)}\int_{\hat{\contour}_\nnofloops} \hat{f}^{(\nnofloops)}(\bs x, y;\bs w) \, \ud \bs{w} ,
\qquad q = \ee^{4 \pi \ii / \kappa} .
\end{align}
It remains to show that the formula~\eqref{eq:integral formula for fused ppf} coincides with the right-hand side of~\eqref{eq: block limit}. 
To this end, we will verify that the ratio of the right-hand sides of~\eqref{eq:integral formula for fused ppf} and~\eqref{eq: block limit} is identically $1$. 
We consider this ratio $G(y)$ 
as a function of $y \in \C \setminus \{x^1,\ldots,x^{n}\}$ with fixed $x^1 < \cdots < x^{n}$.
Because we already know that the limit on the left-hand side of~\eqref{eq:integral formula for fused ppf} 
exists for all $\kappa \in (0,8)$, by continuity it suffices to verify that
\begin{align*}
G(y) = G_{\bs x}(y) := 
\Big( \frac{\Gamma(2-8/\kappa)}{\Gamma(1-4/\kappa)^2} \Big)^n 
\frac{1}{A_n} \, \frac{\Ffusion{n}(\bs x, y)}{\LZfusion{n}(\HH;\bs x, y)} \equiv 1 , \qquad 
\kappa \in (0,8) \setminus Q_n ,
\end{align*}
where $Q_n \subset \QQ$ denotes the finite set of singularities of the multiplicative constant $\kappa \mapsto \big( \frac{\Gamma(2-8/\kappa)}{\Gamma(1-4/\kappa)^2} \big)^n / A_n(\kappa)$.

On the one hand, using~\cite[Theorem~5.3, Part~(4)]{Peltola:Basis_for_solutions_of_BSA_PDEs_with_particular_asymptotic_properties} 
and analytic continuation, we see that 
\begin{align*}
\big| \Ffusion{n}(\bs x, y) \big| 
\sim |y - x^j|^{\mathfrak{h}_{n-1} - \mathfrak{h}_{n} - \mathfrak{h}_1} , \qquad y \to x^j ,
\textnormal{ for any $1 \leq j \leq n$,}
\end{align*}
where the exponent is $\mathfrak{h}_{n-1} - \mathfrak{h}_{n} - \mathfrak{h}_1 = 1-\frac{2}{\kappa}(n+2)$ 
in terms of the Kac conformal weight~\eqref{eqn::Kac_conformal_weight} (and $\mathfrak{h}_1 = \mathfrak{b}$). 
Because $\LZfusion{n}$ defined in~\eqref{eqn::LZfusion_H} has the same asymptotics, 
\begin{align*}
\big| \LZfusion{n}(\HH;\bs x, y) \big|
\sim |y - x^j|^{1-\frac{2}{\kappa}(n+2)}  , \qquad y \to x^j ,
\end{align*}
we see that $y \mapsto G(y)$ is holomorphic on $\C$. 
On the other hand, \cite[Proposition~5.8]{Peltola:Basis_for_solutions_of_BSA_PDEs_with_particular_asymptotic_properties} (or~\cite[Proposition~5.4]{Kytola-Peltola:Conformally_covariant_boundary_correlation_functions_with_quantum_group}) gives
\begin{align*}
\big| \Ffusion{n}(\bs x, y) \big| 
\sim |y|^{-2 \mathfrak{h}_{n}}  , \qquad y \to \infty ,
\end{align*}
and similarly, using the explicit formula~\eqref{eqn::LZfusion_H} involving $1-\frac{2}{\kappa}(n+2) = -\frac{2}{n}\mathfrak{h}_{n}$, we have
\begin{align*}
\big| \LZfusion{n}(\HH;\bs x, y) \big|
\sim |y|^{-2 \mathfrak{h}_{n}}  , \qquad y \to \infty .
\end{align*}
Hence, for any fixed $\bs x = (x^1,\ldots,x^{n})$ with $x^1 < \cdots < x^{n}$, 
we conclude that $y \mapsto G(y)$ is a bounded entire function, thus constant by Liouville's theorem.
Using~\cite[Proposition~5.4]{Kytola-Peltola:Conformally_covariant_boundary_correlation_functions_with_quantum_group},
we readily find that
\begin{align*}
\lim_{y \to + \infty}  y^{2 \mathfrak{h}_{n}} \, 
\lim_{\substack{y^j\to y\\\forall 1\le j\le n}}
\frac{\LZrainbow{n}(\HH;\bs x, \bs y)}{\LZ_{\defpatt_n}(\HH;\bs y)} 
= \; & A_n \, \LZ_{\defpatt_n}(\HH;\bs x) 
, \qquad \kappa \in (0,8) \setminus \QQ ,
\end{align*}
where $A_n$ is the multiplicative constant~\eqref{eq: fusion cst}.  
Using the explicit formula~\eqref{eqn::LZfusion_H} for $\LZfusion{n}(\HH;\bs x, y)$, we have
\begin{align*}
\lim_{y \to + \infty}  y^{2 \mathfrak{h}_{n}} \, \LZfusion{n}(\HH;\bs x, y)
= \; & \LZ_{\defpatt_n}(\HH;\bs x) .
\end{align*}
These together imply that $G(y) \to 1$ as $y \to \infty$, 
so we conclude that $G_{\bs x}(y) \equiv 1$ for all $y$ and $\bs x$.
\end{proof}

\begin{center}
\includegraphics[width=.62\textwidth]{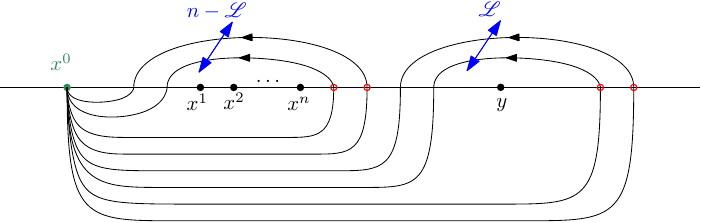}
\end{center}
\refstepcounter{figure}\label{fig:Gammalhat}
{\small\noindent\textbf{Figure~\thefigure} Illustration of the integration contours $\hat{\contour}_\nnofloops$ in the proof of Proposition~\ref{prop: block limit}. 
The red circles indicate the branch choice of the integrand function $\bs w \mapsto \hat{f}^{(\nnofloops)}(\bs x, y;\bs w)$
for fixed $x^1 < \cdots < x^{n} < y$, 
so that it is real and positive when 
$x^{n} < \Re(w^1) < \cdots < \Re(w^{n-\nnofloops}) < y$ and
$y < \Re(w^{n-\nnofloops+1}) < \cdots < \Re(w^n)$. 
By~\cite[Theorem~4.17]{Kytola-Peltola:Conformally_covariant_boundary_correlation_functions_with_quantum_group}, 
while the integration $\hat{\contour}_\nnofloops$ depends on an auxiliary variable $x^0 < x^1$, 
the function~\eqref{eqn::Ffusion} defined as a linear combination of integrals involving $\hat{\contour}_\nnofloops$, for $\nnofloops = 0,1,\ldots,n$, 
is independent of the choice of $x^0$.
}

\begin{example}
When $n=1$, the Gamma functions in $A_n$ cancel out and $\qnum{1} \equiv 1$, so we find
\begin{align*}
\lim_{y \to + \infty}  y^{2\mathfrak{b}} \, 
\frac{\LZrainbow{1}(\HH;x, y)}{\LZ_{\defpatt_1}(\HH;y)} 
= \; & \LZ_{\defpatt_1}(\HH;x) ,
\end{align*}
where $\LZ_{\defpatt_1}(\HH;x) \equiv 1$ and $\LZrainbow{1}(\HH;x, y) = (y-x)^{-2\mathfrak{b}}$
(which is rather trivial and well-known). 
\end{example}

\begin{example}
When $n=2$, we have $\qnum{2} = q+q^{-1} = 2 \cos(4 \pi/\kappa)$ for $q = \ee^{4 \pi \ii / \kappa}$, so we find
\begin{align*}
\lim_{y \to + \infty}  y^{2\mathfrak{h}_{2}} \, 
\lim_{\substack{y^1\to y\\y^2\to y}}
\frac{\LZrainbow{2}(\HH;x^1, x^2, y^2, y^1)}{\LZ_{\defpatt_2}(\HH;y^2,y^1)} 
= \; & \frac{\qnum{2}^{3}}{\qnum{3}}  \, 
\Big( \frac{\Gamma(2-8/\kappa)}{\Gamma(1-4/\kappa)^2} \Big)^2 \,
\Selberg_{2} \, 
\, \LZ_{\defpatt_2}(\HH;x^1, x^2)
\\
= \; & \frac{\LZ_{\defpatt_2}(\HH;x^1, x^2)}{\hF\big(\tfrac{4}{\kappa}, 1-\tfrac{4}{\kappa}, \tfrac{8}{\kappa}; 1\big) } ,
\end{align*}
where $\LZ_{\defpatt_2}(\HH;x, y) = (y-x)^{2/\kappa}$. 
The multiplicative constant can be evaluated by using the formula~\eqref{eq: Selberg} for the Selberg integral:
\begin{align*}
\Selberg_{2} = \; & \frac{1}{2} \; 
\frac{\Gamma( 1 - 8/\kappa)^2}
{\Gamma( 2 - 12/\kappa)} 
\frac{\Gamma( 1 - 4/\kappa)^2 \; \Gamma( 1 + 8/\kappa )}
{\Gamma( 1 + 4/\kappa) \; \Gamma( 2 - 8/\kappa) } ,
\end{align*}
the identity $\Gamma(a+1) = a \Gamma(a)$, which yields
\begin{align*}
\frac{\Gamma( 1 + 8/\kappa )}{\Gamma( 1 + 4/\kappa)}
= 2 \, \frac{\Gamma(8/\kappa )}{\Gamma(4/\kappa)} ,
\end{align*}
the formulas 
(which follow from the identity $\Gamma(1-a) = \frac{\pi}{\Gamma(a) \sin(\pi a)}$ and simple trigonometry)
\begin{align*}
\qnum{2} = \; &
2 \cos(4\pi/\kappa) = \frac{\sin(8\pi/\kappa)}{\sin(4\pi/\kappa)}
= \frac{\Gamma(1-4/\kappa) \, \Gamma(4/\kappa)}{\Gamma(8/\kappa) \, \Gamma(1-8/\kappa)} , \\
\frac{\Gamma(2-8/\kappa)}{\Gamma( 2 - 12/\kappa)} = \; & 
\frac{\Gamma(12/\kappa - 1) }{\Gamma(8/\kappa - 1)}
\frac{\sin(12 \pi /\kappa)}{\sin(8 \pi /\kappa)} 
=\frac{\Gamma(12/\kappa - 1) }{\Gamma(8/\kappa - 1)}
\frac{\qnum{3}}{\qnum{2}} ,
\end{align*}
and using the hypergeometric function~\cite[Eq.~(15.3.1)]{Abramowitz-Stegun:Handbook} 
evaluated at $z=1$: 
\begin{align*}
\hF\big(\tfrac{4}{\kappa}, 1-\tfrac{4}{\kappa}, \tfrac{8}{\kappa}; 1\big) 
= \; & \frac{\Gamma(8/\kappa) \, \Gamma(8/\kappa - 1)}{\Gamma(4/\kappa) \, \Gamma(12/\kappa - 1) } .
&& \textnormal{[by~\cite[Eq.~15.1.20]{Abramowitz-Stegun:Handbook}]}
\end{align*}
\end{example}

We end this section with a technical result controlling the rainbow partition functions (Corollary~\ref{cor: control}). 
It will be used in the proof of Theorem~\ref{thm::multichordal_halfwatermelon}. We first extend
the definition of $\LZrainbow{n}$ to more general polygons $(\Omega; x^1, \ldots, x^{2n})$ via a covariance similar to~\eqref{eqn::chordalCOV}:
\begin{align}\label{eqn::PartF_def_polygon}
\LZrainbow{n}(\Omega; u^{1},\ldots,u^{2n}) = 
\Big( \prod_{j=1}^{2n} |\varphi'(u^{j})| \Big)^{\mathfrak{b}} 
\, \LZrainbow{n}(\HH; \varphi(u^{1}),\ldots,\varphi(u^{2n})).
\end{align}
where $\varphi$ is any conformal map from $\Omega$ onto $\HH$ with $\varphi(u^1)<\cdots<\varphi(u^{2n})$. In particular, for $n=1$, 
the partition function is given in terms of the boundary Poisson kernel in $\Omega$ \textnormal{(}see~\textnormal{(\ref{eqn::Poisson_def_H},~\ref{eqn::bPoisson_cov}))},
\begin{align*}
\LZtwo(\Omega; x, y) = (H(\Omega; x, y) )^{\mathfrak{b}}. 
\end{align*}

\begin{corollary} \label{cor: control}
Fix $\kappa\in (0,8)$ and $n\ge 2$. 
For a $2n$-polygon $(\Omega; \bs{x}, \bs{y})$ 
and two points $x\in (y^1\, y^n)$ and $y\in (x^n\, x^1)$, define 
\begin{align*} 
\LR(\Omega; x, \bs{x}, \bs{y}, y) 
:= \frac{\LZrainbow{n}(\Omega; \bs{x}, \bs{y}) \, ( H(\Omega; x, y) )^{\mathfrak{h}_n}}{\LZfusionInv{n}(\Omega; \bs{x}, y) \, \LZfusion{n}(\Omega; x, \bs{y})}. 
\end{align*}
The function $\LR$ satisfies the following properties. 
\begin{enumerate}
\item \label{item_R-func-1-inv}
It is conformally invariant: for any conformal map $\varphi$ on $\Omega$, we have
\begin{align*} 
\LR(\Omega; x, \bs{x}, \bs{y}, y) = \LR(\varphi(\Omega); \varphi(x), \varphi(\bs{x}), \varphi(\bs{y}), \varphi(y)) .
\end{align*}

\item \label{item_R-func-2-asy}
It satisfies the asymptotics
\begin{align*} 
\lim_{\substack{y^j\to y\\\forall 1\le j\le n}}\LR(\Omega; x, \bs{x}, \bs{y}, y) = A_n 
\qquad \textnormal{and} \qquad 
\lim_{\substack{x^j\to x\\\forall 1\le j\le n}}\LR(\Omega; x, \bs{x}, \bs{y}, y) = A_n,
\end{align*}
where the multiplicative constant $A_n$ is given by~\eqref{eq: fusion cst}. 

\item \label{item_R-func-3-sup}
Consider $\Omega=\HH$ with $x=0$ and $y = \infty$. For $R \in (0,\infty)$, consider
\begin{align*}
\mathfrak{R}(R) := \Big\{  (\bs x, \bs y) \in \R^{2n} \, \bigcond \, 
-\tfrac{R}{2} \leq x^1 \leq \cdots \leq x^n \le \tfrac{R}{2} \, \textnormal{ and } \, 
|y^j| \geq R \textnormal{ for all } j \in \{1,\ldots,n\} \Big\} ,
\end{align*}
where $\bs x = (x^1,\ldots,x^{n})$ and $\bs y = (y^{n},\ldots,y^1)$ appear in counterclockwise order along $\R \cup \{\infty\}$.
Then, 
\begin{align*}
C_{\sup} := \sup_{(\bs x, \bs y) \in \mathfrak{R}(R)} \; \LR(\HH; 0, \bs x, \bs y, \infty)  
\qquad \textnormal{and} \qquad
C_{\inf} := \inf_{(\bs x, \bs y) \in \mathfrak{R}(R)} \; \LR(\HH; 0, \bs x, \bs y, \infty) 
\end{align*}
are finite, nonzero, and independent of $R$. 
\end{enumerate}
\end{corollary}

\begin{proof}
The conformal invariance (Item~\ref{item_R-func-1-inv}) follows from the conformal covariances~(\ref{eqn::LZfusion_cov},~\ref{eqn::bPoisson_cov},~\ref{eqn::PartF_def_polygon}).  
The asymptotics (Item~\ref{item_R-func-2-asy}) follows from Proposition~\ref{prop: block limit} 
combined with the conformal covariances~(\ref{eqn::LZfusion_cov},~\ref{eqn::PartF_def_polygon}). 
It remains to show Item~\ref{item_R-func-3-sup}. 
To this end, fix $\Omega=\HH$. If $x^1 < \cdots < x^{n} < y^{n} < \cdots < y^1$, using the integral formula~\eqref{eq: rainbow ppf} for $\LZrainbow{n}$ (Lemma~\ref{lem: rainbow ppf}) and the formulas~\eqref{eqn::LZfusion_H} for $\LZfusionInv{n}$ and $\LZfusion{n}$, we see that
\begin{align*}
\LR(\HH; 0; \bs x; \bs y; \infty) 
= \; & \Big( \prod_{1\leq j \leq n} y^j \Big)^{\frac{2}{\kappa}(n+2)-1} \,  \frac{\LZrainbow{n}(\HH; \bs{x}, \bs{y})}{\LZ_{\defpatt_n}(\HH;\bs x) \, \LZ_{\defpatt_n}(\HH;\bs y) }
\\
= \; & \Big( \prod_{1\leq j \leq n} y^j \Big)^{\frac{2}{\kappa}(n+2)-1} \, \Big( \prod_{1\leq i,j\leq n} (y^{j}-x^i) \Big)^{2/\kappa} \, 
\sum_{\nnofloops=0}^{n}(-1)^{\nnofloops} q^{\nnofloops(n-\nnofloops-1)} \int_{\contour_\nnofloops} g^{(\nnofloops)}(\bs x, \bs y;\bs w) \, \ud \bs{w} ,
\end{align*}
writing $\ud \bs{w} = \ud w^1 \cdots \ud w^{\nnofloops}$, 
where $g^{(\nnofloops)}$ is a branch of the multivalued function
\begin{align*} 
g(\bs x, \bs y;\bs w)
:=\; & \prod_{\substack{1\leq i\leq n\\1\leq r\leq n}}
(w^{r}-x^{i})^{-\frac{4}{\kappa}}
\prod_{\substack{1\leq i\leq n\\1\leq r\leq n}}
(w^{r}-y^{i})^{-\frac{4}{\kappa}}
\prod_{1\leq r<s\leq n}(w^{s}-w^{r})^{\frac{8}{\kappa}}
\end{align*}
chosen to be real and positive at the point indicated by the red circles in Figure~\ref{fig:Gammal}.
In the integral, we may assume without loss of generality that, for a chosen small $\delta > 0$,
\[
w^1, \ldots, w^{\nnofloops} \in
\C \setminus
\Big( \big( [x^1 - \delta, x^n + \delta] \cup [y^n - \delta, y^1 + \delta] \big)
\times [- \delta , + \delta]\Big),
\]
so 
we see that the contours in $\contour_\nnofloops$ avoid the marked points 
and no contour separates $x^i$ from $x^j$, nor $y^i$ from $y^j$
(indeed, the contours only separate $\{x^1, \ldots, x^n\}$ from $\{y^1, \ldots, y^n\}$).
It follows that the function $(\bs x, \bs y) \mapsto \LR(\HH; 0, \bs x, \bs y, \infty)$ has a continuous extension to $\mathfrak{R}(R)$.

By scale-invariance of $\LR$, the quantities $C_{\sup}$ and $C_{\inf}$ are independent of $R$  (so we may take $R=1$). 
Moreover, because $\LR$ is conformally invariant, it is sufficient to evaluate the quantities $C_{\sup}$ and $C_{\inf}$ in a rectangle $\Omega$ obtained by sending the upper half-plane $\HH$ onto $\Omega$ by a conformal map $\psi$ 
such that the points $-1, \, -1/2, \, 1/2, \, 1$ map to the corners of the rectangle $\Omega$
--- see Figure~\ref{fig::control} for an illustration. 
Now, because $(\bs x, \bs y) \mapsto \LR(\HH; 0; \bs x; \bs y; \infty)$ is continuous on $\mathfrak{R}(1)$,
its pre-composition by the conformal map $\psi^{-1}$ is continuous on the compact set $\psi(\mathfrak{R}(1))$.
In particular, we have 
\begin{align*}
C_{\sup} = \underset{(\bs x, \bs y) \in \psi(\mathfrak{R}(1))}{\sup} \, \LR(\Omega; \psi(0), \psi(\bs x), \psi(\bs y), \psi(\infty)) < \infty.
\end{align*}

Applying the same argument to $1/\LR(\HH; 0, \bs x, \bs y, \infty)$ shows that $C_{\sup} > 0$ and $C_{\inf} \in (0,\infty)$.
\end{proof}

\begin{center}
\includegraphics[width=0.66\textwidth]{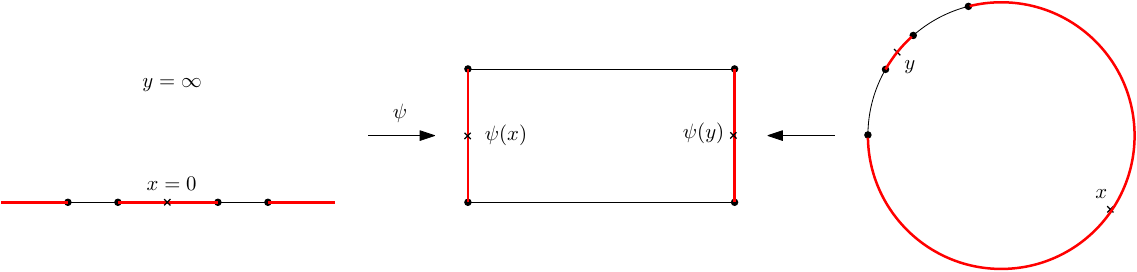}
\end{center}
\refstepcounter{figure}\label{fig::control}
{\small\noindent\textbf{Figure~\thefigure} Illustration of the symmetry argument in the proof of Corollary~\ref{cor: control} and in the proof of Theorem~\ref{thm::multichordal_halfwatermelon}.}

\subsection{From multichordal SLE to half-watermelon SLE --- proof of Theorem~\ref{thm::multichordal_halfwatermelon}~\& Proposition~\ref{prop::halfwatermelon_bp}}
\label{subsec: Proof of perturbation}

The goal of this section is to prove that multichordal $\SLE_{\kappa}$ converges to half-watermelon $\SLE_{\kappa}$ when the endpoints of the $n$ distinct curves tend to a common point on the boundary (Theorem~\ref{thm::multichordal_halfwatermelon}) and establish the boundary perturbation property of half-watermelon $\SLE_{\kappa}$ (Proposition~\ref{prop::halfwatermelon_bp}). 

\smallskip

We first consider rainbow partition functions in the unit disc: the $2n$-polygon $(\Omega;\bs x, \bs y) = (\U; \ee^{\ii\bs{\theta}})$ with $\ee^{\ii\bs{\theta}} = (\ee^{\ii\theta^1},\ldots, \ee^{\ii\theta^{2n}})$
for $(\theta^1, \ldots,\theta^{2n}) \in\LX_{2n}$,
writing $(\ee^{\ii\theta^1},\ldots, \ee^{\ii\theta^{n}})=\bs x$ and $(\ee^{\ii\theta^{n+1}},\ldots, \ee^{\ii\theta^{2n}})=\bs y$, and 
\begin{align*}
\LZrainbow{n}(\bs\theta) = \LZrainbow{n}(\theta^1, \ldots, \theta^{2n}) = \LZrainbow{n}(\U; \bs{x}, \bs{y}).  
\end{align*}
Note that the chordal BPZ equations~\eqref{eqn::chordalBPZ} become radial BPZ equations
\begin{align}\label{eqn::radialBPZ_ppf}
\frac{\kappa}{2} \frac{\partial_j^2 \LZrainbow{n}(\bs\theta)}{\LZrainbow{n}(\bs\theta)} 
+ \underset{1\leq i\neq j \leq 2n}{\sum} \, \bigg( \cot \bigg(\frac{\theta^{i}-\theta^{j}}{2}\bigg)
\frac{\partial_{i} \LZrainbow{n}(\bs\theta)}{\LZrainbow{n}(\bs\theta)} \, - \, \frac{\mathfrak{b}/2}{ \sin^2 \big(\frac{\theta^{i}-\theta^{j}}{2}\big)}\bigg) = \tilde{\mathfrak{b}}  
, \qquad \textnormal{for all }j .
\end{align}

It will be convenient to view $n$-chordal $\SLE_{\kappa}$ as a (local) commuting family of $2n$ growing curves \`a la Dub{\'e}dat~\cite{Dubedat:Commutation_relations_for_SLE}, 
instead of its global version from Definition~\ref{def::globalnSLE}. 
From the results in~\cite{Peltola-Wu:Global_and_local_multiple_SLEs_and_connection_probabilities_for_level_lines_of_GFF}, we know that these two descriptions agree. 
In particular, the Loewner evolution for $n$-chordal $\SLE_{\kappa}$ can be described by rainbow partition functions, which enables us to establish the following comparison.

\begin{lemma}\label{lem::globalnSLE_mart}
Fix $\kappa\in (0,4], \, n\ge 1$, and $\bs{\theta} = (\theta^1, \ldots, \theta^{2n})\in\LX_{2n}$. 
\begin{itemize}
\item Let $\eta^{j}$ be radial $\SLE_{\kappa}$ in $(\U; \ee^{\ii\theta^j}; 0)$ for each $1\le j\le 2n$, 
and let $\mathsf{P}_{2n}$ be the probability measure on $\bs{\eta}=(\eta^{1}, \ldots, \eta^{2n})$ under which the curves are independent. 

\item Let $\bs{\bar{\eta}}=(\bar{\eta}^{1}, \ldots, \bar{\eta}^{n})\sim\QQrainbow{n}(\U; \ee^{\ii\bs{\theta}})$ be $n$-chordal $\SLE_{\kappa}$ in $(\U; \ee^{\ii\bs{\theta}})$. 
\end{itemize}
Viewing both processes as $2n$-tuples of curves, $\bs{\eta}$ identified with initial segments of $\bs{\bar{\eta}}$ growing from both endpoints of each curve, 
we parameterize $\bs{\eta}$ using $2n$-time-parameter $\bs{t}=(t_1, \ldots, t_{2n})$.
Also, we let $\bs{\mslitdriv}_{\bs{t}} = (\mslitdriv^{1}_{\bs{t}},\ldots,\mslitdriv^{2n}_{\bs{t}})$ denote its multi-slit driving function~\eqref{eqn::Ntuple_driving}.
Then, the law of $\bs{\eta}$ under $\QQrainbow{n}$ is the same as $\mathsf{P}_{2n}$ tilted by the following $2n$-time-parameter local martingale, up to the collision time: 
\begin{align*}
M_{\bs{t}}(\LZrainbow{n})
= \; & \one_{\LE_{\emptyset}(\bs{\eta}_{\bs{t}})}  \,\exp\bigg(\frac{\mathfrak{c}}{2}\blm_{\bs{t}} -\tilde{\mathfrak{b}}\sum_{j=1}^{2n}t_j\bigg) \, 
\Big(\prod_{j=1}^{2n} \covmap_{\bs{t},j}'(\xi_{t_j}^{j}) \Big)^{\mathfrak{b}}
\,\LZrainbow{n}(\bs{\mslitdriv}_{\bs{t}}) ,
\end{align*}
where $\LE_{\emptyset}(\bs{\eta}_{\bs{t}}) = \{ \eta_{[0,t_j]}^{j} \cap \eta_{[0,t_i]}^{i}=\emptyset, \, \forall i\neq j\}$ is the event that different curves are disjoint, 
and $\blm_{\bs{t}}$ is the unique potential solving the exact differential equation~\eqref{eqn::mt_def} with $p=2n$ \textnormal{(}see Lemma~\ref{lem::blm_exact}\textnormal{)}.
\end{lemma}

Note that the fact that $M(\LZrainbow{n})$ is $2n$-time-parameter local martingale follows immediately from Proposition~\ref{prop::multitime_mart_universal} because $\LZrainbow{n}(\theta^1, \ldots, \theta^{2n})$ satisfies the system of radial BPZ equations~\eqref{eqn::radialBPZ_ppf}. 

\begin{proof}
Follows from~\cite[Lemmas~2.3~\&~2.4]{FWW:Multiple_Ising_interfaces_in_annulus_and_2N-sided_radial_SLE} and the definitions of the various conformal maps. 
\end{proof}

Fix $\vartheta\in (\theta^{n}, \theta^1+2\pi)$ and write $\ee^{\ii\vartheta} = y$. 
Later, we will send $\theta^j\to\vartheta$ for all $j\in\{n+1,n+2,\ldots, 2n\}$. 
We write the function $\LZfusion{n}$ defined in~(\ref{eqn::LZfusion_H},~\ref{eqn::LZfusion_cov}) in the angle coordinates 
as in~\eqref{LZfusion_U}:
\begin{align*}
\LZfusion{n}(\theta^1, \ldots, \theta^n, \vartheta)
\; = \;& \LZfusion{n}(\U; \bs x, y) 
\;= \;\LZ_{\defpatt_n}(\bs{\theta})\; \Big( \prod_{j=1}^n |\ee^{\ii\vartheta} - \ee^{\ii\theta^j}| \Big)^{1-\frac{2}{\kappa}(n+2)} ,
\\
\textnormal{where} \quad 
\LZ_{\defpatt_n}(\theta^1, \ldots, \theta^n)
\; = \;& \LZ_{\defpatt_n}(\U; \bs x) 
\;= \; \Big( \prod_{1\leq i<j\leq n} |\ee^{\ii\theta^j} - \ee^{\ii\theta^i}| \Big)^{2/\kappa} ,
\end{align*}
with $\bs x=(\ee^{\ii\theta^1},\ldots, \ee^{\ii\theta^{n}})$. 
Proposition~\ref{prop: block limit} yields
\begin{align}\label{eqn::fusion_unitdisc}
\lim_{\substack{\theta^j\to\vartheta\\ \forall n+1\le j\le 2n}}
\frac{\LZrainbow{n}(\theta^1, \ldots, \theta^{2n})}{\LZ_{\defpatt_n}(\theta^{n+1}, \ldots, \theta^{2n})} 
= A_n \, \LZfusion{n}(\theta^1, \ldots, \theta^n, \vartheta) ,
\end{align}
and Corollary~\ref{cor: control} implies that there exist $C_{\inf}, C_{\sup}\in (0,\infty)$ (independent of $r$) such that 
\begin{align}\label{eqn::control_unitdisc}
C_{\inf} \; \le \; \frac{\LZrainbow{n}(\theta^1, \ldots, \theta^{2n})}{\LZ_{\defpatt_n}(\theta^{n+1}, \ldots, \theta^{2n}) \, \LZfusion{n}(\theta^1, \ldots, \theta^n, \vartheta)}  \; \le \; C_{\sup}, 
\end{align}
for all $r>0$ and $|\ee^{\ii\theta^{n+j}}  - \ee^{\ii\vartheta}| \le r$ and $|\ee^{\ii\theta^j}  - \ee^{\ii\vartheta}|\ge 2r$, for $1 \leq j \leq n$. See also Figure~\ref{fig::control}. 

\bigskip

Now, we are ready to prove Theorem~\ref{thm::multichordal_halfwatermelon}.

\begin{proof}[Proof of Theorem~\ref{thm::multichordal_halfwatermelon}]
The convergence in~\eqref{eqn::LZrainbowtoLZfusion} follows from Proposition~\ref{prop: block limit} and conformal covariance (Eq.~\eqref{eqn::fusion_unitdisc}). 
It remains to prove the convergence of rainbow $\SLE_\kappa$ to half-watermelon $\SLE_\kappa$, for which purpose it is sufficient to consider the polygon
$(\Omega;\bs x, \bs y) = (\U; \ee^{\ii\bs{\theta}})$ and $y = \ee^{\ii\vartheta}$ as above.

Let $\bs{\bar{\eta}}=(\bar{\eta}^{1}, \ldots, \bar{\eta}^{n})\sim\QQrainbow{n}(\U; \ee^{\ii\bs{\theta}})$ be $n$-chordal $\SLE_{\kappa}$ 
with the curves oriented such that
each $\bar{\eta}^{j}$ traverses from $x^{j}=\ee^{\ii\theta^j}$ to $y^{j}=\ee^{\ii\theta^{2n+1-j}}$, for $1\le j\le n$. 
Viewing it as an $n$-tuple of curves $\bs{\eta}=(\eta^{1}, \ldots, \eta^{n})$ we parameterize it using $n$-time-parameter $\bs{t}=(t_1, \ldots, t_n)$, 
and we let $\bs{\mslitdriv}_{\bs{t}} = (\mslitdriv^{1}_{\bs{t}},\ldots,\mslitdriv^{n}_{\bs{t}})$ denote the associated multi-slit driving function~\eqref{eqn::Ntuple_driving}.
Note that there are no curves growing from the points $\bs y=(\ee^{\ii\theta^{n+1}},\ldots, \ee^{\ii\theta^{2n}})$: 
we have $(t_1, \ldots, t_n, t_{n+1}, \ldots, t_{2n})=(t_1, \ldots, t_n, 0, \ldots, 0)$.
For ease, let us write $\bs\vartheta := (\theta^{n+1}, \ldots, \theta^{2n})$ for these spectator points, which all converge to $\vartheta$; 
and let us write $\covmap_{\bs{t}}(\bs\vartheta) := (\covmap_{\bs{t}}(\theta^{n+1}), \ldots, \covmap_{\bs{t}}(\theta^{2n}))$ for their Loewner evolution. 
Now, Lemma~\ref{lem::globalnSLE_mart} shows that the law of $\bs{\eta}$ under $\QQrainbow{n}(\U; \ee^{\ii\bs{\theta}})$ is the same as $\mathsf{P}_{2n}$ tilted by the $2n$-time-parameter local martingale
\begin{align*}
M_{\bs{t}}(\LZrainbow{n})
= \; & \one_{\LE_{\emptyset}(\bs{\eta}_{\bs{t}})}  \,\exp\bigg(\frac{\mathfrak{c}}{2}\blm_{\bs{t}} -\tilde{\mathfrak{b}}\sum_{j=1}^{n}t_j\bigg) \, 
\Big(\prod_{j=1}^{n} \covmap_{\bs{t},j}'(\xi_{t_j}^{j}) \Big)^{\mathfrak{b}}
\, \Big(\prod_{j=n+1}^{2n} \covmap_{\bs{t}}'(\ee^{\ii\theta^{j}}) \Big)^{\mathfrak{b}}
\,\LZrainbow{n}(\bs{\mslitdriv}_{\bs{t}}, \covmap_{\bs{t}}(\bs\vartheta)) ,
\end{align*}
up to the collision time. 
Combining this tilting with Lemma~\ref{lem::rainbowSLE_mart}, we see that the Radon-Nikodym derivative of $\QQrainbow{n}(\U; \bs{x}, \bs{y})$ with respect to $\QQfusion{n}(\U; \bs{x}, y)$ is given by 
\begin{align*}
R_{\bs{t}} 
:= \frac{M_{\bs{t}}(\LZrainbow{n})}{M_{\bs{0}}(\LZrainbow{n})}
\frac{M_{\bs{0}}(\LZfusion{n})}{M_{\bs{t}}(\LZfusion{n})}
= & \; \underbrace{\Big(\prod_{j=n+1}^{2n} \covmap_{\bs{t}}'(\theta^{j}) \Big)^{\mathfrak{b}}
\, \frac{\LZrainbow{n}(\bs{\mslitdriv}_{\bs{t}}, \covmap_{\bs{t}}(\bs\vartheta))}{\LZrainbow{n}(\bs{\mslitdriv}_{\bs{0}}, \bs\vartheta)}}_{=: \; Z_{\bs{t}}} 
\; \frac{\LZfusion{n}(\bs{\mslitdriv}_{\bs{0}}, \vartheta) }{(\covmap_{\bs{t}}'(\vartheta))^{\mathfrak{h}_n} \, \LZfusion{n}(\bs{\mslitdriv}_{\bs{t}}, \covmap_{\bs{t}}(\vartheta)) } .
\end{align*}
The convergence from Eq.~\eqref{eqn::fusion_unitdisc} and the covariances of~(\ref{eqn::LZfusion_H},~\ref{eqn::LZfusion_cov},~\ref{eqn::chordalCOV}) together imply that 
\begin{align*}
\lim_{\substack{\theta^j\to\vartheta\\\forall n+1\le j\le 2n}}Z_{\bs{t}} 
= \; & ( \covmap_{\bs{t}}'(\vartheta) )^{n \mathfrak{b} + \frac{n(n-1)}{\kappa}} \, 
\frac{\LZfusion{n}(\bs{\mslitdriv}_{\bs{t}}, \covmap_{\bs{t}}(\vartheta))}{\LZfusion{n}(\bs{\mslitdriv}_{\bs{0}}, \vartheta)}
= \frac{ (\covmap_{\bs{t}}'(\vartheta) )^{\mathfrak{h}_n} \, \LZfusion{n}(\bs{\mslitdriv}_{\bs{t}}, \covmap_{\bs{t}}(\vartheta)) }{\LZfusion{n}(\bs{\mslitdriv}_{\bs{0}}, \vartheta) } ,
\end{align*}
so $R_{\bs{t}} \to 1$ as $\theta^j\to \vartheta$ for all $n+1\le j\le 2n$, for any $\bs{t}$ before the exit time from $U$. 
We will show that in fact, for any such multi-time, $R_{\bs{t}}$ stays uniformly bounded as $\theta^j\to \vartheta$ for all $n+1\le j\le 2n$.

First, recall that $\bs{\eta}_{\bs{t}}=\cup_{j=1}^n \eta^{j}_{[0,t_j]}$ and $g_{\bs{t}}$ is the conformal map from $\U\setminus\bs{\eta}_{\bs{t}}$ onto $\U$ 
normalized at the origin.
By assumption, $U$ has a positive distance from $y = \ee^{\ii\vartheta}$. 
Pick $\delta>0$ so that $B(y, 16\delta)\cap U=\emptyset$. 
Let us write $R_{\bs{t}}$ in the alternative form
\begin{align*}
R_{\bs{t}} 
:= \frac{M_{\bs{t}}(\LZrainbow{n})}{M_{\bs{0}}(\LZrainbow{n})}
\frac{M_{\bs{0}}(\LZfusion{n})}{M_{\bs{t}}(\LZfusion{n})}
= & \; \Big(\prod_{j=n+1}^{2n} \covmap_{\bs{t}}'(\ee^{\ii\theta^{j}}) \Big)^{\mathfrak{b}} \, ( \covmap_{\bs{t}}'(\vartheta) )^{-\mathfrak{h}_n} \, 
\underbrace{\frac{\LZrainbow{n}(\bs{\mslitdriv}_{\bs{t}}, \covmap_{\bs{t}}(\bs\vartheta))}{\LZfusion{n}(\bs{\mslitdriv}_{\bs{t}}, \covmap_{\bs{t}}(\vartheta)) }}_{=: \; Y_{\bs{t}}} 
\; \underbrace{\frac{\LZfusion{n}(\bs{\mslitdriv}_{\bs{0}}, \vartheta) }{\LZrainbow{n}(\bs{\mslitdriv}_{\bs{0}}, \bs\vartheta)}}_{=: \; 1/Y_{\bs{0}}}  .
\end{align*}
We use~\eqref{eqn::control_unitdisc} to control $Y_{\bs{t}}$. 
To this end, we first investigate the distance between $\{\mslitdriv^{1}_{\bs{t}},\ldots,\mslitdriv^{n}_{\bs{t}}\}$ and
$\{\covmap_{\bs{t}}(\theta^{n+1}), \ldots, \covmap_{\bs{t}}(\theta^{2n})\}$. 
From the standard estimates~(\ref{eqn::muliradial_halfwatermelon_RN_aux1},~\ref{eqn::muliradial_halfwatermelon_RN_aux2}), for all $1\le i,j\le n$, we have 
\begin{align*}
2r := 4\delta \, | g_{\bs{t}}'(y) | \le | g_{\bs{t}}(y) - \ee^{\ii\mslitdriv^{i}_{\bs{t}}} |
\qquad \textnormal{and} \qquad
\big| g_{\bs{t}}'(y) - g_{\bs{t}}'(y^j) \big| \le \frac{256}{225}\delta \, |g_{\bs{t}}'(y)| , 
\qquad \textnormal{for } y^j \in B(y,\delta) ,
\end{align*}
where $y^j = \exp(\ii\theta^{2n+1-j})$. 
Hence, $\smash{\underset{1\le i\le n}{\min} \, \big| g_{\bs{t}}(y) - \ee^{\ii\mslitdriv^{i}_{\bs{t}}} \big| \geq 2r}$
and $g_{\bs{t}}(y^j) \in B(g_{\bs{t}}(y),r)$, and~\eqref{eqn::control_unitdisc} yields 
\begin{align*}
Y_{\bs{0}}\ge C_{\inf} \, \LZ_{\defpatt_n}(\covmap_{\bs{0}}(\bs\vartheta)) 
\qquad \textnormal{and} \qquad
Y_{\bs{t}}\le C_{\sup} \, \LZ_{\defpatt_n}(\covmap_{\bs{t}}(\bs\vartheta)) ,
\qquad \textnormal{for} \quad 
\begin{cases}
B(y,16\delta)\cap U=\emptyset , \\ 
\bs{y} = (\ee^{\ii\theta^{n+1}}, \ldots, \ee^{\ii\theta^{2n}}) \in B(y,\delta) .
\end{cases}
\end{align*}
After collecting these estimates, we obtain
\begin{align*}
R_{\bs{t}} 
\le \; & \frac{C_{\sup}}{C_{\inf}} \; 
\Big(\prod_{j=n+1}^{2n} \covmap_{\bs{t}}'(\ee^{\ii\theta^{j}}) \Big)^{\mathfrak{b}} \, ( \covmap_{\bs{t}}'(\vartheta) )^{-\mathfrak{h}_n} \, 
\frac{\LZ_{\defpatt_n}(\covmap_{\bs{t}}(\bs\vartheta))}{\LZ_{\defpatt_n}(\covmap_{\bs{0}}(\bs\vartheta))} 
\\
= \; & \frac{C_{\sup}}{C_{\inf}} \; 
\Big(\prod_{j=n+1}^{2n} \covmap_{\bs{t}}'(\ee^{\ii\theta^{j}}) \Big)^{\mathfrak{b}} \, ( \covmap_{\bs{t}}'(\vartheta) )^{-\mathfrak{h}_n} \, 
\bigg(\prod_{n+1\leq i<j\leq 2n} \frac{|\covmap_{\bs{t}}(\theta^j) - \covmap_{\bs{t}}(\theta^i)|}{|\theta^j - \theta^i|} \bigg)^{2/\kappa} .
\end{align*}
Finally, we may control the derivatives as follows: from~\eqref{eqn::muliradial_halfwatermelon_RN_aux3}, 
we have
\begin{align*}
\covmap_{\bs{t}}'(\ee^{\ii\theta^{j}}) \lesssim \covmap_{\bs{t}}'(\vartheta)
\qquad \textnormal{and} \qquad
\frac{|\covmap_{\bs{t}}(\theta^j) - \covmap_{\bs{t}}(\theta^i)|}{|\theta^j - \theta^i|}
\lesssim \covmap_{\bs{t}}'(\vartheta) ,
\qquad \textnormal{for $n+1\leq i<j\leq 2n$,}
\end{align*}
which implies that $R_{\bs{t}} \lesssim \frac{C_{\sup}}{C_{\inf}} \lesssim 1$,
for $B(y,16\delta)\cap U=\emptyset$ and $\bs{y} = (\ee^{\ii\theta^{n+1}}, \ldots, \ee^{\ii\theta^{2n}}) \subset B(y,\delta)$.
This shows that the Radon-Nikodym derivative $R_{\bs{t}}$ stays uniformly bounded upon taking the limit as $\theta^j\to \vartheta$ for all $n+1\le j\le 2n$, for any $\bs{t}$ before the exit time from $U$, 
and finishes the proof. 
\end{proof}

To conclude this section, we prove the boundary perturbation property of half-$n$-watermelon $\SLE_{\kappa}$ (Proposition~\ref{prop::halfwatermelon_bp}).
The proof relies on Theorem~\ref{thm::multichordal_halfwatermelon} and Proposition~\ref{prop: block limit}, 
as well as on the well-known boundary perturbation property of multichordal $\SLE_{\kappa}$, that we first recall. 

\begin{lemma}[Boundary perturbation;~{\cite[Proposition~3.4]{Peltola-Wu:Global_and_local_multiple_SLEs_and_connection_probabilities_for_level_lines_of_GFF}}]
\label{lem::globalnSLE_bp} 
Fix $\kappa\in (0,4], \, n\ge 1$, and a $2n$-polygon $(\Omega; \bs{x})$. 
Suppose $K$ is a compact subset of $\overline{\Omega}$ such that $\Omega\setminus K$ is simply connected and $K$ has a positive distance from $\{x^1, \ldots, x^{2n}\}$. 
For $n$-chordal $\SLE_{\kappa}$ curves $\bs{\eta}=(\eta^{1}, \ldots, \eta^{n})$ in $(\Omega; \bs x)$, denote $\cup_{i=1}^n\eta^{i} = \bs{\eta}$.  
Then, the $n$-chordal $\SLE_{\kappa}$ probability measure $\QQrainbow{n}$ 
in the smaller polygon $(\Omega\setminus K; \bs{x})$ is absolutely continuous with respect to that in $(\Omega; \bs{x})$, 
with Radon-Nikodym derivative 
\begin{align}
\label{eqn::globalnSLESLEbp}
\frac{\LZrainbow{n}(\Omega; \bs{x})}{\LZrainbow{n}(\Omega\setminus K; \bs{x})} 
\, \one\{ \bs{\eta}\cap K=\emptyset \} \, \exp\Big(\frac{\mathfrak{c}}{2} \blm(\Omega; \bs{\eta},K)\Big). 
\end{align}
\end{lemma}

\begin{proof}[Proof of Proposition~\ref{prop::halfwatermelon_bp}]
By Lemma~\ref{lem::globalnSLE_bp}, we know that the law of $\bs{\eta}$ under $\QQrainbow{n}(\Omega\setminus K; \bs{x}, \bs{y})$ is the same as $\QQrainbow{n}(\Omega; \bs{x}, \bs{y})$ 
weighted by the Radon-Nikodym derivative
\begin{align*}
\frac{\LZrainbow{n}(\Omega; \bs{x}, \bs{y})}{\LZrainbow{n}(\Omega\setminus K; \bs{x}, \bs{y})} \, \one\{\bs{\eta}\cap K=\emptyset\} \, \exp\Big(\frac{\mathfrak{c}}{2}\blm(\Omega; \bs{\eta}, K)\Big). 
\end{align*}
Suppose $U\subset\Omega\setminus K$ is simply connected, has the points $x^1, \ldots, x^n$ on its boundary, and has a positive distance from $y$. 
The asserted Radon-Nikodym derivative~\eqref{eqn::multichordal_bp} follows from gathering these facts:
\begin{itemize}
\item By Proposition~\ref{prop: block limit}, we have
\begin{align*}
\lim_{\substack{y^j\to y\\\forall 1\le j\le n}}\frac{\LZrainbow{n}(\Omega; \bs{x}, \bs{y})}{\LZrainbow{n}(\Omega\setminus K; \bs{x}, \bs{y})}=\frac{\LZfusion{n}(\Omega; \bs{x}, y)}{\LZfusion{n}(\Omega\setminus K; \bs{x}; y)}.
\end{align*}

\item By Theorem~\ref{thm::multichordal_halfwatermelon}, we know that
up to the exit time from $U$, the law of $\bs{\eta}$ under $\QQrainbow{n}(\Omega; \bs{x}, \bs{y})$ converges weakly to $\QQfusion{n}(\Omega; \bs{x}, y)$ as $y^j\to y$ for all $1\le j\le n$.
Similarly, the law of $\bs{\eta}$ under $\QQrainbow{n}(\Omega\setminus K; \bs{x}, \bs{y})$ converges weakly to $\QQfusion{n}(\Omega\setminus K; \bs{x}; y)$ as $y^j\to y$ for all $1\le j\le n$.  
\end{itemize}
Hence, up to the exit time from $U$, the law of $\bs{\eta}$ under $\QQfusion{n}(\Omega\setminus K; \bs{x}; y)$ is the same $\QQfusion{n}(\Omega; \bs{x}, y)$ 
weighted by the Radon-Nikodym derivative~\eqref{eqn::multichordal_bp}. 
Because the set $U$ can be chosen arbitrarily, and all curves in half-$n$-watermelon $\SLE_{\kappa}$ are almost surely transient~\cite{Miller-Sheffield:Imaginary_geometry1}, this yields~\eqref{eqn::multichordal_bp} for the whole process.  
\end{proof}

\section{Multiradial SLE with spiral}
\label{sec::Multiradial_final}
We will prove the following two propositions in this section, which then result in Theorem~\ref{thm::multiradial_resampling}. 

\begin{proposition}\label{prop::multiradial_resampling}
Fix $\kappa\in (0,4], \, \mu\in\R, \, p\ge 2$ and a $p$-polygon $(\Omega; \bs{x})$ with $\bs{x}=( x^1, \ldots, x^p)$ and $z\in\Omega$.
The $p$-radial $\SLE_{\kappa}^\mu$ curves $\bs \gamma = (\gamma^{1}, \ldots, \gamma^{p})$ 
in $(\Omega; \bs x; z)$ satisfy the following properties. 
\begin{enumerate}
\item \label{item::multiradial_resampling1}
The marginal law of $\gamma^{1}$ is radial $\SLE_{\kappa}^{\mu}(2, \ldots, 2)$ in $(\Omega; \bs x; z)$ with force points $\hat{\bs x} = (x^2, \ldots, x^p)$. 

\item \label{item::multiradial_resampling2} 
Given $\gamma^{1}$, the conditional law of $(\gamma^{2}, \ldots, \gamma^{p})$ is half-$(p-1)$-watermelon $\SLE_{\kappa}$ in $(\Omega\setminus\gamma^{1}; \hat{\bs x}; z)$ as in Definition~\ref{def::halfwatermelonSLE}. 
\end{enumerate}
\end{proposition}

Proposition~\ref{prop::multiradial_resampling} is the first part of Theorem~\ref{thm::multiradial_resampling}, and we will prove it in Section~\ref{subsec::resampling}. 
We then utilize it to prove the boundary perturbation property (Proposition~\ref{prop::multiradial_bp}) for multiradial $\SLE_{\kappa}^\mu$ in Section~\ref{subsec::bp}.

\begin{proposition}\label{prop::multiradial_transience} 
Fix $\kappa\in (0,4], \, \mu\in\R, \, p\ge 2$, and $\bs{\theta}=(\theta^1, \ldots, \theta^p)\in \LX_p$.
The $p$-radial $\SLE_{\kappa}^\mu$ curves $\bs \gamma = (\gamma^{1}, \ldots, \gamma^{p})$ 
in $(\U; \ee^{\ii\bs{\theta}}; 0)$ are transient in the following senses: 
\begin{enumerate}
\item \label{item::multiradial_marg} 
Parameterizing $\bs{\gamma}$ by $p$-time parameter $\bs{t}=(t_1, \ldots, t_p) \in \R^p$, we have almost surely
\begin{align}\label{eqn::transience_multitime}
\gamma^{j}_{t_j}\to 0, \quad\textnormal{for all }1\le j\le p, \quad \textnormal{as }\min\{t_1, \ldots, t_p\}\to\infty.
\end{align}

\item \label{item::multiradial_transience} 
Parameterizing $\bs \gamma = (\gamma^{1}, \ldots, \gamma^{p})$ 
by common-time parameter $t \in \R$ \textnormal{(}defined in Section~\ref{subsec::transience} and~\textnormal{\cite{Healey-Lawler:N_sided_radial_SLE})}, we have almost surely
\begin{align}\label{eqn::transience_commontime}
\gamma^{j}_t\to 0, \quad\textnormal{for all }1\le j\le p, \quad \textnormal{as }t\to\infty. 
\end{align}
\end{enumerate}
\end{proposition}

We prove Proposition~\ref{prop::multiradial_transience} in Section~\ref{subsec::transience} using Propositions~\ref{prop::multiradial_resampling} and~\ref{prop::radialSLE_transience} (transience) from Appendix~\ref{app:transience}.

\begin{proof}[Proof of Theorem~\ref{thm::multiradial_resampling}]
With Propositions~\ref{prop::multiradial_resampling} and~\ref{prop::multiradial_transience} at hand, it only remains to verify the radial resampling property. 
To this end, note that because half-watermelon $\SLE_{\kappa}$ satisfies the chordal resampling property, 
the radial one immediately follows from Proposition~\ref{prop::multiradial_resampling} and symmetry 
(we may consider either $j \neq 1$, or $j \neq p$; and Item~\ref{item::multiradial_resampling2} in Proposition~\ref{prop::multiradial_resampling} also holds for $\gamma^{p}$ instead of $\gamma^{1}$ with identical proof).
\end{proof}

\subsection{Marginal and conditional laws --- proof of Proposition~\ref{prop::multiradial_resampling}}
\label{subsec::resampling}

Recall that $\nradpartfn{p}{\mu}$ is the partition function defined in~\eqref{eq:multiradial_partition_function}--\eqref{eqn::Gmu_cov}, 
and $\PPspiral{p}$ denotes the law of $p$-radial $\SLE_{\kappa}^{\mu}$ with spiraling rate $\mu$. 
To prove Proposition~\ref{prop::multiradial_resampling}, we first derive a variant of Item~\ref{item::multiradial_halfwatermelonMEAS} of Theorem~\ref{thm::multiradial_halfwatermelon},
stated in Lemma~\ref{lem::multiradial_spiral_halfwatermelon} below,  whose proof idea is similar to that in the proof of Theorem~\ref{thm::multiradial_halfwatermelon}.

\begin{lemma} \label{lem::multiradial_spiral_halfwatermelon}
Fix $\kappa\in (0,4], \, p\ge 2$, a $p$-polygon $(\Omega; \bs{x})$ with $\bs{x}=( x^1, \ldots, x^p)$, and denote $\hat{\bs{x}} := (x^2, \ldots, x^p)$. 
Suppose $U\subset \Omega$ is simply connected, has the points $\hat{\bs{x}}$
on its boundary, and has a positive distance from~$x^1$. 
Then, as $z\to x^1$, the marginal law of $\hat{\bs{\gamma}}=(\gamma^2, \ldots, \gamma^p)$ under the law $\PPspiral{p}$ of $p$-radial $\SLE_{\kappa}^{\mu}$ with spiraling rate $\mu$ in $(\Omega; \bs{x}; z)$  
converges weakly to the law $\QQfusion{p-1}$ of half-$(p-1)$-watermelon $\SLE_\kappa$ in $(\Omega; \hat{\bs{x}}; x^1)$, up to the exit time from $U$.
\end{lemma}

\begin{proof}
By conformal invariance of both measures $\PPspiral{p}$ and $\QQfusion{p-1}$, it suffices to consider the polygon $(\Omega; \bs{x})=(\U; \ee^{\ii\bs{\theta}})$ with $\bs{\theta} = (\theta^1, \ldots, \theta^p)\in\LX_{p}$. 
Denote $\bs{s}=(t_2, \ldots, t_p)\in \R_{>0}^{p-1}$; 
\begin{align*}
\hat{g}_{\bs{s}}:=g_{(0,t_2, \ldots, t_p)}
\qquad \textnormal{and} \qquad  
\hat{\mslitdriv}_{\bs{s}}^j:=\mslitdriv_{(0,t_2, \ldots, t_p)}^j, \quad \textnormal{for }2\le j\le p, 
\end{align*}
and let $\hat{\covmap}_{\bs{s}}$ be the covering map of $\hat{g}_{\bs{s}}$. 
By Lemma~\ref{lem::multitime_mart_nradial_spiral_z_RN}, 
the marginal law of $\hat{\bs{\gamma}}$ under $\PPspiral{p}(\U;\bs{x};z)$ 
is the same as $\PPspiral{(p-1)}(\U;\hat{\bs{x}};z)$ tilted by 
\begin{align*}
M_{\bs{s}} := \; &  |\hat{g}_{\bs{s}}'(z)|^{\frac{2p-1}{2\kappa}} \, (\hat{\covmap}_{\bs{s}}'(\theta^1))^\mathfrak{b} \, \exp\Big(\! -\frac{\mu}{\kappa} \arg \hat{g}_{\bs{s}}'(z) \Big) \frac{\nradpartfn{p}{\mu}\big(\U; \exp(\ii \hat{\covmap}_{\bs{s}}(\theta^1)),\exp(\ii\hat{\mslitdriv}_{\bs{s}}^2), \ldots, \exp(\ii\hat{\mslitdriv}_{\bs{s}}^p); \hat{g}_{\bs{s}}(z)\big)}{\nradpartfn{(p-1)}{\mu}\big(\U; \exp(\ii\hat{\mslitdriv}_{\bs{s}}^2), \ldots, \exp(\ii\hat{\mslitdriv}_{\bs{s}}^p); \hat{g}_{\bs{s}}(z)\big)}  
\end{align*}
Using equations~(\ref{eqn:spiralpartitionfunction},~\ref{eqn::nradpartfn_mu_z}) and the identities and~(\ref{eqn::CR_H_and_U},~\ref{eqn::Poissonkernel_disc},~\ref{eqn::Poissonkernel_disc_int}) $\mathfrak{h}_p=\frac{p(p+2)}{\kappa}-\frac{p}{2}$, we obtain
\begin{align*}
M_{\bs{s}} = \; & 
\big( H(\U;\hat{g}_{\bs{s}}(z),\exp(\ii\hat{\covmap}_{\bs{s}}(\theta^1)))\big)^{\frac{p-1}{\kappa}+\mathfrak{b}} 
\bigg(\prod_{j=2}^{p} \frac{H(\U; \hat{g}_{\bs{s}}(z), \exp(\ii\hat{\mslitdriv}_{\bs{s}}^j))}{\CR(\U; \hat{g}_{\bs{s}}(z))H(\U; \exp(\ii \hat{\covmap}_{\bs{s}}(\theta^1)), \exp(\ii\hat{\mslitdriv}_{\bs{s}}^j))} \frac{|\hat{g}'_{\bs{s}}(z)|}{\hat{\covmap}_{\bs{s}}'(\theta^1)}\bigg)^{\frac{1}{\kappa}}
\\
\; &  \times  \Big( \frac{\CR(\U;\hat{g}_{\bs{s}}(z))}{|\hat{g}_{\bs{s}}'(z)|} \Big)^{-\frac{1}{2\kappa}} \, (\hat{\covmap}_{\bs{s}}'(\theta^1))^{\frac{p-1}{\kappa}+\mathfrak{b}}   \, \exp \Big( \frac{\mu}{\kappa} \Big( 2 \arg(1-\hat{g}_{\bs{s}}(z)\exp(-\ii \hat{\covmap}_{\bs{s}}(\theta^1))) + \hat{\covmap}_{\bs{s}}(\theta^1) - \arg \hat{g}_{\bs{s}}'(z) \Big) \Big),
\end{align*}
Using then the identities 
\begin{align*}
\CR(\U; \hat{g}_{\bs{s}}(z)) = \; & |\hat{g}_{\bs{s}}'(z)| \, \CR(\U \setminus \hat{\bs{\gamma}}_{\bs{s}}; z) , \\
H(\U\setminus \hat{\bs{\gamma}}_{\bs{s}};z,x^1)= \; & \hat{\covmap}_{\bs{s}}'(\theta^1) \, H(\U;\hat{g}_{\bs{s}}(z),\exp(\ii\hat{\covmap}_{\bs{s}}(\theta^1))) ,
\end{align*}
we then find the Radon-Nikodym derivative of the marginal law of $\hat{\bs{\gamma}}$ as
\begin{align} \label{eqn::multiradial_spiral_halfwatermelon_aux1}
\begin{split}
\frac{\ud \PPspiral{p}(\U; \bs{x}; z)}{\ud \PPspiral{(p-1)}(\U;\hat{\bs{x}};z)} (\hat{\bs{\gamma}})
= \frac{M_{\bs{s}}}{M_{\bs{0}}} 
= \; & \bigg(\prod_{j=2}^p \frac{Y_{\bs{s}}^{j}(z, x^1)}{Y_{\bs{0}}^{j}(z, x^1)} \bigg)^{\frac{1}{\kappa}}
\; \Big( \frac{\CR(\U \setminus \hat{\bs{\gamma}}_{\bs{s}}; z)}{\CR(\U; z)} \Big)^{-\frac{1}{2\kappa}} \; \Big( \frac{H(\U\setminus \hat{\bs{\gamma}}_{\bs{s}};z,x^1)}{H(\U;z,x^1)} \Big)^{\frac{p-1}{\kappa}+\mathfrak{b}}   \\
\; & \times 
\exp \left( \frac{\mu}{\kappa} (I_{\bs{s}}(z, x^1)-I_{\bs{0}}(z, x^1)) \right) ,
\\[.5em]
\textnormal{where}
\quad 
Y_{\bs{s}}^{j}(z, x^1) := \; & \frac{|\hat{g}_{\bs{s}}'(z)|}{\hat{\covmap}_{\bs{s}}'(\theta^1)}\frac{H(\U; \hat{g}_{\bs{s}}(z), \exp(\ii\hat{\mslitdriv}_{\bs{s}}^j))}{\CR(\U; \hat{g}_{\bs{s}}(z)) \, H(\U; \exp(\ii\hat{\covmap}_{\bs{s}}(\theta^1)), \exp(\ii\hat{\mslitdriv}_{\bs{s}}^j))}, \\
I_{\bs{s}}(z, x^1):= \; & 2 \arg(1-\hat{g}_{\bs{s}}(z)\exp(-\ii \hat{\covmap}_{\bs{s}}(\theta^1))) + \hat{\covmap}_{\bs{s}}(\theta^1) - \arg \hat{g}_{\bs{s}}'(z).
\end{split}
\end{align}
We already know that the law $\smash{\PPspiral{(p-1)}(\Omega; \hat{\bs{x}}; z)}$ 
converges weakly to the law $\QQfusion{p-1}(\Omega; \hat{\bs{x}}; x^1)$,  
up to the exit time from $U$ (thanks to Item~\ref{item::multiradial_halfwatermelonMEAS} of Theorem~\ref{thm::multiradial_halfwatermelon}). 
It thus suffices to show that 
$M_{\bs{s}}/M_{\bs{0}} \to 1$ as $z\to x^1$, 
and 
that $M_{\bs{s}}/M_{\bs{0}}$ stays uniformly bounded for any $\bs{s}$ before the exit time from $U$.

\medskip

{\bf Step~1.}
We first prove that a.s., $M_{\bs{s}}/M_{\bs{0}}\to 1$ as $z\to x^1$. 
On the one hand, we have
\begin{align} \label{eqn::multiradial_spiral_halfwatermelon_lim1}
\begin{split}
\lim_{z\to x^1} I_{\bs{s}}(z, x^1) = \; &  2 \arg \big(1-\hat{g}_{\bs{s}}(\ee^{\ii\theta^1})\exp(-\ii \hat{\covmap}_{\bs{s}}(\theta^1))\big) + \hat{\covmap}_{\bs{s}}(\theta^1) - \arg \hat{g}_{\bs{s}}'(\ee^{\ii\theta^1})=0,
\\ 
\lim_{z\to x^1} I_{\bs{0}}(z, x^1) = \; & 0,
\end{split}
\end{align}
using the identities similar to~\eqref{eq:arg_g}. 
On the other hand,~(\ref{eqn::CR_cvg},~\ref{eqn::Poissonkernel_cvg},~\ref{eqn::PoissonKer_convergence}) together imply
\begin{align} \label{eqn::multiradial_spiral_halfwatermelon_lim2}
\begin{cases}
\displaystyle 
\lim_{z\to x^1} \frac{H(\U\setminus \hat{\bs{\gamma}}_{\bs{s}};z,x^1)}{H(\U;z,x^1)} = 1 , \\[.9em]
\displaystyle
\lim_{z\to x^1} \frac{\CR(\U \setminus \hat{\bs{\gamma}}_{\bs{s}}; z)}{\CR(\U; z)} = 1 , 
\end{cases}
\qquad 
\begin{cases}
\displaystyle
\lim_{z\to x^1}Y_{\bs{0}}^{\ell}(z,x^1) = 1 , \\[.5em]
\displaystyle
\lim_{z\to x^1} Y_{\bs{s}}^{\ell}(z,x^1)=1 , \quad \textnormal{for } 2\le \ell \le p ,
\end{cases}
\end{align} 
Plugging~(\ref{eqn::multiradial_spiral_halfwatermelon_lim1},~\ref{eqn::multiradial_spiral_halfwatermelon_lim2}) into~\eqref{eqn::multiradial_spiral_halfwatermelon_aux1}, we obtain $M_{\bs{s}}/M_{\bs{0}}\to 1$.

\medskip

{\bf Step~2.}
Second, we show that $M_{\bs{s}}/M_{\bs{0}}$ stays uniformly bounded as $z\to y$ when $\hat{\bs{\gamma}}_{\bs{s}}\subset U$. Using the monotonicity $H(\U\setminus \hat{\bs{\gamma}}_{\bs{s}};z,x^1)\le H(\U;z,x^1)$ and $\CR(\U \setminus \hat{\bs{\gamma}}_{\bs{s}}; z)\ge \CR(\U \setminus U; z)$, we have 
\begin{align} \label{eqn::multiradial_spiral_halfwatermelon_aux2}
\frac{M_{\bs{s}}}{M_{\bs{0}}}\le \bigg(\prod_{j=2}^p \frac{Y_{\bs{s}}^{j}(z, x^1)}{Y_{\bs{0}}^{j}(z, x^1)} \bigg)^{\frac{1}{\kappa}} \bigg( \frac{\CR(\U \setminus U; z)}{\CR(\U; z)} \bigg)^{-\frac{1}{2\kappa}} \exp \left( \frac{\mu}{\kappa} (I_{\bs{s}}(z, x^1)-I_{\bs{0}}(z, x^1)) \right).
\end{align}
As before, we can now control $M_{\bs{s}}/M_{\bs{0}}$ and conclude the proof by using the following estimates: 
\begin{itemize}
\item Picking $\delta>0$ such that $B(x^1,16\delta)\cap U=\emptyset$ and combining~(\ref{eqn::muliradial_halfwatermelon_RN_aux1}--\ref{eqn::muliradial_halfwatermelon_RN_aux3}), we have
\begin{align} \label{eqn::multiradial_spiral_halfwatermelon_aux3}
Y_{\bs{s}}^{j}(z, x^1)\lesssim 1, \quad \textnormal{ for } z\in B(x^1,\delta).
\end{align}

\item Lemma~\ref{lem::CR_cvg} shows that $z \mapsto \frac{\CR(\U \setminus U; z)}{\CR(\U; z)}$ 
is a continuous function of $z\in \overline{\U} \setminus \overline{U}$. Thus, we have
\begin{align} \label{eqn::multiradial_spiral_halfwatermelon_aux4}
\frac{\CR(\U \setminus U; z)}{\CR(\U; z)} \gtrsim 1 , \quad \textnormal{ for } z\in \overline{B(x^1,\delta)}.
\end{align}
\item From~(\ref{eqn::muliradial_halfwatermelon_RN_aux4}--\ref{eqn::muliradial_halfwatermelon_RN_control3}) we have
\begin{align} \label{eqn::multiradial_spiral_halfwatermelon_aux5}
|I_{\bs{s}}(z, x^1)|\lesssim 1, \quad \textnormal{ for } z\in B(x^1,\delta/100).
\end{align}
\end{itemize}
From~(\ref{eqn::multiradial_spiral_halfwatermelon_aux2}--\ref{eqn::multiradial_spiral_halfwatermelon_aux5}), 
we see that $M_{\bs{s}}/M_{\bs{0}}$ stays uniformly bounded when $z\in B(x^1,\delta/100)$ and $\hat{\bs{\gamma}}_{\bs{s}}\subset U$. 
\end{proof}

\begin{proof}[Proof of Proposition~\ref{prop::multiradial_resampling}]
By conformal invariance, it suffices to consider the polygon $(\Omega; \bs{x}; z)=(\U; \ee^{\ii\bs{\theta}}; 0)$ with $\bs{\theta}= (\theta^1, \ldots, \theta^p)\in\LX_{p}$. 
For each $j\in \{1,\ldots, p\}$, let $\gamma^{j}$ be radial $\SLE_{\kappa}$ in $(\U; \ee^{\ii\theta^j}; 0)$ 
and let $\mathsf{P}_p$ 
be the probability measure on $\bs{\gamma}=(\gamma^{1}, \ldots, \gamma^{p})$ under which the curves are independent. 
Recall from Definition~\ref{def::multiradialSLEspiral} that 
$\PPspiral{p}$ is defined as $\mathsf{P}_p$ tilted by the $p$-time parameter local martingale
$\smash{M(\nradpartfn{p}{\mu})}$ defined in Eq.~\eqref{eqn::multiradialSLE_mart}.
By Lemmas~\ref{lem::multiradial_marginal}~\&~\ref{lem::multiradial_collision}, for any $t_1 \geq 0$ the law of $\gamma^1|_{[0,t_1]}$ under $\PPspiral{p}$ is the same as 
radial $\SLE_\kappa$ in $\smash{(\U; \ee^{\ii\theta^1}; 0)}$ tilted by 
$M_{(t_1, 0, \ldots, 0)}(\nradpartfn{p}{\mu})$, given by Eq.~\eqref{eqn::multiradialSLE_mart_cor}. 
Thus, the \emph{marginal law} of $\gamma^1$ is radial $\SLE_{\kappa}^{\mu}(2, \ldots, 2)$ as derived in Lemma~\ref{lem::multiradial_marginal}. 
This proves Item~\ref{item::multiradial_resampling1} of Proposition~\ref{prop::multiradial_resampling}.

Next, by the domain Markov property, the \emph{conditional law} of 
$\smash{(\gamma^1|_{[t_1,\infty)}, \gamma^2, \ldots, \gamma^p)}$ given $\smash{\gamma^1_{[0,t_1]}}$ 
is $p$-radial $\SLE_{\kappa}^{\mu}$ in the remaining domain $\smash{(\U\setminus\gamma^1_{[0,t_1]}; \gamma^1_{t_1}, \hat{\bs{x}}; 0)}$. 
Moreover, we have $\gamma^1_{t_1}\to 0$ as $t_1\to\infty$ almost surely due to Proposition~\ref{prop::radialSLE_transience} in Appendix~\ref{app:transience}, 
and moreover, the curve $\gamma^1$ is simple (since $\kappa \leq 4$). 
Therefore, Lemma~\ref{lem::multiradial_spiral_halfwatermelon} shows that the marginal law of the last $(p-1)$ curves of $p$-radial $\SLE_{\kappa}^{\mu}$ 
in the polygon $(\U\setminus\gamma^1_{[0,t_1]}; \gamma^1_{t_1}, \hat{\bs{x}}; 0)$ 
converges weakly to half-$(p-1)$-watermelon $\SLE_{\kappa}$ in the polygon $(\U\setminus\gamma^1; \hat{\bs{x}}; 0)$ as $t_1\to\infty$. 
The conditional law of $\hat{\bs{\gamma}}=(\gamma^2, \ldots, \gamma^p)$ given $\gamma^1$ is thus that of half-$(p-1)$-watermelon $\SLE_{\kappa}$, as claimed. 
This shows Item~\ref{item::multiradial_resampling2} and finishes the proof of Proposition~\ref{prop::multiradial_resampling}.
\end{proof}

\subsection{Boundary perturbation --- proof of Proposition~\ref{prop::multiradial_bp}}
\label{subsec::bp}

Now, we will derive the boundary perturbation property for radial $\SLE_{\kappa}^{\mu}(2, \ldots, 2)$, see Lemma~\ref{lem::radialSLEkapparho_bp}. (See~\cite{Dubedat:SLE_kappa_rho_martingales_and_duality} for the analogue boundary perturbation property for $\SLE_{\kappa}(\bs{\rho})$ in the chordal case.)
It will be a key ingredient in the proof of Proposition~\ref{prop::multiradial_bp} in the end of this section,
together with the boundary perturbation property for half-$n$-watermelon $\SLE_{\kappa}$ (Proposition~\ref{prop::halfwatermelon_bp}) and Theorem~\ref{thm::multiradial_resampling}.

\begin{lemma}[Boundary perturbation]\label{lem::radialSLEkapparho_bp}
Fix $\kappa\in (0,4], \, \mu\in\R, \, p\ge 2$ and a $p$-polygon $(\Omega; \bs{x}; z)$ with $\bs{x}=( x^1, \ldots, x^p)$ and $z\in\Omega$. 
Denote $\hat{\bs{x}}=(x^2, \ldots, x^p)$. 
Consider radial $\SLE_{\kappa}^{\mu}(2, \ldots, 2)$ curve $\gamma$ in $(\Omega; \bs{x};z)$. 
Suppose $K$ is a compact subset of $\overline{\Omega}$ such that $\Omega\setminus K$ is simply connected and contains $z$, 
and $K$ has a positive distance from $\{x^1, \ldots, x^p\}$. 
Then, the radial $\SLE_{\kappa}^{\mu}(2, \ldots, 2)$ probability measure in the smaller polygon $(\Omega\setminus K; \bs{x}; z)$ 
is absolutely continuous with respect to that in $(\Omega; \bs{x}; z)$, with Radon-Nikodym derivative
\begin{align}
\label{eqn::radialSLEkapparho_bp_RN}
\frac{\nradpartfn{p}{\mu}(\Omega;\bs{x};z)}{\nradpartfn{p}{\mu}(\Omega\setminus K;\bs{x};z)} \, 
\frac{\LZfusion{p-1} (\Omega\setminus (\gamma\cup K);\hat{\bs{x}};z)}{\LZfusion{p-1} (\Omega\setminus \gamma;\hat{\bs{x}};z)} \,  \one\{\gamma\cap K=\emptyset \} \, \exp\Big(\frac{\mathfrak{c}}{2} \blm(\Omega;\gamma,K)\Big). 
\end{align}
\end{lemma}

\begin{proof}
It again suffices to consider the polygon $(\Omega; \bs{x}; z)=(\U; \ee^{\ii\bs{\theta}}; 0)$ with $(\theta^1, \ldots, \theta^p)\in\LX_{p}$, and a subset $K \subset \overline{\U}$ as in the statement. 
Let $g_K \colon \U\setminus K \to \U$ (resp.~$g_{t,K} \colon \U \setminus g_t(K) \to \U$ for $t \geq 0$)
be the conformal bijection normalized at the origin and $\covmap_K$ (resp.~$\covmap_{t,K}$) its covering map.

We claim that law of radial $\SLE_{\kappa}^{\mu}(2, \ldots, 2)$, whose driving function $\xi_t$ in $\U$ solves the SDE~\eqref{eqn::SLEkapparhomu_SDE} with $\rho_j=2$ for all $j$, 
in the smaller domain $\U\setminus K$ is the same as that in $\U$ tilted by 
\begin{align*}
N_t = \; & \one\{\gamma_{[0,t]} \cap K=\emptyset \} \exp\Big(\frac{\mathfrak{c}}{2} \blm(\U;\gamma_{[0,t]},K)\Big) 
\, (g_{t,K}'(0))^{\tilde{\mathfrak{b}} +\frac{p^2-1}{2\kappa}-\frac{\mu^2}{2\kappa}}
( \covmap_{t,K}'(\xi_t) )^\mathfrak{b} \Big( \prod_{j=2}^p \covmap_{t, K}'(V_t^j) \Big)^\mathfrak{b}
\\
\; & \times \frac{\nradpartfn{p}{\mu}(\covmap_{t,K}(\xi_t),\covmap_{t,K}(V_t^2),\ldots,\covmap_{t,K}(V_t^p))}{\nradpartfn{p}{\mu}(\xi_t,V_t^2,\ldots,V_t^p)}
\\
= \; & \one\{\gamma_{[0,t]} \cap K=\emptyset \} \exp\Big(\frac{\mathfrak{c}}{2} \blm(\U;\gamma_{[0,t]},K)\Big) 
\, (g_{t,K}'(0))^{\tilde{\mathfrak{b}} +\frac{p^2-1}{2\kappa}-\frac{\mu^2}{2\kappa}}
(\covmap_{t,K}'(\xi_t) )^\mathfrak{b} \Big( \prod_{j=2}^p \covmap_{t, K}'(V_t^j) \Big)^\mathfrak{b}
\\
\; & \times \exp\Big( \frac{\mu}{\kappa} ( \covmap_{t,K}(\xi_t)-\xi_t ) \Big) 
\prod_{j=2}^{p} \exp\Big( \frac{\mu}{\kappa} ( \covmap_{t,K}(V_t^j)-V_t^j ) \Big)
\\
\; & \times
\bigg( \underset{1\leq j \leq p}{\prod} \, \frac{\sin \Big(\tfrac{1}{2} \big( \covmap_{t,K}(V_t^j) - \covmap_{t,K}(\xi_t) \big) \Big)}{\sin \Big(\tfrac{1}{2} \big( V_t^j - \xi_t \big) \Big)} \bigg)^{\frac{2}{\kappa}} 
\bigg( \underset{2\leq i < j \leq p}{\prod} \, \frac{\sin \Big(\tfrac{1}{2} \big( \covmap_{t,K}(V_t^j) - \covmap_{t,K}(V_t^i) \big) \Big)}{\sin \Big(\tfrac{1}{2} \big( V_t^j - V_t^i \big) \Big)} \bigg)^{\frac{2}{\kappa}} ,
\end{align*}
which is moreover a uniformly integrable martingale.
First, we verify that $(N_t)_{t\ge 0}$ is a local martingale. 
Standard computations using the radial Loewner equation~\eqref{eq:single_radial_Loewner_equation} and the evolution~\eqref{eqn::mt_def} yield
\begin{align*}
\frac{\partial_t g'_{t,K}(0)}{g'_{t,K}(0)} = \; & (\covmap_{t,K}'(\xi_t))^2-1, 
&& \partial_t \blm(\U;\gamma_{[0,t]},K) = -\frac{1}{3} \LS \covmap_{t,K}(\xi_t) + \frac{1}{6}(1-\covmap'_{t,K}(\xi_t)^2), 
\\
\partial_t \covmap_{t,K}(\xi_t) = \; & -3\covmap_{t,K}''(\xi_t), 
&& \qquad \; \frac{\partial_t \covmap'_{t,K}(\xi_t)}{\covmap'_{t,K}(\xi_t)} = \frac{1}{2} \bigg( \frac{\phi''_{t,K}(\xi_t)}{\covmap'_{t,K}(\xi_t)}\bigg)^2 - \frac{4}{3}\frac{\phi'''_{t,K}(\xi_t)}{\covmap'_{t,K}(\xi_t)} 
- \frac{1}{6} \big((\covmap'_{t,K}(\xi_t))^2 - 1 \big), 
\\
\partial_t V_t^j = \; & \cot \bigg(\frac{V_t^j-\xi_t}{2}\bigg), 
&& \qquad \partial_t \covmap_{t,K}(V_t^j) = (\covmap'_{t,K}(\xi_t))^2 \, \cot \bigg(\frac{\covmap_{t,K}(V_t^j)- \covmap_{t,K}(\xi_t)}{2}\bigg),
\end{align*}
\vspace*{-6mm}
\begin{align*}
\textnormal{and} \qquad  \qquad 
\frac{\partial_t \covmap'_{t,K}(V_t^j)}{\covmap'_{t,K}(V_t^j)} = \; & \frac{1}{2} \csc^2 \bigg(\frac{V_t^j-\xi_t}{2}\bigg)
- \frac{1}{2} (\covmap'_{t,K}(\xi_t))^2 \, \csc^2 \bigg(\frac{\covmap_{t,K}(V_t^j)- \covmap_{t,K}(\xi_t)}{2}\bigg).
\end{align*}
Combining these with~\eqref{eqn::SLEkapparhomu_SDE} yields 
(writing $\bs{V}_t = (V_t^2, \ldots, V_t^p)$)
\begin{align}\label{eq:mgleN}
\frac{\ud N_t}{N_t} 
= \; & \sqrt{\kappa} \, \mathfrak{b} \, \frac{\phi''_{t,K} (\xi_t)}{\covmap'_{t,K} (\xi_t)} \, \ud B_t
+ \frac{\mu}{\sqrt{\kappa}} ( \covmap'_{t,K}(\xi_t) - 1 ) \, \ud B_t 
\\
\; & + \frac{1}{\sqrt{\kappa}} \sum_{j=2}^{p} \bigg( \cot \bigg(\frac{V_t^j-\xi_t}{2}\bigg) - \covmap'_{t,K}(\xi_t) \, \cot \bigg(\frac{\covmap_{t,K}(V_t^j)- \covmap_{t,K}(\xi_t)}{2}\bigg) \bigg) \ud B_t 
\; =: \; F_K(\xi_t,\bs{V}_t) \, \ud B_t,
\nonumber
\end{align}
which implies that $(N_t)_{t\ge 0}$ is a local martingale.

Second, we prove that $N=(N_t)_{t\ge 0}$ is a uniformly integrable martingale and the law of radial $\SLE_{\kappa}^{\mu}(2, \ldots, 2)$ in the smaller domain $\U\setminus K$ is the same as that in $\U$ tilted by $N$. 
Denote by $T_n$ the first time when the distance between the curve and $K$ reaches $1/n$. 
Then, $(N_{t\wedge T_n})_{t\ge 0}$ is a bounded martingale. 
Denote by $\mathsf{Q}$ the probability measure of $\SLE_{\kappa}^{\mu}$ in $(\U; \ee^{\ii\theta}; 0)$ and by $\mathsf{Q}_n^*$ the probability measure obtained by weighting $\mathsf{Q}$ by $N_{T_n}$. 
Combined with~\eqref{eq:mgleN}, Girsanov's theorem yields
\begin{align*}
\ud B_t = \; & \ud B^*_t + F_K(\xi_t,\bs{V}_t) \, \ud t, \\
\ud \covmap_{t,K}( \xi_t ) = \; & \sqrt{\kappa}  \, \covmap'_{t,K}(\xi_t) \, \ud B_t^* + \mu \, (\covmap'_{t,K}(\xi_t))^2 \ud t + \sum_{j=2}^{p}  \cot \bigg( \frac{\covmap_{t,K}( \xi_t)-\covmap_{t,K}(V_t^j)}{2}\bigg) (\covmap'_{t,K}(\xi_t))^2 \ud t,
\end{align*}
where $B_t^*$ is Brownian motion under $\mathsf{Q}_n^*$. 
Thus, $\mathsf{Q}_n^*$ is nothing but the law of a (time-changed) $\SLE_{\kappa}^{\mu}(2, \ldots, 2)$ process in $\U\setminus K$ up to time $T_n$. 
Moreover, the sequence $\{\mathsf{Q}_n^*\}_n$ is compatible by definition, so Kolmogorov's extension theorem guarantees that
there exists a probability measure $\mathsf{Q}_{\infty}^*$ such that $\mathsf{Q}_{\infty}^*$ is the same as $\mathsf{Q}_n^*$ up to $T_n$ for all $n$. 
In other words, $\mathsf{Q}_{\infty}^*$ is the same as the law of radial $\SLE_{\kappa}^{\mu}(2, \ldots, 2)$ in $\U\setminus K$ up to $T_n$ for all $n$. 
Note also that radial $\SLE_{\kappa}^{\mu}(2, \ldots, 2)$ in $\U\setminus K$ has a positive distance from $K$ almost surely (see Proposition~\ref{prop::radialSLE_avoidboundary}). 
Thus, $\mathsf{Q}_{\infty}^*$ is indeed the same as the law of radial $\SLE_{\kappa}^{\mu}(2, \ldots, 2)$ in $\U\setminus K$ for all time. 
This shows that radial $\SLE_{\kappa}^{\mu}(2, \ldots, 2)$ in $\U\setminus K$ is the same as that in $\U$ tilted by $(N_t)_{t\ge 0}$, 
and $(N_t)_{t\ge 0}$ is a uniformly bounded martingale, as claimed.  

Finally, let us consider the a.s.~limiting value $\underset{t\to\infty}{\lim} \, N_t$. 
From the definitions~(\ref{eqn:spiralpartitionfunction},~\ref{eqn::Gmu_cov}), we obtain
\begin{align}
N_t = \; & \one\{\gamma_{[0,t]} \cap K=\emptyset \} \exp\Big(\frac{\mathfrak{c}}{2} \blm(\U;\gamma_{[0,t]},K)\Big) 
\, (g_{t,K}'(0))^{\tilde{\mathfrak{b}} +\frac{p^2-1}{2\kappa} - \frac{\mu^2}{2\kappa}}
 \Big( \covmap_{t,K}'(\xi_t) \, \prod_{j=2}^p \covmap_{t, K}'(V_t^j) \Big)^\mathfrak{b}
\nonumber
\\
\; & \times \frac{\nradpartfn{p}{0}(\U;\covmap_{t,K}(\xi_t),\covmap_{t,K}(V_t^2),\ldots,\covmap_{t,K}(V_t^p))}{\nradpartfn{p}{0}(\U;\xi_t,V_t^2,\ldots,V_t^p)}
\; \exp\Big( \frac{\mu}{\kappa} ( \covmap_{t,K}(\xi_t)-\xi_t ) \Big) 
\prod_{j=2}^{p} \exp\Big( \frac{\mu}{\kappa} ( \covmap_{t,K}(V_t^j)-V_t^j ) \Big)
\nonumber
\\
= \; & \one\{\gamma_{[0,t]} \cap K=\emptyset \} \exp\Big(\frac{\mathfrak{c}}{2} \blm(\U;\gamma_{[0,t]},K)\Big) 
\, (g_{t,K}'(0))^{-\frac{\mu^2}{2\kappa}}
\nonumber
\\
\; & \times \frac{\nradpartfn{p}{0}(\U\setminus (K\cup \gamma_{[0,t]}); \gamma_t, \hat{\bs{x}};0)}{\nradpartfn{p}{0}(\U\setminus \gamma_{[0,t]}; \gamma_t, \hat{\bs{x}};0)} 
\; \exp\Big( \frac{\mu}{\kappa} ( \covmap_{t,K}(\xi_t)-\xi_t ) \Big) 
\prod_{j=2}^{p} \exp\Big( \frac{\mu}{\kappa} ( \covmap_{t,K}(V_t^j)-V_t^j ) \Big) .
\label{eqn::radialSLEkapparho_bp_aux0} 
\end{align}
On the one hand, by Corollary~\ref{cor::phitKtoid} in Appendix~\ref{app:transience}, we have
\begin{align}\label{eqn::radialSLEkapparho_bp_aux1}
\lim_{t\to \infty} g_{t,K}'(0)=1, \qquad \lim_{t\to \infty} ( \covmap_{t,K}(\xi_t)-\xi_t)=0, \qquad \lim_{t\to \infty} ( \covmap_{t,K}(V_t^j)-V_t^j )=0 , \quad \textnormal{ for } 2\le j\le p.
\end{align}
On the other hand, the convergence~\eqref{eqn::LGtoLZfusion} in Theorem~\ref{thm::multiradial_halfwatermelon}(\ref{item::multiradial_halfwatermelonPF}) gives 
\begin{align} \label{eqn::radialSLEkapparho_bp_aux3}
\; & \lim_{t\to \infty} \frac{\nradpartfn{p}{0}(\U\setminus (K\cup\gamma_{[0,t]}); \gamma_t, \hat{\bs{x}};0)}{\nradpartfn{p}{0}(\U\setminus \gamma_{[0,t]}; \gamma_t, \hat{\bs{x}};0)}\notag\\
= \; & \lim_{t\to\infty}\Big(\frac{\CR(\U\setminus (K\cup\gamma_{[0,t]}); 0)}{\CR(\U\setminus \gamma_{[0,t]}; 0)}\Big)^{\mathfrak{h}_n - \tilde{\mathfrak{b}} - \frac{n^2-1}{2\kappa}} \; \frac{\LZfusion{p-1} (\U\setminus (K\cup\gamma);\hat{\bs{x}};0)}{\LZfusion{p-1} (\U\setminus \gamma;\hat{\bs{x}};0)}\notag\\
= \; & \lim_{t\to\infty} (g_{t, K}'(0))^{\tilde{\mathfrak{b}} + \frac{n^2-1}{2\kappa} - \mathfrak{h}_n} \; \frac{\LZfusion{p-1} (\U\setminus (K\cup\gamma);\hat{\bs{x}};0)}{\LZfusion{p-1} (\U\setminus \gamma;\hat{\bs{x}};0)}
\; =\; \frac{\LZfusion{p-1} (\U\setminus (K\cup\gamma);\hat{\bs{x}};0)}{\LZfusion{p-1} (\U\setminus \gamma;\hat{\bs{x}};0)}. 
\end{align}
Plugging~(\ref{eqn::radialSLEkapparho_bp_aux1},~\ref{eqn::radialSLEkapparho_bp_aux3}) into~\eqref{eqn::radialSLEkapparho_bp_aux0}, we obtain
\begin{align*}
N_{\infty} := \lim_{t\to\infty}N_t =
 \one\{\gamma\cap K=\emptyset\} \, \exp\Big(\frac{\mathfrak{c}}{2} \blm(\U; \gamma, K)\Big) \, \frac{\LZfusion{p-1} \left(\U\setminus \left(K\cup\gamma\right);\hat{\bs{x}};0\right)}{\LZfusion{p-1} \left(\U\setminus \gamma;\hat{\bs{x}};0\right)}.
\end{align*}
We thus conclude that radial $\SLE_{\kappa}^{\mu}(2, \ldots, 2)$ is the smaller domain $\U\setminus K$ is absolutely continuous with respect to that in $\U$, 
with Radon-Nikodym derivative given by $N_{\infty}/N_0$. 
Lastly, note that 
\begin{align*}
N_0 = (g_K'(0))^{\tilde{\mathfrak{b}} +\frac{p^2-1}{2\kappa}-\frac{\mu^2}{2\kappa}}
\Big( \prod_{j=1}^p \covmap_{K}'(\theta^j) \Big)^\mathfrak{b}
\; \frac{\nradpartfn{p}{\mu}(\covmap_K(\theta^1), \ldots, \covmap_K(\theta^p))}{\nradpartfn{p}{\mu}(\theta^1, \ldots, \theta^p)}=\frac{\nradpartfn{p}{\mu}(\U\setminus K; \bs{x}; 0)}{\nradpartfn{p}{\mu}(\U; \bs{x}; 0)} 
\end{align*}
by~\eqref{eqn::Gmu_cov}, so $N_{\infty}/N_0$ gives the sought Radon-Nikodym derivative in~\eqref{eqn::radialSLEkapparho_bp_RN}. 
\end{proof}

\begin{proof}[Proof of Proposition~\ref{prop::multiradial_bp}]
The case where $p=1$ is addressed in Lemma~\ref{lem::radialSLE_bp}, so we focus on $p\ge 2$. 
The proof is a combination of Theorem~\ref{thm::multiradial_resampling}, Lemma~\ref{lem::radialSLEkapparho_bp} and Proposition~\ref{prop::halfwatermelon_bp}. 
Let $\bs{\gamma}=(\gamma^1, \ldots, \gamma^p)$ and $\bs{\gamma}_K=(\gamma^1_K, \ldots, \gamma^p_K)$ 
be $p$-radial $\SLE_{\kappa}^{\mu}$ processes in $(\Omega; \bs{x};z)$ and $(\Omega\setminus K; \bs{x};z)$, respectively,
and denote
\begin{align*}
\bs{\gamma}=\bigcup_{i=1}^p\gamma^i,\qquad \hat{\bs{\gamma}}=\bigcup_{i=2}^p\gamma^i;\qquad 
\bs{\gamma}_K=\bigcup_{i=1}^p\gamma^i_K, \qquad \hat{\bs{\gamma}}_K=\bigcup_{i=2}^p\gamma^i_K.
\end{align*}
\begin{itemize}
\item From Lemma~\ref{lem::radialSLEkapparho_bp}, we see that the law of $\gamma_K^1$ is the same as the law of $\gamma^1$ weighted by 
\begin{align*}
\frac{\nradpartfn{p}{\mu}(\Omega;\bs{x};z)}{\nradpartfn{p}{\mu}(\Omega\setminus K;\bs{x};z)} \, 
\frac{\LZfusion{p-1} (\Omega\setminus (\gamma^1\cup K);\hat{\bs{x}};z)}{\LZfusion{p-1} (\Omega\setminus \gamma^1;\hat{\bs{x}};z)} \, \one\{\gamma^1\cap K=\emptyset \} \, \exp\Big(\frac{\mathfrak{c}}{2} \blm(\Omega;\gamma^1,K)\Big). 
\end{align*} 
\item From Theorem~\ref{thm::multiradial_resampling}, we see that
\begin{itemize}
\item given $\gamma^1$, the conditional law of $\hat{\bs{\gamma}}$ is the same as multichordal $\SLE_{\kappa}$ in $(\Omega\setminus \gamma^1; \hat{\bs{x}}; z)$;
\item given $\gamma^1_K$, the conditional law of $\hat{\bs{\gamma}}_K$ is the same as multichordal $\SLE_{\kappa}$ in $(\Omega\setminus (\gamma_K^1\cup K); \hat{\bs{x}}; z)$. 
\end{itemize}
\end{itemize}
Combining these facts with Proposition~\ref{prop::halfwatermelon_bp}, 
we see that on the event $\gamma^1=\gamma_K^1$, the conditional law of $\hat{\bs{\gamma}}_K$ is the same as $\hat{\bs{\gamma}}$ weighted by 
\begin{align*}
\frac{\LZfusion{p-1} (\Omega\setminus \gamma^1;\hat{\bs{x}};z)}{\LZfusion{p-1} (\Omega\setminus (\gamma^1\cup K);\hat{\bs{x}};z)} \, \one\{\hat{\bs{\gamma}}\cap K=\emptyset\} \, \exp\Big(\frac{\mathfrak{c}}{2}\blm(\Omega\setminus\gamma^1; \hat{\bs{\gamma}}, K)\Big) ,
\end{align*}
and that the law of $\bs{\gamma}_K$ is the same as the law of $\bs{\gamma}$ weighted by~\eqref{eqn::multiradialSLEbp}: 
\begin{align*}
& \; \frac{\nradpartfn{p}{\mu}(\Omega;\bs{x};z)}{\nradpartfn{p}{\mu}(\Omega\setminus K;\bs{x};z)} \, \one\{\gamma^1\cap K=\emptyset,\hat{\bs{\gamma}}\cap K=\emptyset\} \, \exp\Big(\frac{\mathfrak{c}}{2} \blm(\Omega;\gamma^1,K)+\frac{\mathfrak{c}}{2}\blm(\Omega\setminus\gamma^1; \hat{\bs{\gamma}}, K)\Big)\\
=& \; \frac{\nradpartfn{p}{\mu}(\Omega;\bs{x};z)}{\nradpartfn{p}{\mu}(\Omega\setminus K;\bs{x};z)} \, \one\{\bs{\gamma}\cap K=\emptyset\} \, \exp\Big(\frac{\mathfrak{c}}{2}\blm(\Omega; \bs{\gamma}, K)\Big), 
\end{align*}
This concludes the proof. 
\end{proof}

\subsection{Radial resampling --- proof of Proposition~\ref{prop::multiradial_transience}}
\label{subsec::transience}

Multiradial $\SLE_\kappa$ was first constructed in the so-called common (capacity) time parameterization in the work~\cite{Healey-Lawler:N_sided_radial_SLE} 
of Healey and Lawler. 
Roughly speaking, in the common parameterization each curve is locally growing at the same rate and the total capacity of the $p$ curves at time $t$ equals $pt$.
The usage of the common parameterization enabled the authors of~\cite{Healey-Lawler:N_sided_radial_SLE} to construct multiradial $\SLE_\kappa$ using weighting by an appropriately normalized Brownian loop measure term.  

\paragraph*{Common-time parameter.}
Consider a $p$-tuple 
$\bs{\gamma}_{\bs{t}} := ( \gamma^{1}_{t_1}, \ldots, \gamma^{p}_{t_p} )$ of curves.
Recall that $g_{\bs{t}}$ is the conformal map from $\smash{\U \setminus \cup_{j=1}^p \gamma^{j}_{[0,t_j]}}$ onto $\U$ normalized at the origin, i.e., $g_{\bs{t}}(0)=0$ and $g'_{\bs{t}}(0)>0$. 
We say that $\bs{\gamma}_{\bs{t}}$ has \emph{common-time parameterization} if  
\begin{align}\label{eqn::commontime_def}
a_{\bs{t}}^{j} := (\partial_{t_j} \log g'_{\bs{t}}(0))|_{\bs t = (t, \ldots, t)} = 1 
, \qquad 
\textnormal{for each} \quad 0 \leq  t < \min_{1\le i\le p}t_i
\textnormal{ and } 1\le j\le p.
\end{align}
\cite[Lemma~3.2]{Healey-Lawler:N_sided_radial_SLE} implies that 
$(\covmap_{\bs{t},j}' (\xi_{t}^{j}))^2 = 1$ for each $t$ as above, and the time change 
\begin{align}\label{eqn::multi_common}
\ud t_j(t) = \frac{\ud t}{(\covmap'_{\bs{t},j} ( \xi_{t_j(t)}^{j} ))^2} 
\end{align}
applied to the multi-time parameter $\bs{t}(t) = (t_1(t), \ldots, t_p(t))$ gives the common-time parameter, and
\begin{align}
\label{eqn::timechange_ineq}
t\le t_j(t) \le pt, \qquad \textnormal{ for all } 1\le j\le p .
\end{align}
The lifetime of the process is the first time when $\bigcup_{j=1}^p \gamma^{j}_{[0,t]}$ hits the origin or two of the curves in $\bs{\gamma}_{t}$ meet.

\begin{proof}[Proof of Proposition~\ref{prop::multiradial_transience}]
The case where $p=1$ is addressed in Proposition~\ref{prop::radialSLE_transience} in Appendix~\ref{app:transience}, so we focus on $p\ge 2$. 
Let $\bs{\gamma}=(\gamma^1, \ldots, \gamma^p)$ be $p$-radial $\SLE_{\kappa}^{\mu}$ in $(\U; \ee^{\ii\bs{\theta}};0)$. 
The transience~\eqref{eqn::transience_multitime} in the $p$-time parameterization is a consequence of the following observations:
\begin{itemize}
\item By Lemma~\ref{lem::multiradial_marginal} and Proposition~\ref{prop::radialSLE_transience}, the marginal law of $\gamma^1$ is radial $\SLE_{\kappa}^{\mu}(2, \ldots, 2)$, which is almost surely generated by a continuous curve and is almost surely transient. 

\item Thus, almost surely, the set $\gamma^1\cup\partial\U$ is locally connected,
so any conformal map $\psi \colon \U \to \U\setminus\gamma^1$ extends continuously to $\overline{\U}$ by Carath\'eodory's theorem~\cite[Theorems~2.1~\&~2.6, and Corollary~2.8]{Pommerenke:Boundary_behaviour_of_conformal_maps}). 

\item Let $\hat{\bs{\eta}}=(\eta^2, \ldots, \eta^p)$ be half-$(p-1)$-watermelon $\SLE_{\kappa}$ in $(\U; \psi^{-1}(\ee^{\ii\theta^2}), \ldots, \psi^{-1}(\ee^{\ii\theta^p}); \psi^{-1}(0))$. 
Then, by Proposition~\ref{prop::multiradial_resampling}, the conditional law of $\hat{\bs{\gamma}}=(\gamma^2, \ldots, \gamma^p)$ given $\gamma^1$ is the same as $\psi(\hat{\bs{\eta}})$. 
Since the half-$n$-watermelon $\SLE_{\kappa}$ curves $\hat{\bs{\eta}}$ are almost surely transient by~\cite[Section~4]{Miller-Sheffield:Imaginary_geometry1} 
and $\psi$ extends continuously to the boundary, the curves $\hat{\bs{\gamma}}$ are also almost surely transient. 
\end{itemize} 
The transience~\eqref{eqn::transience_commontime} in the common-time parameterization then follows from this and~\eqref{eqn::timechange_ineq}. 
\end{proof}

\appendix
\section{Radial SLE with force points: transience}
\label{app:transience}

For the radial resampling property to make sense, one has to know that the curves are transient.
In this appendix, we establish the transience for  radial $\SLE_{\kappa}^{\mu}(\bs{\rho})$ processes with non-negative weights ($\rho_j\ge 0$ for all $j$).
With one force point, the transience for radial $\SLE_{\kappa}(\rho)$ with $\rho=2$ was proven in~\cite[Theorem~1.3]{Lawler:Continuity_of_radial_and_two-sided_radial_SLE_at_the_terminal_point} 
(see also~\cite{Field-Lawler:Escape_probability_and_transience_for_SLE}). 
We generalize Lawler's proof to derive the transience of radial SLE with multiple force points and spiral in Proposition~\ref{prop::radialSLE_transience},
whose proof we finish in Section~\ref{subsec::radialSLE_continuity}. 
The key ingredient to the proof is Lemma~\ref{lem::S2} --- indeed, with Lemma~\ref{lem::S2} at hand, 
the rest of the proof is exactly the same as it is in~\cite[proof of Theorem~1.3]{Lawler:Continuity_of_radial_and_two-sided_radial_SLE_at_the_terminal_point}. 
To derive this key ingredient, we will utilize an estimate on radial Bessel processes in Lemma~\ref{lem::Bessel_transience} and a monotone coupling argument in Section~\ref{subsec::monotone_coupling}. 

\begin{proposition}\label{prop::radialSLE_transience}
Fix $\kappa\in (0,4], \, \mu\in\R, \, p\ge 2, \, \bs{\rho}=(\rho_2, \ldots, \rho_p)\in\R^{p-1}$, and $\bs{\theta}=(\theta^1, \ldots, \theta^p)\in\LX_p$. Assume $\rho_j\ge 0$ for all $2\le j\le p$. 
Then, radial $\SLE_{\kappa}^{\mu}(\bs{\rho})$ in $(\U; \ee^{\ii\bs{\theta}}; 0)$ is almost surely generated by a continuous curve $\gamma$ and $\smash{\underset{t\to\infty}{\lim} \, \gamma_t=0}$. 
\end{proposition}

\begin{proposition}\label{prop::radialSLE_avoidboundary}
Assume the same setup as in Proposition~\ref{prop::radialSLE_transience}. Radial $\SLE_{\kappa}^{\mu}(\bs{\rho})$ in $(\U; \ee^{\ii\bs{\theta}}; 0)$ hits the boundary $\partial\U$ only at $\ee^{\ii\theta^1}$.  
\end{proposition}

The analogous result to Proposition~\ref{prop::radialSLE_avoidboundary} is well understood in the chordal setting (e.g.~\cite{Dubedat:Duality_of_SLE}). 
Because we have not been able to identify a reference for the radial setting, we provide a proof in Section~\ref{subsec::radialSLE_avoidboundary},
both for readers' convenience and for later reference.

\subsection{Radial Bessel processes}

Fix $\alpha>0, \, \mu\in\R$, and $0<\kappa\le 4\alpha$.
Let $X =(X_t)_{t \geq 0}$ be the (unique) solution to the SDE 
\begin{align}\label{eqn::Bessel_SDE}
\ud X_t=-\sqrt{\kappa} \, \ud B_t + \big( \alpha \cot(X_t/2)-\mu\big)  \, \ud t, \qquad X_0 \in (0,2\pi),
\end{align}
where $B_t$ is standard one-dimensional Brownian motion. 

\begin{lemma} \label{lem::Bessel_avoid_0and2pi}
Fix $\alpha>0, \, \mu\in\R$, and $0<\kappa\le 4\alpha$.  The solution $(X_t)_{t \geq 0}$ of~\eqref{eqn::Bessel_SDE} never hits $0$ nor $2\pi$. 
\end{lemma}
\begin{proof}
Define $T_{\eps}:=\inf\{t \geq 0 \colon X_t\le \eps \textnormal{ or } X_t\ge 2\pi-\eps \}$ and consider the density
\begin{align*}
f(\theta):=\frac{2}{\kappa} \int_{\theta}^{X_0} \ud y \int_{y}^{X_0} \ud u \, 
\Big(\frac{\sin(u/2)}{\sin(y/2)}\Big)^{4\alpha/\kappa}
\exp \big( \tfrac{2\mu}{\kappa} (y-u)\big) \; \ge 0 ,\qquad \theta\in (0,2\pi) ,
\end{align*}
which satisfies the differential equation
\begin{align*}
\frac{\kappa}{2} f''(\theta)+ f'(\theta) \big( \alpha \cot(\theta/2)-\mu \big) = 1 ,\qquad \theta\in (0,2\pi) .
\end{align*}
It\^o's formula shows that, for any $s \geq 0$, the process $\big(f(X_{t\wedge s\wedge T_\eps})-t\wedge s\wedge T_\eps \big)_{t\ge 0}$ is a bounded martingale.
Optional stopping theorem thus gives
\begin{align} \label{eqn::Bessel_avoid_0and2pi_aux1}
\E[f(X_{s\wedge T_\eps})]-\E[s\wedge T_\eps]=f(X_0)-0=0.
\end{align}
Using the estimate
\begin{align*}
\E[f(X_{s\wedge T_\eps})]\ge \PP[T_\eps\le s] \, \big( f(\eps) \wedge f(2\pi-\eps) \big),
\end{align*}
we obtain
\begin{align} \label{eqn::Bessel_avoid_0and2pi_aux2}
\PP[T_\eps\le s] \le \frac{\E[s\wedge T_\eps]}{f(\eps) \wedge f(2\pi-\eps)} \le \frac{s}{f(\eps) \wedge f(2\pi-\eps)}.
\end{align}
Now, since $\kappa\le 4\alpha$, we have $f(\eps), \, f(2\pi-\eps)\to \infty$ as $\eps\to 0$. 
Combining this with~\eqref{eqn::Bessel_avoid_0and2pi_aux2} and letting $\eps\to 0$, we obtain $\PP[T_0\le s]=0$. 
Letting $s\to\infty$, we the obtain $\PP[T_0< \infty]=0$, yielding the asserted result.
\end{proof}

The proof of the transience in Proposition~\ref{prop::radialSLE_transience} will need the following estimate.
\begin{lemma}\label{lem::Bessel_transience}
Fix $\alpha>0, \, \mu\in\R$, and $0<\kappa< 4\alpha$. 
For any $t_0>0$, there exists a constant $C_{\eqref{eqn::Bessel_transience}}\in (0,\infty)$ depending on $\alpha$, $\kappa$, $\mu$, and $t_0$ such that 
\begin{align}\label{eqn::Bessel_transience}
\PP\Big[\min_{0\le t\le t_0} X_t\le \eps X_0\Big] \le C_{\eqref{eqn::Bessel_transience}} \, \eps^{4\alpha/\kappa-1} , \qquad \eps > 0.
\end{align}
\end{lemma}

\begin{proof} 
Define $T_{\eps}:=\inf\{t \geq 0 \colon X_t\le \eps X_0 \textnormal{ or }X_t\ge \pi \}$.
We first show that  
\begin{align}\label{eqn::Bessel_estimate1}
\PP[X_{T_{\eps}}=\eps X_0]\le C_{\eqref{eqn::Bessel_estimate1}} \, \eps^{4\alpha/\kappa-1},\qquad\textnormal{where }C_{\eqref{eqn::Bessel_estimate1}}=\exp \big( \tfrac{4\pi|\mu| }{\kappa} \big) \big(\tfrac{\pi}{2}\big)^{4\alpha/\kappa}. 
\end{align}
It suffices to show~\eqref{eqn::Bessel_estimate1} for $X_0\in (0,\pi)$. 
Let us consider the density
\begin{align*}
f(\theta):=\int_{\theta}^{\pi} \exp \big( \tfrac{2\mu u}{\kappa} \big) \, \big(\sin \big(\tfrac{u}{2} \big) \big)^{-4\alpha/\kappa} \ud u,\qquad \theta\in (0,\pi) ,
\end{align*}
which satisfies the differential equation
\begin{align*}
f''(\theta)+f'(\theta) \bigg(\frac{4\alpha}{2\kappa} \cot (\theta/2) 
- \frac{2\mu}{\kappa}\bigg)=0,\qquad \theta\in (0,\pi) ,
\end{align*}
so It\^o's formula shows that $(f(X_{t\wedge T_{\eps}}))_{t\ge 0}$ is a bounded martingale.
Optional stopping theorem thus gives
\begin{align}\label{eqn::Bessel_aux1}
\PP[X_{T_{\eps}}=\eps X_0] \, f(\eps X_0)=f(X_0). 
\end{align}
Now, because $u/\pi \le \sin (u/2) \le u/2$ for $u\in (0,\pi)$, we see that
\begin{gather*}
\; \frac{2^{4\alpha/\kappa} \, (\pi^{-(4\alpha/\kappa -1)} - \theta^{-(4\alpha/\kappa-1)})}{1 - \tfrac{4\alpha}{\kappa}}
\, \exp \big( \! -\tfrac{2|\mu|\pi}{\kappa}\big)
\; = \; \exp \big( \! -\tfrac{2|\mu|\pi}{\kappa}\big) \int_{\theta}^{\pi} (u/2)^{-4\alpha/\kappa}\ud u 
\\
\; \le \; f(\theta) \le \; \exp \big( \tfrac{2|\mu|\pi}{\kappa}\big) \int_{\theta}^{\pi} (u/\pi)^{-4\alpha/\kappa}\ud u 
\; = \; \frac{\pi^{4\alpha/\kappa} \, (\pi^{-(4\alpha/\kappa -1)} - \theta^{-(4\alpha/\kappa-1)})}{1 - \tfrac{4\alpha}{\kappa}}
\, \exp \big(\tfrac{2|\mu|\pi}{\kappa}\big) .
\end{gather*}
Plugging these estimates into~\eqref{eqn::Bessel_aux1}, we obtain~\eqref{eqn::Bessel_estimate1}.

Next, we show~\eqref{eqn::Bessel_transience} when $X_0\in (0,\pi/2]$. Define
$\tilde{T}_\eps:=\inf \{ t>T_\eps \colon X_t=X_{0} \}$.
We have
\begin{align}
\;&\PP\Big[\min_{0\le t \le t_0} X_t \le \eps X_0\Big] \notag\\
\le \; &\PP \big[ X_{T_\eps}=\eps X_0\big] + \PP \Big[ T_{\eps}<\tilde{T}_\eps<t_0, \; X_{T_\eps}=\pi, \; \min_{\tilde{T}_\eps \le t \le t_0} X_t \le \eps X_0 \Big] \notag\\
= \; &\PP \big[ X_{T_\eps}=\eps X_0\big] + \PP\big[T_{\eps}<\tilde{T}_\eps<t_0, \; X_{T_\eps}=\pi\big] \; \PP \Big[ \min_{\tilde{T}_\eps \le t \le t_0} X_t \le \eps X_0 \Bigcond T_{\eps}<\tilde{T}_\eps<t_0, X_{T_\eps}=\pi\Big] \notag\\
\le \; &\PP \big[ X_{T_\eps}=\eps X_0\big] + \PP\big[T_{\eps}<\tilde{T}_\eps<t_0, \; X_{T_\eps}=\pi\big] \; \PP \Big[ \min_{0 \le t \le t_0} X_t \le \eps X_0 \Big] .\label{eqn::Bessel_aux2}
\end{align}
There is positive probability that $X$ starting at $\pi$ stays in $(\pi/2,3\pi/2)$ up to time $t_0$. As $X_0\in (0,\pi/2]$, there exists a constant $C_{\eqref{eqn::Bessel_aux3}}\in (0,1)$ depending on $\alpha$, $\kappa$, $\mu$, and $t_0$ such that
\begin{align}
\label{eqn::Bessel_aux3}
\PP \big[\tilde{T}_\eps\ge t_0 \cond X_{T_\eps}=\pi\big] \ge C_{\eqref{eqn::Bessel_aux3}}.
\end{align}
Plugging~\eqref{eqn::Bessel_aux3} into~\eqref{eqn::Bessel_aux2}, we obtain
\begin{align*}
\PP\Big[\min_{0\le t \le t_0} X_t \le \eps X_0\Big]\le \PP \big[ X_{T_\eps}=\eps X_0\big] + (1-C_{\eqref{eqn::Bessel_aux3}}) \, \PP \Big[ \min_{0 \le t \le t_0} X_t \le \eps X_0 \Big]. 
\end{align*}
Combining with~\eqref{eqn::Bessel_estimate1}, we obtain
\begin{align*}
\PP\Big[\min_{0\le t \le t_0} X_t \le \eps X_0\Big] \le \frac{C_{\eqref{eqn::Bessel_estimate1}}}{C_{\eqref{eqn::Bessel_aux3}}}\, \eps^{4\alpha/\kappa-1},
\end{align*}
which yields the asserted estimate~\eqref{eqn::Bessel_transience} when $X_0\in (0,\pi/2]$.

Finally, we show~\eqref{eqn::Bessel_transience} when $X_0\in (\pi/2,2\pi)$. 
It suffices to prove the result for $\eps<1/4$. Define $T':=\inf\{ t\geq 0 \colon X_t=\pi/2 \}$.
Since~\eqref{eqn::Bessel_transience} holds with $C_{\eqref{eqn::Bessel_transience}}=\frac{C_{\eqref{eqn::Bessel_estimate1}}}{C_{\eqref{eqn::Bessel_aux3}}}$ when $X_0=\pi/2$, we have
\begin{align}\label{eqn::Bessel_aux4}
\PP\Big[ \min_{T'<t<t_0} X_t \le \eps \pi/2 \Big] \le \PP\Big[ \min_{T'\le t<T'+t_0} X_t \le \eps \pi/2 \Big] \le \frac{C_{\eqref{eqn::Bessel_estimate1}}}{C_{\eqref{eqn::Bessel_aux3}}} \, \eps^{4\alpha/\kappa-1}.
\end{align}
Therefore, we obtain
\begin{align*}
\PP\Big[\min_{0\le t\le t_0}X_t\le \eps X_0\Big] \le \; & \PP\Big[\min_{0\le t\le t_0}X_t\le 2\eps\pi \Big] 
\, = \, \PP\Big[\min_{T'\le t\le t_0}X_t\le 2\eps\pi \Big] 
\, \le \, \frac{C_{\eqref{eqn::Bessel_estimate1}}}{C_{\eqref{eqn::Bessel_aux3}}} \, (4\eps)^{4\alpha/\kappa-1} ,
&& \textnormal{[by~\eqref{eqn::Bessel_aux4}]}
\end{align*}
which yields the asserted estimate~\eqref{eqn::Bessel_transience} when $X_0\in (\pi/2,2\pi)$. 
This concludes the proof.
\end{proof}

\subsection{A monotone coupling}
\label{subsec::monotone_coupling}

Let us now come back to the setup in Proposition~\ref{prop::radialSLE_transience}. 
Recall that the driving process for radial $\SLE_{\kappa}^{\mu}(\bs\rho)$ with $\bs\rho = (\rho_2, \ldots, \rho_p)$ satisfies the SDE~\eqref{eqn::SLEkapparhomu_SDE}:
\begin{align*}
\begin{cases}
\ud \xi_t = \sqrt{\kappa} \, \ud B_t 
- \displaystyle \sum_{j=2}^p \frac{\rho_j}{2} \cot \bigg(\frac{V_t^{j} - \xi_t}{2}\bigg) \ud t + \mu \, \ud t , 
\qquad \xi_0 = \theta^1; \\
\ud V_t^{j} =  \displaystyle  \cot \bigg(\frac{V_t^{j} - \xi_t}{2}\bigg) \ud t, \quad V_0^{j}=\theta^{j}, 
\qquad j\in \{2, \ldots, p\} ,
\end{cases}
\end{align*}
where $B_t$ is standard one-dimensional Brownian motion. 
Consider the difference $\Diff_{t} := V_t^2-\xi_t$ of the driving process $\xi_t$ and the first force point $V_t^2$, 
which solves the SDE
\begin{align*}
\ud\Diff_{t} = - \sqrt{\kappa} \, \ud B_t + \Big(1 + \frac{\rho_2}{2} \Big) \cot (\Diff_{t}/2) \, \ud t
+ \sum_{j=3}^p \frac{\rho_j}{2} \cot (\Diff_{t}^{j}/2) \, \ud t - \mu \, \ud t 
\end{align*}
writing $\Diff_{t}^{j} := V_t^j-\xi_t$, for $j\in \{3, \ldots, p\}$.
We will show in Lemma~\ref{lem::Delta2} that it 
satisfies analogous conclusions as the radial Bessel process in Lemma~\ref{lem::Bessel_transience}. 
Consequently, we derive an estimate for the process $\sin (\tfrac{\Diff_{t}}{2})$ in Lemma~\ref{lem::S2}, which will also be a key ingredient in the proof of Proposition~\ref{prop::radialSLE_transience}.

\begin{lemma}\label{lem::Delta2}
Assume that $\theta^p - \theta^2 \le\pi/4$ and $\rho_j\ge 0$ for all $2\le j\le p$.
For any $t_0>0$, there exists a constant $C_{\eqref{eqn::Delta2_estimate2}}\in (0,\infty)$ depending on $\kappa$, $\bs\rho$, $\mu$, and $t_0$ such that 
\begin{align}\label{eqn::Delta2_estimate2}
\PP\Big[\min_{0\le t\le t_0}\Diff_{t}\le \eps \, \Diff_{0}\Big]\le C_{\eqref{eqn::Delta2_estimate2}} \,\eps^{\beta_2},\qquad\textnormal{where }\beta_2 = \tfrac{2(2+\rho_2)}{\kappa} - 1 > 0.
\end{align}
\end{lemma}

To derive~\eqref{eqn::Delta2_estimate2}, we will compare $\Diff^{p}$ and $\Diff = \Diff^{2}$ to the radial Bessel process $X^\alpha$ solving~\eqref{eqn::Bessel_SDE}. 
It turns out that the former can be coupled with $X^\alpha$ so that it stays above $X^\alpha$ for all time;
and the latter can be coupled with $X^\alpha$ so that it stays below $X^\alpha$ for all time.

\begin{lemma}\label{lem::coupleXDeltaP}
Assume $\rho_j\ge 0$ for all $2\le j\le p$. 
Write $\bs{\Diff}_t = (\Diff_{t}^{2},\ldots,\Diff_{t}^{p}) = (V_{t}^{2}-\xi_t,\ldots,V_{t}^{p}-\xi_t)$. 
Consider $X^\alpha = (X^\alpha_t)_{t \geq 0}$ solving~\eqref{eqn::Bessel_SDE}, 
with $\alpha = 1 + \tfrac{1}{2} \big(\rho_2+\cdots+\rho_p \big)$. 
\begin{itemize}
\item There exists a coupling between $X^\alpha$ and $(\bs{\Diff}_t)_{t \geq 0}$ such that $X^\alpha_0=\Diff_{0}^{p}$ and $X^\alpha_t\le \Diff_{t}^{p}$ for all $t\geq 0$. 
\item There exists a coupling between $X^\alpha$ and $(\bs{\Diff}_t)_{t \geq 0}$ such that $X^\alpha_0=\Diff_{0}^{2}$ and $X^\alpha_t\ge \Diff_{t}^{2}$ for all $t\geq 0$.
\end{itemize}
\end{lemma}

\begin{proof}
Given $X^\alpha$, let $\bs Y_t = (Y_t^{2}, \ldots, Y_t^{p})$ be the solution to the following SDE system: 
for $2\le \ell\le p$, 
\begin{align}\label{eqn::Delta_Y_coupling}
\begin{split}
\ud Y_t^{\ell} = \; & \alpha \cot\Big(\tfrac{X^\alpha_t}{2}\Big) \ud t - \cot\Big(\tfrac{X^\alpha_t-Y_t^{\ell}}{2}\Big) \ud t - \sum_{j=2}^p\frac{\rho_j}{2} \cot\Big(\tfrac{X^\alpha_t-Y_t^{j}}{2}\Big) \ud t , 
\\
Y_0^{\ell} = \; & \Diff_{0}^{p}-\Diff_{0}^{\ell} .
\end{split}
\end{align}
As the right-hand side of this SDE is locally Lipschitz, there is a unique strong solution to $(\bs Y_t)_{t \geq 0}$. 
Note that $\bs{\Diff}_t$ then satisfies $\bs{\Diff}_t^{\ell} =X^\alpha_t-Y_t^{\ell}$, for $2\le \ell\le p$.
This gives a coupling between $X^\alpha$ and $\bs\Diff$, and it remains to check that 
$\Diff_{t}^{p}=X^\alpha_t-Y_t^{p}\ge X^\alpha_t$ for all $t \geq 0$. To this end, since $Y_t^{p}$ solves~\eqref{eqn::Delta_Y_coupling} with $\ell = p$,
\begin{align}\label{eqn::Delta_Yp}
\ud Y_t^{p} = \; & \bigg( \cot\Big(\tfrac{X^\alpha_t}{2}\Big) - \cot\Big(\tfrac{X^\alpha_t-Y_t^{p}}{2}\Big) \bigg) \ud t + \sum_{j=2}^{p} \frac{\rho_j}{2} \bigg( \cot\Big(\tfrac{X^\alpha_t}{2}\Big) - \cot\Big(\tfrac{X^\alpha_t-Y_t^{j}}{2}\Big) \bigg) \ud t , 
\end{align}
and we have $0 = Y_0^{p} < \cdots< Y_0^{2}$, while 
the cotangent is decreasing on $(0,\pi)$ and we assumed that $\rho_j \geq 0$ for all $j$, we see that 
the right-hand side of~\eqref{eqn::Delta_Yp} is negative, whence $Y_t^{p}\le 0$ for all $t \geq 0$.
\end{proof}

\begin{proof}[Proof of Lemma~\ref{lem::Delta2}]
For $\Diff_{t} = \Diff_{t}^{2}$, define 
$T_\eps:=\inf \{ t\ge 0 \colon \Diff_{t}\le \eps \, \Diff_{0}\textnormal{ or } \Diff_{t}\ge \pi \}$.
We first show that
\begin{align}\label{eqn::Delta2_estimate1}
\begin{split}
\PP[\Diff_{T_{\eps}}^{2} =\eps \, \Diff_{0}]\le  \; & C_{\eqref{eqn::Delta2_estimate1}} \, \eps^{\beta_2}, 
\\
\textnormal{where} \qquad
C_{\eqref{eqn::Delta2_estimate1}} = \; & \exp \big( \tfrac{4\pi|\hat{\mu}|}{\kappa} \big) \big(\tfrac{\pi}{2}\big)^{\beta_2+1} ,
\qquad \textnormal{and} \qquad
\hat{\mu} := \mu-\sum_{j=3}^p\tfrac{\rho_j}{2}\cot(5\pi/8) .
\end{split}
\end{align}
It suffices to prove~\eqref{eqn::Delta2_estimate1} for $\Diff_{0}\in (0,\pi)$. Let us consider the density
\begin{align*}
f(\theta):=\int_{\theta}^{\pi} \frac{\exp \big( \tfrac{2\hat\mu u}{\kappa} \big) \, \ud u}{\big(\sin \big(\tfrac{u}{2} \big) \big)^{1+\beta_2}}  \; \le 0 ,\qquad \theta\in (0,\pi) ,
\end{align*}
which satisfies the differential equation
\begin{align*}
f''(\theta)+f'(\theta)\bigg(\frac{\beta_2+1}{2}\cot(\theta/2)-\frac{2\hat{\mu}}{\kappa}\bigg)=0,\qquad \theta\in (0,\pi). 
\end{align*}
Using It\^{o}'s formula, we compute
\begin{align*}
\ud f( \Diff_{t}) 
=\;& \sqrt{\kappa} f'( \Diff_{t} ) \ud B_t 
\; + \; \Big( \underbrace{\frac{\kappa}{2} f''( \Diff_{t}) + f'( \Diff_{t} )   \Big( \cot(\Diff_{t}/2) + \sum_{j=2}^{p} \frac{\rho_j}{2} \cot(\Diff_{t}^{j}/2) -\mu \Big) }_{=: \, F_{t}}\Big) \ud t.
\end{align*}
When $0 \le t\le T_\eps$, since $f'\le 0$ and $\rho_j\ge 0$ and 
\begin{align*}
\Diff_{t}^{j} \, \le \, \Diff_{t}+\Diff_{0}^{p}-\Diff_{0} \, \le \,  \pi+\Diff_{0}^{p}-\Diff_{0} \, \le \, 5\pi/4,
\qquad 3\le j\le p ,
\end{align*}
we see that
\begin{align*}
F_t\le  \frac{\kappa}{2} f''( \Diff_{t}) + f'( \Diff_{t})  \bigg( \frac{(\beta_2+1) \kappa}{4} \cot(\Diff_{t}/2)  -\hat{\mu}\bigg) =0.
\end{align*}
This shows that $(f( \Diff_{t\wedge T_\eps}))_{t\ge 0}$ is a bounded supermartingale, and optional stopping theorem yields
\begin{align*}
\PP [ \Diff_{T_\eps}^{2}=\eps \, \Diff_{0} ] f(\eps \, \Diff_{0})  \le f(\Diff_{0}).
\end{align*}
which gives the desired~\eqref{eqn::Delta2_estimate1}. 

Next, we show~\eqref{eqn::Delta2_estimate2} for $\Diff_{0}\in (0,\pi/4]$. 
Define 
\begin{align*}
q(\eps):=\sup\Big\{\PP \Big[ \min_{0 \le t \le t_0} \Diff_{t} \le \eps \, \Diff_{0} \Big] \colon \Diff_{0}\in(0,\pi/4] , \;  \Diff_{0} = \Diff_{0}^{2}\le\Diff_{0}^{3}\le\cdots\le\Diff_{0}^{p}\le \Diff_{0}+\pi/4 \Big\}.
\end{align*}
It suffices to show that
\begin{align}\label{eqn::Delta2_estimate2_first}
q(\eps)\le \frac{C_{\eqref{eqn::Delta2_estimate1}}}{C_{\eqref{eqn::Bessel_aux3}}} \, \eps^{\beta_2}.
\end{align}
Define $\tilde{T}_\eps:=\inf \{ t>T_\eps \colon \Diff_{t}=\Diff_{0} \}$ and $\hat{T}^p_\eps:=\inf \{ t>T_\eps \colon \Diff_{t}^{p}=\pi \}$. 
For $\Diff_{0}\in(0,\pi/4]$, we have
\begin{align}
&\PP\Big[\min_{0\le t \le t_0} \Diff_{t} \le \eps \, \Diff_{0}\Big] \notag\\
\le \;&\PP[ \Diff_{T_\eps} =\eps \, \Diff_{0}] + \PP \Big[ T_{\eps}<\tilde{T}_\eps<t_0, \; \Diff_{T_\eps}=\pi, \; \min_{\tilde{T}_\eps \le t \le t_0} \Diff_{t} \le \eps \, \Diff_{0} \Big] \notag\\
= \;&\PP [ \Diff_{T_\eps}=\eps \, \Diff_{0}] + \PP[T_{\eps}<\tilde{T}_\eps<t_0, \; \Diff_{T_\eps}=\pi]\;\PP \Big[ \min_{\tilde{T}_\eps \le t \le t_0} \Diff_{t} \le \eps \, \Diff_{0} \Bigcond T_{\eps}<\tilde{T}_\eps<t_0, \; \Diff_{T_\eps}=\pi\Big]\notag\\
\le \;&\PP [ \Diff_{T_\eps}=\eps \, \Diff_{0}] + \PP[T_{\eps}<\tilde{T}_\eps<t_0, \; \Diff_{T_\eps}=\pi]\; q(\eps).\label{eqn::Delta2_estimate_aux2}
\end{align}
On the one hand, when $\Diff_{T_\eps}=\pi$, since $\Diff_{t}^{p}-\Diff_{t}< \pi/4$ and $\Diff_{0}\le \pi/4$, we have $\Diff_{\hat{T}^p_\eps}^{p}=\pi$ and $\hat{T}^p_\eps<\tilde{T}_\eps$. 
Thus, we have
\begin{align}\label{eqn::Delta2_estimate_aux3}
\PP[T_{\eps}<\tilde{T}_\eps<t_0, \; \Diff_{T_\eps}=\pi] \le \PP [ \hat{T}^p_\eps<\tilde{T}_\eps<t_0, \; \Diff_{\hat{T}^p_\eps}^{p}=\pi ].
\end{align}
On the other hand, since $\Diff_{t}^{p}-\Diff_{t}< \pi/4$ and $\Diff_{0}\le \pi/4$, if $\Diff_{t}^{p}$ stays in $[\pi/2,3\pi/2]$ for $\hat{T}^p_\eps\le t\le t_0$, 
then we have $\tilde{T}_\eps\ge t_0$. 
Combining  this with Lemma~\ref{lem::coupleXDeltaP} and Equation~\eqref{eqn::Bessel_aux3}, we find that
\begin{align}\label{eqn::Delta2_estimate_aux4}
\PP [ \hat{T}^p_\eps<\tilde{T}_\eps<t_0, \; \Diff_{\hat{T}^p_\eps}^{p}=\pi ]\le 1-C_{\eqref{eqn::Bessel_aux3}}.
\end{align}
Combining~(\ref{eqn::Delta2_estimate1},~\ref{eqn::Delta2_estimate_aux2},~\ref{eqn::Delta2_estimate_aux3},~\ref{eqn::Delta2_estimate_aux4}), we obtain
\begin{align*}
q(\eps) \le C_{\eqref{eqn::Delta2_estimate1}} \, \eps^{\beta_2} + (1-C_{\eqref{eqn::Bessel_aux3}}) q(\eps),
\end{align*}
which implies~\eqref{eqn::Delta2_estimate2_first}.

Finally, we show~\eqref{eqn::Delta2_estimate2} for $\Diff_{0}\in (\pi/4,2\pi)$. It suffices to show the result for $\eps<1/8$. 
Defining $T'=\inf\{ t \ge 0 \colon \Diff_{t}=\pi/4 \}$ and using~\eqref{eqn::Delta2_estimate2_first} we obtain
\begin{align*}
\PP \Big[ \min_{T' \le t \le t_0} \Diff_{t} \le \frac{\eps \pi}{4} \Big]\le \PP \Big[ \min_{T' \le t \le T'+t_0} \Diff_{t} \le \frac{\eps \pi}{4} \Big]\le \frac{C_{\eqref{eqn::Delta2_estimate1}}}{C_{\eqref{eqn::Bessel_aux3}}} \, \eps^{\beta_2}.
\end{align*}
Therefore, we have 
\begin{align*}
\PP \Big[ \min_{0 \le t \le t_0} \Diff_{t} \le \eps \, \Diff_{0} \Big] \le& \PP \Big[ \min_{0 \le t \le t_0} \Diff_{t} \le 2\eps\pi \Big] 
\le \PP \Big[ \min_{T' \le t \le t_0} \Diff_{t} \le 2\eps\pi \Big] 
\le \frac{C_{\eqref{eqn::Delta2_estimate1}}}{C_{\eqref{eqn::Bessel_aux3}}} \, (8\eps)^{\beta_2},
\end{align*}
as desired in~\eqref{eqn::Delta2_estimate2} with $C_{\eqref{eqn::Delta2_estimate2}}=8^{\beta_2}\frac{C_{\eqref{eqn::Delta2_estimate1}}}{C_{\eqref{eqn::Bessel_aux3}}}$.
\end{proof}

\begin{lemma}\label{lem::S2}
Assume that $\theta^p - \theta^2 \le\pi/4$ and $\rho_j\ge 0$ for all $2\le j\le p$.
For any $t_0>0$, there exists a constant $C_{\eqref{eqn::S2_estimate}}\in (0,\infty)$ depending on $\kappa$, $\bs\rho$, $\mu$, and $t_0$ such that 
\begin{align}\label{eqn::S2_estimate}
\PP \Big[ \min_{0\le t \le t_0} \sin \big(\tfrac{\Diff_{t}}{2} \big) \le \eps \, \sin \big(\tfrac{\Diff_{0}}{2} \big) \Big] \le C_{\eqref{eqn::S2_estimate}} \, \eps^{\beta_2},\qquad\textnormal{where }\beta_2= \frac{2(2+\rho_2)}{\kappa} - 1 > 0 .
\end{align}
\end{lemma}

\begin{proof} 
We first show that
\begin{align}\label{eqn::Delta2_estimate3}
\PP\Big[\max_{0\le t\le t_0}\Diff_{t}\ge 2\pi-\eps(2\pi- \Diff_{0})\Big]\le C_{\eqref{eqn::Bessel_transience}}\eps^{\beta},\qquad\textnormal{where }\beta= \frac{2(2+\rho_2+\cdots+\rho_p)}{\kappa} - 1 > 0 , 
\end{align}
and $C_{\eqref{eqn::Bessel_transience}}$ is defined in~\eqref{eqn::Bessel_transience}. 
We use the coupling from Lemma~\ref{lem::coupleXDeltaP} between  $X^\alpha = (X^\alpha_t)_{t \geq 0}$ solving~\eqref{eqn::Bessel_SDE} and $\bs\Diff$ such that $X_0^{\alpha}=\Diff_{0}$ and $X_t^{\alpha}\ge \Diff_{t}$ for all $t$, to obtain~\eqref{eqn::Delta2_estimate3}:
\begin{align}
\PP\Big[\max_{0\le t\le t_0}\Diff_{t}\ge 2\pi-\eps(2\pi- \Diff_{0})\Big]\le \PP\Big[\max_{0\le t\le t_0} X_t^{\alpha}\ge 2\pi-\eps(2\pi- X_0^{\alpha})\Big]\le C_{\eqref{eqn::Bessel_transience}} \, \eps^{\beta}. 
&& \textnormal{[by~\eqref{eqn::Bessel_transience}]}
\end{align}
Combining~\eqref{eqn::Delta2_estimate2} and~\eqref{eqn::Delta2_estimate3}, we then obtain~\eqref{eqn::S2_estimate}.
\end{proof}

Lemma~\ref{lem::S2} is the key ingredient in the proof of Proposition~\ref{prop::radialSLE_transience}. 
Let us point out that the assumption $\Diff_{0}^{p}-\Diff_{0} = \theta^p - \theta^2 \le \pi/4$ in Lemma~\ref{lem::S2} cannot be easily dropped. 
Whereas, it turns out that the difference $\Diff_{t}^{p}-\Diff_{t}$ decays exponentially fast
(cf.~Lemma~\ref{lem::expdecay}) which makes it possible to apply Lemma~\ref{lem::S2} in the proof of Proposition~\ref{prop::radialSLE_transience}. 

\begin{lemma}\label{lem::expdecay}
The difference $\Diff_{t}^{p}-\Diff_{t}$ decays exponentially fast to $0$ as $t\to\infty$\textnormal{:} 
\begin{align}\label{eqn::expdecay}
0 \, \le \; & \Diff_{t}^{p}-\Diff_{t} \, \le \,  (\Diff_{0}^{p}-\Diff_{0}) \,  \exp(-C_{\eqref{eqn::expdecay}}t), 
\\
\nonumber
\textnormal{where} \qquad 
C_{\eqref{eqn::expdecay}} = \; &
\inf\Big\{\frac{\sin(u/2)}{u} \colon 0\le u\le \Diff_{0}^{p}-\Diff_{0}\Big\}>0.
\end{align}
\end{lemma}
\begin{proof}
We have
\begin{align*}
\frac{\ud }{\ud t}( \Diff_{t}^{p}-\Diff_{t} ) = \cot \big( \tfrac{\Diff_{t}^{p}}{2} \big) - \cot \big( \tfrac{\Diff_{t}}{2} \big)  \le -\sin \big( \tfrac{\Diff_{t}^{p}-\Diff_{t}}{2} \big) , \qquad t \geq 0 . 
\end{align*}
Because $0\le \Diff_{t}^{p}-\Diff_{t}\le \Diff_{0}^{p}-\Diff_{0}$, we see that
\begin{align*}
\frac{\ud }{\ud t}( \Diff_{t}^{p}-\Diff_{t}) 
\le  -\sin \big( \tfrac{\Diff_{t}^{p}-\Diff_{t}}{2} \big) 
\le -C_{\eqref{eqn::expdecay}}\, ( \Diff_{t}^{p}-\Diff_{t} ). 
\end{align*}
This implies the sought~\eqref{eqn::expdecay}. 
\end{proof}

\subsection{Transience --- proof of Proposition~\ref{prop::radialSLE_transience} and consequences}
\label{subsec::radialSLE_continuity}

\begin{proof}[Proof of Proposition~\ref{prop::radialSLE_transience}]
We first gather some notation and terminology.
\begin{itemize}
\item We denote $\U_n:=\ee^{-n} \U$, and $\tau_n:=\inf \{ t\ge 0 \colon |\gamma_t|=\ee^{-n} \}$, and $\LF_n:=\sigma ( \{ \gamma_t \colon 0\le t\le \tau_n \} )$, for $n \in \N$.  

\item A \emph{crosscut} of a domain $\Omega$ is the image of a simple curve $\eta \colon (0,1)\to \Omega$ with $\eta_{0+},\eta_{1-} \in \partial \Omega$.

\item We denote by $H_n = H_n(\gamma)$ the connected component of $\U \setminus \gamma_{(0,\tau_n]}$ containing the origin.

\item We denote $\partial_n^0:=\partial H_n \setminus \gamma_{(0,\tau_n]}$, which is either empty or an open subarc of $\partial \U$.

\item For $0<k<n$, we denote by $V_{n,k}$ the connected component of $H_n \cap \U_k$ that contains the origin.

\item We denote $\partial_{n,k}:=\partial V_{n,k} \cap H_n$. 
Note that the connected components of $\partial_{n,k}$ comprise a collection $\LA_{n,k}$ of open subarcs of $\partial \U_k$, 
which are crosscuts of appropriate domains. 

\item Indeed, each arc $\eta \in \LA_{n,k}$ is a crosscut of $H_n$ such that $H_n \setminus \eta$ has two connected components. 

We denote by $V_{n,k}(\eta)$ the connected component of $H_n\setminus \eta$ that does not contain the origin. 
Note that $V_{n,k}(\eta)\cap V_{n,k}(\eta')=\emptyset$ for different $\eta,\eta'\in \LA_{n,k}$.

\item If $\partial_n^0\neq \emptyset$, then there is a unique arc $\eta^*=\eta^*_{n,k} \in \LA_{n,k}$ such that $\partial_n^0 \subset \partial V_{n,k}(\eta^*)$.
\end{itemize}

In the proof, there are two main steps:
\begin{itemize}
\item [{\bf Step~1.}] 
(Cf.~\cite[Proposition~4.4~and~Theorem~1.2]{Lawler:Continuity_of_radial_and_two-sided_radial_SLE_at_the_terminal_point}). 
There exists a constant $C_{\eqref{eqn::SLEcontinuity_Step1.1}}\in (0,\infty)$ depending on $\kappa$, $\bs\rho$ and $\mu$ such that for all $k,n \in \N_{>0}$, 
if $\eta=\eta^*_{n+k,k}$, then we have
\begin{align}\label{eqn::SLEcontinuity_Step1.1}
\PP [ \gamma_{[\tau_{n+k},\infty)} \cap \eta \neq \emptyset \cond \LF_{n+k} ] \le C_{\eqref{eqn::SLEcontinuity_Step1.1}} \,  \ee^{-n\beta_2/2},\qquad \textnormal{where } \beta_2 = \frac{2(2+\rho_2)}{\kappa} - 1 > 0.
\end{align}
Moreover, for all positive integers $k\ge j\ge 1$ and $n\ge 1$, we have
\begin{align}\label{eqn::SLEcontinuity_Step1.2}
\PP [ \gamma_{[\tau_{n+k},\infty)} \subset \U_j \cond \LF_{n+k} ] \ge \big( 1- C_{\eqref{eqn::SLEcontinuity_Step1.1}} \, \ee^{-n\beta_2/2} \big) \one \{ \gamma_{[\tau_k,\tau_{n+k}]} \subset \U_j \}.
\end{align}

\item [{\bf Step~2.}]
(Cf.~\cite[Theorem~1.3]{Lawler:Continuity_of_radial_and_two-sided_radial_SLE_at_the_terminal_point}).
There exist constants $C_{\eqref{eqn::SLEcontinuity_Step2}}\in (0,\infty)$ and $u\in (0,\infty)$ depending on $\kappa$, $\bs\rho$ and $\mu$ such that for all $k,n \in \N_{>0}$, we have
\begin{align}\label{eqn::SLEcontinuity_Step2}
\PP [ \gamma_{[\tau_{n+k},\infty)} \cap \partial \U_k \neq \emptyset ] \le C_{\eqref{eqn::SLEcontinuity_Step2} \,} \ee^{-un}.
\end{align}
\end{itemize}
We will prove~(\ref{eqn::SLEcontinuity_Step1.1},~\ref{eqn::SLEcontinuity_Step1.2}) in Lemma~\ref{lem::SLEcontinuity_Step1}. 
Thereafter, using the same analysis as in the proof of~\cite[Theorem~1.3]{Lawler:Continuity_of_radial_and_two-sided_radial_SLE_at_the_terminal_point},  
we obtain~\eqref{eqn::SLEcontinuity_Step2} (see also Remark~\ref{rem:Greg_lemma}). 
Lastly, knowing~\eqref{eqn::SLEcontinuity_Step2}, Borel-Cantelli lemma yields $\smash{\underset{t\to\infty}{\lim}} \, \gamma_t=0$, which gives the desired transience of the curve.
\end{proof}

\begin{lemma}\label{lem::SLEcontinuity_Step1}
There exists a constant $C_{\eqref{eqn::SLEcontinuity_Step1.1}}\in (0,\infty)$ depending on 
$\kappa$, $\bs\rho$, and $\mu$ such that~\textnormal{(\ref{eqn::SLEcontinuity_Step1.1},~\ref{eqn::SLEcontinuity_Step1.2})} hold. 
\end{lemma}

\begin{proof}
First, we prove~\eqref{eqn::SLEcontinuity_Step1.1} (cf.~\cite[Proposition~4.4]{Lawler:Continuity_of_radial_and_two-sided_radial_SLE_at_the_terminal_point}). 
Define $\sigma^*_{n+k,k}:=\inf \{ t\ge \tau_{n+k} \colon \gamma_t\in \eta \}$. 
It~suffices to show there exists a constant $C_{\eqref{eqn::SLEcontinuity_Step1.1_suffice}}$ such that 
\begin{align}\label{eqn::SLEcontinuity_Step1.1_suffice}
\PP [ \sigma^*_{n+k,k}<\tau_{n+k+1} \cond \LF_{n+k} ] \le C_{\eqref{eqn::SLEcontinuity_Step1.1_suffice}} \, \ee^{-n\beta_2/2},
\end{align}
for then we can iterate and sum over $n$. 
We may conclude with the help of the following three facts. 
\begin{itemize}

\item It was proved in~\cite[Lemma~2.10]{Lawler:Continuity_of_radial_and_two-sided_radial_SLE_at_the_terminal_point} that there exists a constant $C_{\eqref{eqn::SLEcontinuity_aux1}}\in (0,\infty)$ such that
\begin{align}\label{eqn::SLEcontinuity_aux1}
\sin \big( \tfrac{1}{2}\Diff_{\sigma^*_{n+k,k}} \big) \le C_{\eqref{eqn::SLEcontinuity_aux1}} \, \ee^{-n/2} \, \sin \big( \tfrac{1}{2}\Diff_{\tau_{n+k}} \big) .
\end{align} 

\item  We have the following estimate on $\tau_n$:
\begin{align} \label{eqn::SLEcontinuity_aux2}
n-\log 4\le \tau_{n}\le n.
\end{align}
Indeed, on the one hand, applying Koebe's one-quarter theorem (e.g.~\cite[Theorem~2.3]{Duren:Univalent_functions})
with $g^{-1}_{\tau_n}$, we have $\frac{1}{4}\ee^{-\tau_n}\U\subset \U\setminus \gamma_{[0,\tau_n]}$, which implies the lower bound of~\eqref{eqn::SLEcontinuity_aux2}. On the other hand, since $\U_n\subset \U \setminus \gamma_{[0,\tau_n]}$, we have $g'_{\tau_n}(0)\le \ee^{n}$, which implies the upper bound of~\eqref{eqn::SLEcontinuity_aux2}.
\item When 
\begin{align*}
n \; > \;  n_0 := \frac{1}{C_{\eqref{eqn::expdecay}}} \, \log\Big(\tfrac{\Diff_{0}^{p}-\Diff_{0}}{\pi} \Big) + \Big( 1 + \frac{1}{C_{\eqref{eqn::expdecay}}} \Big) \log 4  ,
\end{align*}
combining~\eqref{eqn::SLEcontinuity_aux2} and Lemma~\ref{lem::expdecay}, we have $\Diff_{\tau_n}^{p}-\Diff_{\tau_n}^{2}<\pi/4$. 
Furthermore by Lemma~\ref{lem::S2}, we have
\begin{align} \label{eqn::SLEcontinuity_aux3}
\PP \bigg[ \min_{\tau_{n+k} \le t \le \tau_{n+k}+1+\log 4} \sin \big( \tfrac{\Diff_{t}}{2} \big) 
\le \eps \, \sin \big( \tfrac{\Diff_{\tau_{n+k}}}{2} \big) \cond \LF_{n+k} \bigg] \le C_{\eqref{eqn::S2_estimate}} \, \eps^{\beta_2}.
\end{align}
\end{itemize}
Gathering these estimates yields
\begin{align*}
&\PP [ \sigma^*_{n+k,k}<\tau_{n+k+1} \cond \LF_{n+k} ] \\
\le \; &\PP [ \sigma^*_{n+k,k}<\tau_{n+k}+1+\log 4 \cond \LF_{n+k} ] 
&& \textnormal{[by~\eqref{eqn::SLEcontinuity_aux2}]} \\
\le \; & \PP \bigg[ \min_{\tau_{n+k} \le t \le \tau_{n+k}+1+\log 4} \sin \big( \tfrac{\Diff_{t}}{2} \big)  
\le \sin \big( \tfrac{1}{2}\Diff_{\sigma^*_{n+k,k}} \big) 
\le C_{\eqref{eqn::SLEcontinuity_aux1}} \, \ee^{-n/2} \, \sin \big( \tfrac{1}{2}\Diff_{\tau_{n+k}} \big)  \cond \LF_{n+k} \bigg] 
&& \textnormal{[by~\eqref{eqn::SLEcontinuity_aux1}]} \\
\le \; & C_{\eqref{eqn::S2_estimate}} \, C_{\eqref{eqn::SLEcontinuity_aux1}}^{\beta_2} \, \ee^{-n\beta_2/2} ,
&& \textnormal{[by~\eqref{eqn::SLEcontinuity_aux3}]} \\
\end{align*}
which implies~\eqref{eqn::SLEcontinuity_Step1.1_suffice} and thus~\eqref{eqn::SLEcontinuity_Step1.1}. 
We now obtain~\eqref{eqn::SLEcontinuity_Step1.2} from~\eqref{eqn::SLEcontinuity_Step1.1} because 
by~\cite[Lemma~2.5]{Lawler:Continuity_of_radial_and_two-sided_radial_SLE_at_the_terminal_point}, if $\gamma_{[\tau_k,\tau_{n+k}]} \subset \U_j$, then in order for $\gamma_{[\tau_{n+k},\infty)}$ to intersect $\U_j$, it has to intersect $\eta$.
\end{proof}

\begin{remark}\label{rem:Greg_lemma}
A cautious reader might observe that we rely on~\cite[Lemma~2.10]{Lawler:Continuity_of_radial_and_two-sided_radial_SLE_at_the_terminal_point}, but not on~\cite[Lemma~2.9]{Lawler:Continuity_of_radial_and_two-sided_radial_SLE_at_the_terminal_point}\footnote{A corrected modification of~\cite[Lemma~2.9]{Lawler:Continuity_of_radial_and_two-sided_radial_SLE_at_the_terminal_point} appears in~\cite[Lemma~5.5]{Field-Lawler:Escape_probability_and_transience_for_SLE}; see also~\cite[Appendix~D]{Abuzaid-Peltola:Large_deviations_of_radial_SLE0}.}. 
The proof of~\cite[Lemma~2.10]{Lawler:Continuity_of_radial_and_two-sided_radial_SLE_at_the_terminal_point} strongly uses the fact that the crosscut $\eta$ separates the origin from $\partial \U$ in $\U\setminus \gamma_{[0,\tau_{n+k}]}$.
The obtained estimate is, in a sense, weaker than that derived 
in~\cite[Lemma~2.9]{Lawler:Continuity_of_radial_and_two-sided_radial_SLE_at_the_terminal_point}~\&~\cite[Lemma~5.5]{Field-Lawler:Escape_probability_and_transience_for_SLE}.
The proof of~\cite[Theorem~1.3]{Lawler:Continuity_of_radial_and_two-sided_radial_SLE_at_the_terminal_point}, which we follow in this appendix,  
combines this weaker estimate with a Markov chain argument, which loses the explicit $\kappa$-dependence of the exponent $u$ in~\eqref{eqn::SLEcontinuity_Step2} 
(which is not needed to prove the transience of the radial $\SLE_{\kappa}(\rho)$ curve).
\end{remark}

\begin{corollary}\label{cor::phitKtoid}
Assume the same setup as in Proposition~\ref{prop::radialSLE_transience}.
Let $\gamma$ be radial $\SLE_{\kappa}^{\mu}(\bs\rho)$ in $(\U; \ee^{\ii\bs{\theta}}; 0)$ and let $(g_t)_{t \geq 0}$ the the associated mapping-out function as in~\eqref{eq:single_radial_Loewner_equation}. 
Suppose $K$ is a compact subset of $\overline{\U}$ such that $\U\setminus K$ is simply connected and contains the origin, 
and $K$ has a positive distance from $\{\ee^{\ii\theta^1}, \ldots, \ee^{\ii\theta^p}\}$. 
Consider the conformal map $g_{t,K} \colon \U \setminus g_t(K)\to \U$ normalized at the origin, i.e., $g_{t,K}(0)=0,g'_{t,K}(0)>0$, and the associated covering map $\covmap_{t,K}$.  
Then, conditionally on $\{\gamma\cap K=\emptyset\}$, we have almost surely
\begin{align}\label{eqn::phitKtoid_1}
\lim_{t\to\infty}g_{t, K}'(0)=1,\qquad \lim_{t\to \infty} ( \covmap_{t,K}(\xi_t)-\xi_t )=0, \qquad \lim_{t\to\infty} (\covmap_{t, K}(V_t^j)-V_t^j)=0, \textnormal{ for }2\le j\le p. 
\end{align}
\end{corollary}

\begin{proof}
Schwarz lemma (e.g.~\cite[Page~3]{Duren:Univalent_functions}) gives the bounds $1 \leq g_{t, K}'(0) \leq 1/\dist(0,g_{t}(K))$,
as noted in \cite[Lemma~17]{Jahangoshahi-Lawler:Multiple-paths_SLE_in_multiply_connected_domains}. 
The first limit in~\eqref{eqn::phitKtoid_1} 
follows by observing that the harmonic measure of $g_{t}(K)$ in $\U \setminus g_{t}(K)$ seen from zero converges almost surely to zero as $t \to \infty$, because $\gamma(t) \to 0$ as $t \to \infty$ by Proposition~\ref{prop::radialSLE_transience}.
Moreover, $\{g_{t, K}^{-1}\}_{t \geq 0}$ is a normal family by Montel's theorem (e.g.~\cite[Page~7]{Duren:Univalent_functions}), so there is a.s.~a uniformly convergent subsequential limit 
$\smash{g_{t_n, K}^{-1} \overset{t_n \to \infty}{\to} g_{\infty, K}^{-1}}$, and the derivatives also converge uniformly (by Cauchy's integral formula). 
Combined with Schwarz lemma, the first limit in~\eqref{eqn::phitKtoid_1} now shows that $g_{\infty, K}^{-1} \colon \U \to \U$ must be the identity map. This yields the other two limits in~\eqref{eqn::phitKtoid_1}.
\end{proof}

\subsection{Proof of Proposition~\ref{prop::radialSLE_avoidboundary}}
\label{subsec::radialSLE_avoidboundary}
\begin{lemma} \label{lem::radialSLE_avoidboundary1}
Fix $\kappa\in (0,4]$, $\mu\in\R$ and $\theta\in\R$.
Radial $\SLE_{\kappa}^{\mu}$ in $(\U; \ee^{\ii\theta}; 0)$ intersects $\partial\U$ only at $\ee^{\ii\theta}$.  
\end{lemma} 

\begin{proof}
Let us first note that $\covmap_t(\theta')-\xi_t$ never hits $0$ or $2\pi$ at any finite time $t$ by Lemma~\ref{lem::Bessel_avoid_0and2pi}, since
\begin{align} \label{eqn::radialSLE_avoidboundary1}
\ud (\covmap_t(\theta')-\xi_t)=-\sqrt{\kappa} \ud B_t+ \cot \bigg( \frac{\covmap_t(\theta')-\xi_t}{2} \bigg) \ud t - \mu \, \ud t , \qquad \theta<\theta'<\theta+2\pi .
\end{align}
Therefore, for any $0<\eps<2\pi$ and any finite time $t$, 
the points $\ee^{\ii(\theta^1+\eps)}$ and $\ee^{\ii(\theta^1-\eps)}$ cannot be swallowed by $\gamma_{[0,t]}$. 
The result follows letting $\eps\to 0$ and $t\to \infty$, as radial $\SLE_{\kappa}^{\mu}$ 
is transient by Proposition~\ref{prop::radialSLE_transience}. 
\end{proof}

\begin{lemma} \label{lem::radialSLE_avoidboundary2}
Fix $\kappa\in (0,4]$, $\mu\in\R$, $\rho\ge 0$ and $\bs{\theta}=(\theta_1, \theta_2)\in\LX_2$.
Radial $\SLE_{\kappa}^{\mu}(\rho)$ in $(\U; \ee^{\ii\bs{\theta}}; 0)$ intersects $\partial\U$ only at $\ee^{\ii\theta^1}$.  
\end{lemma}

\begin{proof}
The process $\covmap_t(\theta^2)-\xi_t$ never hits $0$ or $2\pi$ at any finite time $t$ by Lemma~\ref{lem::Bessel_avoid_0and2pi}: 
from~\eqref{eqn::SLEkapparhomu_SDE}, we have
\begin{align} \label{eqn::radialSLE_avoidboundary2_aux1}
\ud (\covmap_t(\theta^2)-\xi_t)=-\sqrt{\kappa} \ud B_t + \big(1+\tfrac{\rho_2}{2} \big) \cot \bigg( \frac{\covmap_t(\theta^2)-\xi_t}{2} \bigg) \ud t - \mu \, \ud t,
\end{align}
Therefore, the point $\ee^{\ii\theta^2}$ cannot be swallowed by $\gamma$. 
Thus, by Lemma~\ref{lem::sleakapprhomu_mart}, we see that the law of radial $\SLE_{\kappa}^{\mu}(\rho_2)$ in $(\U;\ee^{\ii\bs\theta};0)$ is absolutely continuous with respect to the law of 
radial $\SLE_{\kappa}$ in $(\U;\ee^{\ii\theta^1};0)$ up any finite time $t$. 
So by Lemma~\ref{lem::radialSLE_avoidboundary1}, up to any finite time $t$, radial $\SLE_{\kappa}^{\mu}(\rho_2)$ in $(\U;\ee^{\ii\bs\theta};0)$ touches 
$\partial\U$ only at $e^{\ii\theta^1}$. 
The result follows letting $t\to \infty$, as radial $\SLE_{\kappa}^{\mu}(\rho_2)$ is transient by Proposition~\ref{prop::radialSLE_transience}. 
\end{proof}

\begin{proof}[Proof of Proposition~\ref{prop::radialSLE_avoidboundary}]
Assume $\theta^1<\theta<\theta^2$. From~\eqref{eqn::SLEkapparhomu_SDE}, 
we have
\begin{align} \label{eqn::slespiral_avoidboundary_aux1}
\ud (\covmap_t(\theta)-\xi_t)=-\sqrt{\kappa} \ud B_t+ \cot \bigg( \frac{\covmap_t(\theta)-\xi_t}{2} \bigg) \ud t + \sum_{j=2}^p \frac{\rho_j}{2} \cot \bigg( \frac{\covmap_t(\theta^j)-\xi_t}{2} \bigg) \ud t - \mu \, \ud t.
\end{align}
Note that
\begin{align*}
\sum_{j=2}^p \frac{\rho_j}{2} \cot \bigg( \frac{\covmap_t(\theta^j)-\xi_t}{2} \bigg) \ge  \frac{\bar{\rho}}{2} \cot \bigg( \frac{\covmap_t(\theta^p)-\xi_t}{2} \bigg), \quad \textnormal{where } \bar{\rho}=\sum_{j=2}^p \rho_j.
\end{align*}
Let $\tilde{\Diff}_t$ denote the process $\covmap_t(\theta)-\xi_t$ for $\SLE_{\kappa}(\bar{\rho})$ in $(\U;\ee^{\ii\theta^1},\ee^{\ii\theta^p};0)$, 
and $\Diff_t$ the process $\covmap_t(\theta)-\xi_t$ for $\SLE_{\kappa}(\bs{\rho})$ in $(\U;\ee^{\ii\bs{\theta}};0)$.
Applying a similar argument as in the proof of Lemma~\ref{lem::coupleXDeltaP}, we can construct a coupling between $\Diff_t$ and $\tilde{\Diff}_t$ such that $\tilde{\Diff}_t\le \Diff_t$.  
By Lemma~\ref{lem::radialSLE_avoidboundary2}, we can conclude that $\tilde{\Diff}_t$ never hits $0$ at any finite time $t$. 
The coupling implies that $\Diff_t$ never hits $0$ at any finite time $t$ --- therefore, for any $0<\eps<\min\{\theta^2-\theta^1,2\pi-\theta^p+\theta^1\}$ and for any finite time $t$,
the point $\ee^{\ii (\theta^1+\eps)}$ cannot be swallowed by $\gamma_{[0,t]}$ from the left hand side. 
By symmetry, $\ee^{\ii (\theta^1-\eps)}$ cannot be swallowed by $\gamma_t$ from the right hand side. 
Again, the result follows letting $\eps\to 0$ and $t\to \infty$, as radial $\SLE_{\kappa}^{\mu}(\bs{\rho})$ 
is transient by Proposition~\ref{prop::radialSLE_transience}. 
\end{proof}


{\small
\bigskip{}
\bibliographystyle{annotate}
\newcommand{\etalchar}[1]{$^{#1}$}

}
\end{document}